\PassOptionsToPackage{unicode}{hyperref}
\PassOptionsToPackage{hyphens}{url}
\PassOptionsToPackage{dvipsnames,svgnames,x11names}{xcolor}
\PassOptionsToPackage{format=plain,
  labelfont={bf,sf,small,singlespacing},
  textfont={sf,small,singlespacing},
  justification=justified,
  margin=0.25in}{caption}%
\documentclass[
  11pt,
]{article}
\usepackage{xcolor}
\usepackage[lmargin=1in,rmargin=1in,tmargin=1in,bmargin=1in]{geometry}
\usepackage{amsmath,amssymb}
\usepackage{iftex}
\ifPDFTeX
  \usepackage[T1]{fontenc}
  \usepackage[utf8]{inputenc}
  \usepackage{textcomp} 
    \usepackage[p,osf,swashQ,scale=1.10]{cochineal}
  \usepackage{amsthm}
  \usepackage[cochineal,vvarbb]{newtxmath}
      \usepackage{helvet}
    \usepackage[scale=0.95,varl]{inconsolata}
\else 
    \usepackage{unicode-math} 
    \usepackage{mathpazo}
  \usepackage{newpxtext}
    \usepackage[scaled]{helvet}
    \defaultfontfeatures{Scale=MatchLowercase}
  \defaultfontfeatures[\rmfamily]{Ligatures=TeX,Scale=1}
\fi

\usepackage{sectsty}
\ifPDFTeX\else
\fi
\IfFileExists{upquote.sty}{\usepackage{upquote}}{}
\IfFileExists{microtype.sty}{
  \usepackage[]{microtype}
  \UseMicrotypeSet[protrusion]{basicmath} 
}{}
\makeatletter
\@ifundefined{KOMAClassName}{
  \IfFileExists{parskip.sty}{%
    \usepackage{parskip}
  }{
    \setlength{\parindent}{0pt}
    \setlength{\parskip}{6pt plus 2pt minus 1pt}}
}{
  \KOMAoptions{parskip=half}}
\makeatother

\usepackage{longtable,booktabs,array}
\usepackage{calc} 
\usepackage{etoolbox}
\makeatletter
\patchcmd\longtable{\par}{\if@noskipsec\mbox{}\fi\par}{}{}
\makeatother
\IfFileExists{footnotehyper.sty}{\usepackage{footnotehyper}}{\usepackage{footnote}}
\makesavenoteenv{longtable}
\usepackage{graphicx}
\makeatletter
\newsavebox\pandoc@box
\newcommand*\pandocbounded[1]{
  \sbox\pandoc@box{#1}%
  \Gscale@div\@tempa{\textheight}{\dimexpr\ht\pandoc@box+\dp\pandoc@box\relax}%
  \Gscale@div\@tempb{\linewidth}{\wd\pandoc@box}%
  \ifdim\@tempb\p@<\@tempa\p@\let\@tempa\@tempb\fi
  \ifdim\@tempa\p@<\p@\scalebox{\@tempa}{\usebox\pandoc@box}%
  \else\usebox{\pandoc@box}%
  \fi%
}
\def\fps@figure{htbp}
\makeatother

\providecommand{\tightlist}{%
  \setlength{\itemsep}{0pt}\setlength{\parskip}{0pt}}

\usepackage[]{natbib}
\usepackage{mathtools}
\usepackage{amsmath}
\usepackage{amsthm}
\usepackage{amssymb}
\usepackage[italicdiff]{physics}

\renewcommand{\bf}{g} 
\newcommand{\tf}{h} 
\renewcommand{\b}{{\mathrm{BART}}}
\newcommand{\vmin}{\wedge}
\newcommand{\vmax}{\vee}

\newcommand{\cC}{\mathcal{C}}
\newcommand{\cS}{\mathcal{S}}
\newcommand{\cN}{\mathcal{N}}
\newcommand{\cF}{\mathcal{F}}

\newcommand{\X}{\mathcal{X}}
\newcommand{\A}{\mathcal{A}}
\newcommand{\V}{\mathcal{V}}

\newcommand{\bx}{\mathbf{x}}
\newcommand{\bu}{\mathbf{u}}
\newcommand{\by}{\mathbf{y}}

\newcommand{\adj}{\sim_\tf}

\newcommand{\D}{\mathcal{D}}
\renewcommand{\H}{\mathcal{H}}
\newcommand{\GP}{\mathrm{GP}}

\newcommand{\R}{\mathbb{R}}
\newcommand{\N}{\mathbb{N}}

\newcommand{\iid}{\stackrel{\text{iid}}{\sim}}
\DeclareMathOperator{\Norm}{\mathcal{N}}

\DeclareMathOperator*{\argmin}{arg\,min}
\newcommand{\Unif}{\mathrm{Unif}}
\newcommand{\Pois}{\mathrm{Pois}}

\newcommand{\E}{\mathbb{E}}
\renewcommand{\Pr}{\mathbb{P}}

\newcommand{\Cov}{\mathrm{Cov}}

\newcommand{\vbg}[1]{\vb*{#1}}
\providecommand{\norm}[1]{\lVert #1\rVert}

\newcommand{\ind}{\mathbf{1}}

\newcommand{\cvd}{\Rightarrow}

\AtBeginDocument{%
  \renewcommand{\bf}{g}%
  \renewcommand{\b}{{\mathrm{BART}}}%
  \renewcommand{\H}{\mathcal{H}}%
  \renewcommand{\L}{\mathcal{L}}
}
\makeatletter
\@ifpackageloaded{caption}{}{\usepackage{caption}}
\AtBeginDocument{%
\ifdefined\contentsname
  \renewcommand*\contentsname{Table of contents}
\else
  \newcommand\contentsname{Table of contents}
\fi
\ifdefined\listfigurename
  \renewcommand*\listfigurename{List of Figures}
\else
  \newcommand\listfigurename{List of Figures}
\fi
\ifdefined\listtablename
  \renewcommand*\listtablename{List of Tables}
\else
  \newcommand\listtablename{List of Tables}
\fi
\ifdefined\figurename
  \renewcommand*\figurename{Figure}
\else
  \newcommand\figurename{Figure}
\fi
\ifdefined\tablename
  \renewcommand*\tablename{Table}
\else
  \newcommand\tablename{Table}
\fi
}
\@ifpackageloaded{float}{}{\usepackage{float}}
\floatstyle{ruled}
\@ifundefined{c@chapter}{\newfloat{codelisting}{h}{lop}}{\newfloat{codelisting}{h}{lop}[chapter]}
\floatname{codelisting}{Listing}

\usepackage{amsthm}
\theoremstyle{plain}
\newtheorem{theorem}{Theorem}[section]
\theoremstyle{plain}
\newtheorem{proposition}{Proposition}[section]
\theoremstyle{plain}
\newtheorem{lemma}{Lemma}[section]
\theoremstyle{plain}
\newtheorem{corollary}{Corollary}[section]
\theoremstyle{remark}
\AtBeginDocument{}

\makeatother
\makeatletter
\@ifpackageloaded{caption}{}{\usepackage{caption}}
\@ifpackageloaded{subcaption}{}{\usepackage{subcaption}}
\makeatother
\usepackage{bookmark}
\IfFileExists{xurl.sty}{\usepackage{xurl}}{} 
\hypersetup{
  pdftitle={The Gaussian Process Limit of BART},
  pdfkeywords={BART, Gaussian process, random features},
  colorlinks=true,
  linkcolor={black},
  filecolor={Maroon},
  citecolor={VioletRed4},
  urlcolor={DodgerBlue4},
  pdfcreator={LaTeX via pandoc}}

\makeatletter
\newcommand\@shorttitle{}
\newcommand\shorttitle[1]{\renewcommand\@shorttitle{#1}}
\usepackage{fancyhdr}
\fancyheadoffset{0pt}
\makeatother
\renewenvironment{abstract}{
  \centerline
  {\large\sffamily\bfseries Abstract}\vspace{-0.25em}
  \begin{quote}\small
}{
  \end{quote}
}

\title{\sffamily\bfseries\huge\parfillskip=0pt
\rightskip=0pt plus .5\textwidth
\leftskip=0pt plus .5\textwidth
\emergencystretch=.3\textwidth Seeing the Forest for the Trees: The
Gaussian Process Limit of BART}
\shorttitle{The Gaussian Process Limit of BART}
\author{\textbf{Cory McCartan}\footnote{
To whom correspondence should be addressed.
Email: \texttt{\href{mailto:mccartan@psu.edu}{mccartan@psu.edu}}.
Website: \url{https://corymccartan.com/}.
Address:
326 Thomas Building, University Park, PA 16802.
The authors thank Bharath K. Sriperumbudur for helpful comments.}
\\Department of Statistics%
\\Pennsylvania State University%
\vspace{2pt}
 \and \textbf{Melody Huang}
\\Department of Political Science%
\vspace{2pt}
\\Department of Statistics \& Data Science%
\\Yale University%
\vspace{2pt}
 }
\date{August 31, 2026}

\begin{document}
\allsectionsfont{\sffamily}

\maketitle

\begin{abstract}
Bayesian Additive Regression Trees (BART) have shown state-of-the-art
performance in both prediction and causal inference problems. Previous
theoretical work has attempted to explain BART's superior performance by
establishing posterior contraction rates for standard BART models, but
these rates depend strongly on the number of covariates. Here, we take a
different approach and study the behavior of BART as the number of trees
grows towards infinity. We show that in this regime, BART converges to a
Gaussian process (GP) with a particular kernel. The kernel and its
corresponding reproducing kernel Hilbert space (RKHS) have favorable
inferential properties that help explain BART's excellent performance.
We introduce \emph{random tree features} as an approximation to this
limiting GP, and establish minimax-optimal learning rates for ridge
regression on these random features that depend only logarithmically on
dimension. In addition to providing insight into the empirical success
of BART, random tree features offer a computational benefit over
traditional MCMC estimation. The random-features approximation also
allows practitioners to easily incorporate BART into any model which has
a linear predictor, expanding the applicability and flexibility of BART.
\end{abstract}

\textbf{\textit{Keywords}}\quad BART~\textbullet~Gaussian
process~\textbullet~random features

\newpage

\section{Introduction}\label{sec-intro}

Bayesian Additive Regression Trees
\citep[BART,][]{chipman2007bayesian, chipman2010bart} is a nonparametric
Bayesian regression model that is often considered the Bayesian analogue
of boosted trees. BART models the regression function as a sum of many
decision trees, which are constrained to be small by the prior on each
tree. BART has demonstrated strong empirical performance across both
prediction and causal inference tasks, often outperforming alternative
approaches such as random forests and gradient boosting, and is widely
applied across the social and biomedical sciences
\citep{linero2017review, hill2020bayesian, dorie2019automated}.

BART's practical success has spurred theoretical investigations into its
asymptotic properties. While posterior contraction rates have been
established for BART and its variants under a variety of function
classes \citep[see, e.g.,][]{rovckova2019theory}, they suggest but do
not isolate a specific aspect of the BART model that is responsible for
its superior performance. Previous studies have hypothesized several
explanations for BART's success, including additivity, the Bayesian
inferential framework, the model's ability to adapt the tree structure,
and the flexible nature of the trees themselves
\citep{chipman2010bart, rovckova2020posterior, linero2018bayesian}.

We show that besides additivity, these other explanations are not
empirically supported when the number of trees are large. Instead, we
argue that BART is most usefully viewed as a random-feature
approximation to a limiting Gaussian process. This perspective provides
theoretical, explanatory, computational, and practical benefits for
working with BART.

\subsection{Contributions}\label{contributions}

We first investigate the load-bearing components of the BART model
through a simple ablation study, presented briefly in
Section~\ref{sec-ablate} and in more detail in
Section~\ref{sec-sim-ablate}. By removing aspects of the model, we show
that Bayesian learning of the tree structure, and even aspects of the
tree structure itself, do not appear to drive the bulk of BART's
performance, so long as the number of trees is sufficiently large.

This motivates a theoretical investigation into the large-tree limit of
BART. In Section~\ref{sec-gp}, we establish that a BART prior on
symmetric trees (defined below) converges weakly to a Gaussian process
(GP) in the infinite-tree limit, and derive a series expression for the
covariance function and a closed-form special case. While the GP limit
has been conjectured \citep{linero2017review} and the finite-dimensional
covariance derived in the asymmetric case \citep{petrillo2024gaussian},
we are, to our knowledge, the first to formally state and prove the weak
convergence result.

We investigate the reproducing kernel Hilbert space (RKHS) corresponding
to the limiting GP and show that it is equivalent to \(S_1([0,1]^p)\),
the Sobolev space with dominating mixed derivatives \citep[and defined
below]{nikol1962boundary, lin2000tensor}. As we discuss in
Section~\ref{sec-rkhs}, this space can be interpreted as the set of
functions that are additively decomposable as \(p\)-way interactions of
univariate functions. This result helps explain the performance of BART,
because the minimax rate for regression in \(S_1([0,1]^p)\) depends on
the covariate dimension \(p\) only logarithmically.

Of course, these results only apply with infinitely many trees. In
Section~\ref{sec-rf}, we study the behavior of a random-feature
approximation to the limiting GP. Practically, this means a ridge
regression on \emph{random tree features} drawn from the prior, without
any adaptation. We show that, under appropriate regularity conditions,
as few as \(T_n\asymp n^{1/3}\log(n)^{1 + 2(p-1)/3}\) trees can be
sufficient to achieve an error rate of \(n^{-1/3}\log(n)^{(p-1)/3}\),
which is minimax-optimal for regression in \(S_1([0,1]^p)\)
\citep{lin2000tensor}. The error rate is particularly notable as it is
faster than \(n^{-1/4}\) for \emph{any} \(p\), and relies only on weak
smoothness assumptions on the underlying function. This allows
regression on random tree features to be used within modern debiased
machine learning estimators, which require a rate faster than
\(n^{-1/4}\) for nuisance function estimation
\citep{chernozhukov2018double}.

The result has immediate computational implications as well: rather than
carry out relatively expensive MCMC sampling for a BART model, a single
ridge regression on sufficiently many completely random tree features
can provide optimal inferential rates. Even when computation is not a
bottleneck, the random basis expansion perspective allows practitioners
to achieve BART-like performance in a much wider variety of models,
simply by including the random tree features as part of any linear
predictor. This is a marked improvement over current practice, which
requires experts to write custom MCMC samplers that may suffer from poor
mixing, especially in non-Gaussian settings where data augmentation can
be required
\citep{sparapani2016nonparametric, murray2021log, deshpande2026vcbart}.

We conclude in Section~\ref{sec-sim} with a real-data evaluation of the
predictive performance of random tree features compared to full BART,
boosted trees, and random forests. We also compare the uncertainty
quantification of random tree features to full BART as well as the
limiting GP, and find strong similarities between random features and
full BART.

\subsection{Related Work}\label{related-work}

This paper connects to three distinct branches of existing work: (1)
prior theoretical work on BART's performance, (2) nonparametric methods
that target functions in spaces like \(S_1([0, 1]^p)\), and (3) random
feature approximations to Gaussian processes and kernel methods.

Existing BART literature has established different contraction rates for
BART for a variety of different function classes. For example,
\citet{van2017bayesian} proves near-minimax rates for step mean
functions in Gaussian regression; \citet{rovckova2019theory} and
\citet{rovckova2020posterior} provide rates for Hölder continuous
functions. These results all suffer from the curse of dimensionality.
More specifically, the rates generally all boil down to the typical
nonparametric rate \(n^{-\alpha/(2\alpha + p)}\) for \(p\) covariates
and a smoothness parameter \(\alpha\). Follow-up work has tackled this
challenge through new proof techniques as well as modifications to the
BART prior, leading to results in more esoteric function spaces that are
anisotropic or sparse
\citep{jeong2023art, yee2024scalable, rovckova2020posterior, linero2018var, linero2018bayesian}.
By and large, these results alter the basic rate by letting \(p\) refer
to the \emph{effective} number of covariates, either across the entire
function or after decomposing into a small number of additive
components. While these results shed light on BART's adaptability, the
curse of dimensionality remains: even a moderate number of effective
covariates still leads to a slow rate. With \(\alpha<1\) (as is often
required), an effective \(p\ge 3\) leads to a rate slower than
\(n^{-1/4}\).

All of these papers focus on the \(n\to\infty\) limit, with \(T\) fixed
(or drawn from a fixed prior). Instead, we focus on the \(T \to \infty\)
limit, which leads to a different theoretical approach. Growing \(T\)
with \(n\) preserves the ``many small trees'' intuition that motivated
BART \citep{chipman2010bart}. The resulting learning rates, which unlike
existing rates do not depend as strongly on the covariate dimension,
highlight the importance of this intuition.

A different recent strand of literature, which initially appears to be
far distant from BART, develops regression methods with a large basis
expansion consisting of interactions of step functions
\citep[e.g.,][]{fang2021multivariate, benkeser2016highly, vanderlaan2023higher, schuler2024highly, schuler2022ltb}.
The proposed approaches obtain learning rates similar to the ones we
find here, for similar or identical function classes. This includes the
highly adaptive lasso
\citep[HAL,][]{benkeser2016highly, vanderlaan2023higher}, highly
adaptive ridge \citep[HAR,][]{schuler2024highly}, and lassoed tree
boosting \citep[LTB,][]{schuler2022ltb}. The basis features in these
approaches are of an identical form to the ones studied here under the
symmetric-tree BART model we define. The key difference is that HAL and
HAR use all possible interactions of \emph{all} features. Instead, we
use a random set of such features, with priority to low-dimensional
interactions. One interpretation of our main random-feature result here
is as showing that a randomized version of HAR can perform just as well
asymptotically, while avoiding HAR's \(O(n^3)\) computational cost.

The developers of HAL, HAR, and LTB operate within slightly different
function classes, those defined by properties of their sectional
derivatives, but as they explain, their classes are similar or identical
to \(S_1([0, 1]^p)\) under certain assumptions. Other authors, most
notably \citet{zhang2023regression}, have studied regression in
\(S_1([0, 1]^p)\) directly. \citet{zhang2023regression} develop a
deterministic sieve basis expansion based on reordered tensor products
of continuous basis functions of univariate Sobolev spaces. While not
identical, the pattern of features in their sieve is similar to the
random tree features developed here, as they primarily consists of
interactions of simple nonlinear functions of a few covariates. The
sieve perspective naturally leads to questions about a more
deterministic way of selecting random tree features from the prior,
which may pay dividends theoretically and practically.

Finally, our use of random tree features builds directly on prior work
on random feature approximations to Gaussian processes, a framework
first introduced by \citet{rahimi2007random}. Most of the theoretical
work has focused on random Fourier features, which are slightly simpler
to analyze than the multivariate and discontinuous random features
studied here. We build particularly on the results of
\citet{rudi2017generalization}, generalizing some of their key results
to handle these complications. More applied literature has directly
proposed random features not dissimilar from the random trees here,
including `random binning features' proposed by \citet{rahimi2007random}
and the Mondrian kernel studied by \citet{balog2016mondrian}. Very
recently, \citet{linero2026bayesian} proposed random tree features as a
way of approximating a BART-based kernel mean embedding. However, none
of these papers formally established weak convergence to a GP, nor
learning rates comparable to those developed here.

\section{Setup}\label{sec-setup}

We focus on the setting of estimating a regression function \(g(\vb x)\)
on predictors \(\vb x\) lying in the unit hypercube \([0, 1]^p\). This
section introduces the BART model for \(g\) and establishes our
notation. Once the BART model is defined, we empirically investigate the
performance of its components through an ablation study, which motivates
our theoretical investigations in the rest of the paper.

Let \(\vb 1\{\cdot\}\) be an indicator function for the event
\(\{\cdot\}\). By \([p]\) for \(p\in\N\) we mean the set
\(\{1, \dots, p\}\). We denote weak convergence by \(\cvd\). The
notation \(a \lesssim b\) means that \(a\le C\cdot b\) for a universal
constant \(C\); \(a\asymp b\) means \(a\lesssim b\) and \(b\lesssim a\),
or equivalently \(C\le a/b \le C'\) for universal constants \(C, C'\).

\subsection{Decision trees}\label{decision-trees}

A \emph{decision tree} is a random function of \(\bx\in[0, 1]^p\)
defined by a partition of \([0,1]^p\) into \(L\) hyperrectangular
regions or \emph{blocks} of the form
\([a_1, b_1)\times\cdots\times[a_p, b_p)\), where each
\([a_j, b_j)\subseteq[0, 1]\), and the right endpoint is closed if
\(b_j=1\).

We index blocks by binary words \(\vb l\in \L \subseteq \{-1, 1\}^D\),
where \(D\) is the \emph{depth} of the tree: the number of binary
decisions that define membership in the block indexed by \(\vb l\). In
general we have \(D+1\le |\L|\le 2^D\). We denote by
\(\psi_{\vb l}(\bx)\) the indicator variable that \(\bx\) is in region
indexed by \(\vb l\).

Each block is associated with a \emph{leaf parameter}
\(\mu_{\vb l}\in\R\). The tree function itself is constant within each
block, and so can be represented as a linear combination:
\begin{equation}\protect\phantomsection\label{eq-bart}{
\tf(\bx) = \sum_{\vb l\in\L} \mu_{\vb l} \psi_{\vb l}(\bx).
}\end{equation}

A \emph{symmetric decision tree} is a special case where the decision
rule and cutpoint are the same for all nodes at the same level of the
tree. Formally, \(\L=\{-1,1\}^D\) and the blocks take a certain
structure, defined by a tuple of \(D\) decision rules, each consisting
of a \emph{variable} \(V\in[p]^D\) and a cutpoint or \emph{split value}
\(S\in[0, 1]\). Together, the rules map the input \(\bx\) to one of
\(2^D\) regions, which we can represent as \[
    \psi_{\vb l}(\bx; D, \vb V, \vb S) := \prod_{k=1}^D \ind\{l_kx_{V_k} <_{l_k} l_kS_k\}.
\] By \(<_{l_k}\) we mean \(<\) if \(l_k=1\) and \(\le\) if \(l_k=-1\),
so that each \(\bx\) is mapped to exactly one leaf. Since multiplication
is commutative, the order in which the decision rules are applied in a
symmetric tree is arbitrary; each symmetric tree has \(D!\) equivalent
representations yielding the same partition of \([0, 1]^p\).

\subsection{BART}\label{bart}

A BART function with \(T\) trees is defined by \[
\bf_T(\bx) := T^{-1/2} \sum_{j=1}^T \tf_j(\bx)
= T^{-1/2} \sum_{j=1}^T \sum_{\vb l\in\L} \mu_{j\vb l} \psi_{j\vb l}(\bx)
= T^{-1/2} \sum_{j=1}^T \Psi_j(\bx)^\top \vbg\mu_j,
\] where \(\Psi_j = (\psi_{j\vb l})_{\vb l\in\L}\) denotes the vector of
indicator functions. We draw the parameters of each \(\tf_j\) i.i.d.
from the following generic prior: \[
\Psi_j \sim f_\Psi \qand
\mu_{j\vb l}\mid D_j \iid \Norm(0, \sigma_\mu^2).
\] In general BART models, the tree structure prior \(f_\Psi\) is
defined by a branching process. For symmetric trees, we can factor this
prior into pieces specific to \(D\), \(\vb V\), and \(\vb S\). \[
D_j \sim f_D, \qquad
V_{jk}\mid D_j \iid f_V, \qand
S_{jk}\mid D_j \iid f_S.
\] Here \(f_D\), \(f_V\), and \(f_S\) are generic density functions,
with \(f_D\) supported on the positive integers, \(f_V\) supported on
\([p]\) with \(f_V(v)>0\) everywhere, and \(f_S\) continuous and
supported on \([0, 1]\) with strictly positive density everywhere.
Often, a BART function is used directly as the model for a conditional
expectation function (CEF) \(\E[Y\mid \bx, \bf_T]=\bf_T(\bx)\), as we do
here, but BART functions can be used within other parts of a model as
well.

\subsection{Motivating ablation study}\label{sec-ablate}

In fitting a BART model, both the tree structure encoded by
\((D_j, \vb V_j, \vb S_j)_{j=1}^T\) and the leaf parameters
\((\vbg\mu_j)_{j=1}^T\) are learned from the data. It is natural to ask
which of these parameters are most important for explaining BART's
empirical success, or whether both are equally important. Two other
parts of the BART model are found universally in BART applications, but
each impose a nontrivial computational cost: the use of Bayesian
inference for the parameters, and the use of fully general trees versus
symmetric trees. It is likewise natural to wonder whether these parts of
the model are strictly necessary for BART's performance.

\begin{figure}[ht]

\centering{

\pandocbounded{\includegraphics[keepaspectratio]{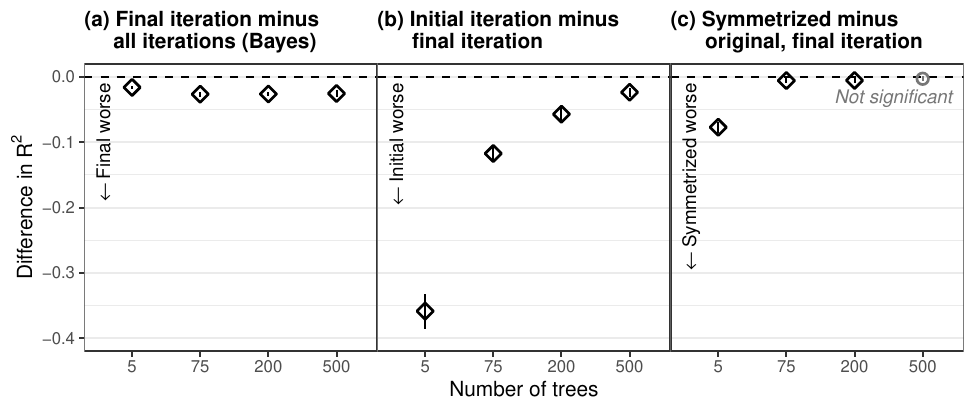}}

}

\caption{\label{fig-ablate-intro}\textbf{Three ablation studies on
\texttt{abalone} data} (\(n=4\,177, p=8\)). Each panel compares
out-of-sample \(R^2\) values of the posterior mean, for two different
regression fits and different numbers of regression trees. The average
difference in \(R^2\) across 20 cross validation splits and multiple
replicates is plotted along with 95\% confidence intervals. \textbf{(a)
Ablating Bayes:} \(R^2\) for the full posterior mean versus the
posterior mean conditioned on the tree structure at the final MCMC
iteration. \textbf{(b) Ablating tree learning:} \(R^2\) conditioned on
the final-iteration trees versus the initial-iteration trees.
\textbf{(c) Ablating asymmetric trees:} \(R^2\) conditioned on the
final-iteration trees versus a symmetrized version of those same trees.}

\end{figure}%

We investigate these three aspects of the BART model through a simple
ablation study. In this section, we provide the results for a single
dataset, with comprehensive results for a battery of datasets in
Section~\ref{sec-sim-ablate}. We consider the \texttt{abalone} dataset,
which contains the age of 4,177 abalone and 8 predictors such as sex,
diameter, and shell weight \citep{abalone}. For each part of the model,
we compare the predictive performance of the full BART model, as
measured with out-of-sample \(R^2\), to a version of the model with
different parts removed, in the ways we describe below.
Figure~\ref{fig-ablate-intro} plots the difference in \(R^2\) for each
of these experiments as the number of trees \(T\) varies.

First, to measure the impact of Bayesian inference, we compute the
posterior conditional on the tree structure at the final MCMC iteration,
and compare its performance to the full posterior mean. As panel (a)
shows, using a single sampled tree structure incurs only a small loss in
predictive performance, which appears relatively constant across the
number of trees.

Second, to measure the impact of learning the tree structure, we compare
the posterior conditional on the tree structure at the final MCMC
iteration to that from the initial MCMC iteration. To the extent that
the posterior exhibits adaptation to the tree structure, we would expect
a large gap in performance between these two samples. While this gap is
evident for small \(T\) in panel (b), it narrows as \(T\) increases, and
is relatively small at \(T=500\) for these data.

Finally, to measure the impact of asymmetric trees, we build a
symmetrized version of the final-iteration tree structure and compare
posteriors conditional on the asymmetric and symmetric trees. To
symmetrize, we replace all nodes at each level of a tree with a single
node chosen uniformly from all the nodes at that level. We randomly
terminate this process with a certain probability at each depth, set so
that the expected number of nodes in the symmetrized tree matches the
number of nodes in the original tree. This ensures that the symmetrized
tree is no more expressive than the original tree, on average. As panel
(c) shows, the impact of using symmetrized trees is nearly zero as long
as \(T\) is not very small. If anything, this figure overstates the
performance of asymmetric trees, since the original MCMC did not sample
with symmetric trees in mind.

Taken together, these experiments suggest that once \(T\) is large, none
of Bayesian averaging, tree structure learning, or asymmetric trees is
critical to BART's performance. The additive combination of many small
interacted step functions appears to be the key ingredient, at least
asymptotically. This motivates our investigation into the \(T\to\infty\)
limit of BART, and allows us to focus on symmetric trees, which are
computationally more efficient and theoretically simpler to study.

\section{BART converges to a Gaussian process}\label{sec-gp}

In this section, we investigate the distribution of \(\bf_T\) as a
random function, and its limiting behavior as \(T\to\infty\). We
establish that \(\bf_T \cvd \GP(0, k_\b)\) for a certain kernel \(k_\b\)
and on a suitably defined function space. First, we derive the
covariance function of \(\bf_T\), which in turn yields convergence of
the finite-dimensional distributions of \(\bf_T\). We then turn to the
weak convergence of \(\bf_T\) itself.

\subsection{Covariance function}\label{covariance-function}

From the definition of the BART model, for any fixed \(\bx\), \[
\bf_T(\bx)\mid \{\Psi_j\}_{j=1}^T \sim \Norm(0, \sigma^2_\mu),
\] since exactly one \(\mu_{j\vb l}\) contributes for each \(j\), and
each is independently drawn with variance \(\sigma_\mu^2\). With
finitely many trees, \(\bf_T(\bx)\) and \(\bf_T(\bx')\) are not jointly
Gaussian for \(\bx\neq\bx'\); however, we can derive their covariance.
We do so in two steps: first, showing that the covariance function of a
single tree, i.e., between \(\tf(\bx)\) and \(\tf(\bx')\), is
proportional to the probability of \(\bx\) and \(\bx'\) being in the
same leaf node, and second, computing that probability for symmetric
trees specifically.

\begin{proposition}[]\protect\hypertarget{prp-tree-cov}{}\label{prp-tree-cov}

For a tree \(\tf\) drawn from its prior, and any fixed
\(\bx, \bx'\in[0, 1]^p\), \[
\Cov[\tf(\bx), \tf(\bx')] = k_\b(\bx, \bx')
:= \sigma^2_\mu \Pr[\bx\adj \bx'],
\] where \(\bx\adj \bx'\) is the event that \(\bx\) and \(\bx'\) are in
the same leaf node in \(\tf\).

\end{proposition}

We now turn to characterizing \(\Pr[\bx\adj \bx']\) for symmetric trees.
To do so, define the \emph{symmetric BART covariate metric} as \[
d_\b(\bx, \bx')
:= \sum_{v=1}^p f_V(v) \cdot \left|F_S(x_v) - F_S(x'_v)\right|.
\] \(d_\b\) is clearly symmetric. Furthermore, because \(f_S\) has
strictly positive density, \(F_S\) is strictly increasing and so
\(d_\b\) also satisfies positivity. The triangle inequality follows from
the triangle inequality for absolute values. As a result, \(d_\b\) is
indeed a metric.

When \(f_V\) is uniform on \([p]\), and \(f_S\) is uniform on
\([0, 1]\), then \[
d_\b(\bx, \bx') = p^{-1}\norm{\bx - \bx'}_1.
\] Because of the rescaling, \(d_\b(\,\cdot\,,\,\cdot\,)\le 1\). We can
then express \(\Pr[\bx\adj \bx']\) as a series involving \(d_\b\) and
the distribution of tree depths \(f_D\).

\begin{proposition}[]\protect\hypertarget{prp-tree-cov-exact}{}\label{prp-tree-cov-exact}

For a tree \(\tf\) drawn from the symmetric BART tree prior, and any
fixed \(\bx, \bx'\in[0, 1]^p\), \[
\Pr[\bx\adj \bx'] = \sum_{k=1}^\infty \left(1 - d_\b(\bx, \bx')\right)^k f_D(k).
\]

\end{proposition}

\begin{figure}[t]

\centering{

\pandocbounded{\includegraphics[keepaspectratio]{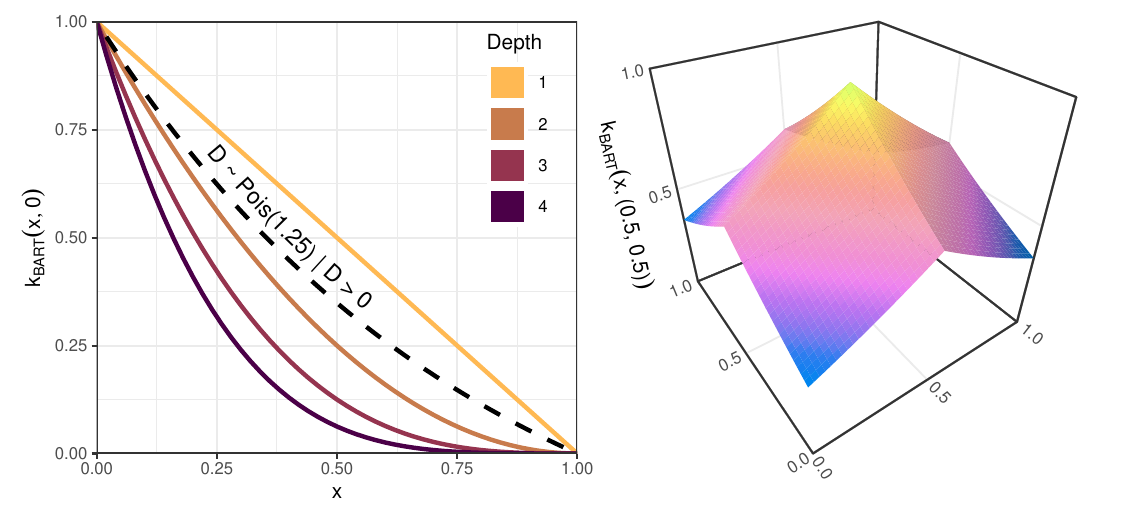}}

}

\caption{\label{fig-kernel}\textbf{Visualization of \(k_\b\)}. The
\textbf{left} panel shows \(k_\b(\cdot, 0)\) for different choices of
the depth prior \(f_D\). The colored lines correspond to \(f_D\) that
place all mass on a single depth, while the black dashed line
corresponds to the Poisson \(f_D\) in Corollary~\ref{cor-cov-pois}, with
\(r=1.25\) chosen to match the mean depth of the traditional BART prior.
The \textbf{right} panel shows \(k_\b(\cdot, (0.5, 0.5))\) on the unit
square, using the same calibrated Poisson \(f_D\).}

\end{figure}%

Certain choices of \(f_D\) yield even simpler closed-form expressions
for \(k_\b\) in symmetric trees. For example, if all trees have depth
\(1\) such that \(f_D(1)=1\) and \(f_D(k)=0\) for all \(k>1\), then
\(k_\b(\bx,\bx')=  1 - d_\b(\bx,\bx')\). This example, as well as the
cases of all trees having depths 2, 3, or 4, are visualized in the left
panel of Figure~\ref{fig-kernel}. A more interesting case is the
following result.

\begin{corollary}[]\protect\hypertarget{cor-cov-pois}{}\label{cor-cov-pois}

Let \(W\sim\Pois(r)\). If \(D\sim(W\mid W>0)\), then \[
k_\b(\bx, \bx') = \frac{\sigma^2_\mu}{1-e^{-r}} \left(\exp(-r\,d_\b(\bx, \bx')) - e^{-r}\right).
\]

\end{corollary}

The rescaling and shifting according to \(-e^{-r}\) differentiates
\(k_\b\) in this case from the form of the \(\ell_1\)-Laplacian kernel
and ensures that when \(d_\b(\bx,\bx')=1\) (i.e., its maximum), the
covariance is zero, and when \(d_\b(\bx,\bx')=0\), the covariance is
\(\sigma^2_\mu\).

This formulation makes clear the role of average depth in determining
the covariance structure. Larger \(r\) produces deeper trees and also
shrinks the length scale of the prior. The right panel of
Figure~\ref{fig-kernel} visualizes \(k_\b\) with this choice of \(f_D\)
on the unit square, with \(r\) chosen by moment-matching the mean tree
depth to the traditional BART prior of \citet{chipman2007bayesian}. The
visible ridge along \(x_1=0.5\) and \(x_2=0.5\) clearly demonstrates the
anisotropic nature of \(k_\b\), a key difference from standard isotropic
kernels like the squared-exponential or Matérn.

\subsection{Convergence result}\label{convergence-result}

Having derived the form of the covariance function for tree functions,
we can apply the central limit theorem to derive the
(finite-dimensional) distributional limit of a BART function.

\begin{lemma}[]\protect\hypertarget{lem-clt-fdd}{}\label{lem-clt-fdd}

For any fixed \(\bx_1,\dots,\bx_k \in[0, 1]^p\), as \(T\to\infty\), \[
(\bf_T(\bx_1), \dots, \bf_T(\bx_k)) \cvd
\Norm(0, K),
\] where \(K_{ij} := k_\b(\bx_i, \bx_j)\).

\end{lemma}

Thus any finite collection of points behaves like a Gaussian process
under \(\bf_T\), as \(T\to\infty\). Showing that \(\bf_T\) itself
converges weakly to a Gaussian process requires additional work, and
also depends on a chosen function space on which the limit will be
defined.

Following \citet{neuhaus1971weak} and \citet{bickel1971convergence}, we
generalize the notion of càdlàg functions to real-valued functions on
\([0, 1]^p\). To do so, notice that every point \(\vb z\in[0,1]^p\)
defines a partition \(\mathcal{B}(\vb z)\) of that space into \(2^p\)
blocks. For example, when \(p=3\),
\([0, z_1)\times[z_2, 1]\times[0, z_3)\) is one such block. We define
the function space \(\D_p\) as the set of functions \(f:[0, 1]^p\to\R\)
such that for every \(\vb x\in[0, 1]^p\), every nonempty block
\(B\in\mathcal{B}(\vb x)\), and every sequence \(\vb x_n\in B\) with
\(\vb x_n\to\vb x\),

\begin{enumerate}
\def\labelenumi{(\arabic{enumi})}
\tightlist
\item
  the sequence \(f(\vb x_n)\) converges; and
\item
  if \(\vb x\in B\), then \(f(\vb x_n)\to f(\vb x)\).
\end{enumerate}

The point \(\vb x\) is only in one of the blocks, so these conditions
encode the idea of ``continuity from above, with limits from below.'' It
is possible to define a metric topology on \(\D_p\) so the space is
Polish, and such that when \(p=1\) we recover the Skorohod function
space of càdlàg functions on \([0, 1]\) \citep{bickel1971convergence}.
As a result, Prokhorov's theorem applies, which means that the tightness
of the sequence \((\bf_T)_{T=1}^\infty\) is sufficient to establish weak
convergence, along with the convergence of finite-dimensional
distributions, which was shown in Lemma~\ref{lem-clt-fdd}. See
\citet{neuhaus1971weak} for more discussion, and
\citet{kern2024skorokhod} for a review of topologies in Skorohod spaces.

To show tightness of a random function sequence, the general approach is
to rely on the Arzelà--Ascoli theorem, which establishes that uniform
boundedness and uniform equicontinuity of the sequence are sufficient
for tightness. However, Arzelà--Ascoli applies to sequences of
continuous functions, while the \(\bf_T\) are discontinuous. We instead
adapt an argument developed by \citet{bickel1971convergence} to show
convergence of sequences of random functions on \(\D_p\), i.e.,
\(\D_p\)-valued stochastic processes.

We now state our convergence result along with a short sketch of our
proof. The full proof is contained in the appendix.

\begin{theorem}[]\protect\hypertarget{thm-gp}{}\label{thm-gp}

Let \(\{\bf_T\}_{T=1}^\infty\) be a sequence of symmetric BART functions
on \([0,1]^p\). If \(F_S\) is Lipschitz continuous and
\(\E[D^{2p}]<\infty\), then as \(T\to\infty\),
\(\bf_T \cvd \GP(0, k_\b)\) on \(\D_p\).

\end{theorem}

\begin{proof}[Proof Sketch]
Given Lemma~\ref{lem-clt-fdd}, we must show that the sequence
\((\bf_T)_{T=1}^\infty\) is tight in \(\D_p\). To do so, we first prove
a modified version of a key result from \citet{bickel1971convergence},
which establishes that if a stochastic process belongs to a specific
class, then the sequence is tight in \(\D_p\). We then prove that
\(\bf_T\) belongs to this class of functions, allowing us to establish
weak convergence in the infinite-tree limit.

More specifically, the corollary to \citet{bickel1971convergence}
Theorem 2 provides sufficient conditions for tightness to hold. The main
condition is that a certain modulus of continuity, defined essentially
to ignore a finite number of discontinuities, converge uniformly to
zero. This modulus contains a supremum over all pairs of points in the
unit hypercube. Their theorem 3 replaces this supremum with another that
is only over adjacent triplets of points, as long as the sequence of
stochastic processes belongs to a class \(\cC(\beta, \gamma)\).
Informally, the class \(\cC(\beta, \gamma)\) constrains the fluctuations
in the stochastic process for a given measure across adjacent blocks.
\(\beta\) controls the decay in block size, while \(\gamma\) controls
the deviations in the stochastic process. The advantage of working with
\(\cC(\beta,\gamma)\) is that membership can be established via a moment
bound on elements in the sequence, which is easier to work with.
Unfortunately, a key requirement in \citet{bickel1971convergence} is
that the stochastic processes ``vanish on the lower boundary,'' i.e.,
are fixed at zero whenever any coordinate is zero. In our specific case,
\(\bf_T\) is not fixed at zero whenever a coordinate is zero.
Furthermore, the original moment bound is not satisfied by \(\bf_T\) for
finite \(T\). As a result, we cannot directly apply their results.

We develop an alternative \(T\)-dependent class of functions
\(\cC_T([0,1]^p, \R)\) that is analogous to \(\cC(2,4)\) of Bickel and
Wichura, but does not require vanishing on the lower boundary.
Helpfully, our alternative class \(\cC_T([0,1]^p, \R)\) allows for
relaxations of the moment bound. We then prove a version of their
Theorem 1 for \(\cC_T([0,1]^p, \R)\) using a similar approach via
induction on \(p\) (i.e., Lemma~\ref{lem-condition-bound}), which allows
us to establish an analogous result to their Theorem 3.

It remains to show that in fact \(\bf_T\in\cC_T([0,1]^p, \R)\) for all
\(T\), i.e., that the moment bound is satisfied. We establish the bound
in Lemma~\ref{lem-inc-bc}, relying on a preceding series of technical
lemmas. The bound relies on controlling the probability that the
vertices of a block in \([0,1]^p\) end up in different leaf nodes of a
random tree. If any pair of vertices are in the same leaf node, then the
key function increment whose moment we must control is zero, and so the
probability of all vertices being in different leaf nodes critically
controls our upper bound. It is here that the key features of the BART
model, and the assumptions of the theorem, are used. The negative
association of the split locations conditional on \(D\) helps factor
joint probabilities over \(p\) dimensions into a product over each
dimension, and Lipschitz continuity of \(F_S\) allows us to bound these
probabilities by the side lengths of the block. The finite \(2p\)-th
moment of \(D\) ensures that there are not too many splits, which would
otherwise make the probability of all vertices being in different leaf
nodes too large.
\end{proof}

The conditions needed for convergence are relatively weak and encompass
the BART priors used in practice, in which \(D\) is light-tailed and
\(S\) is uniform.

\subsection{Corresponding kernel and RKHS}\label{sec-rkhs}

Having shown that the BART model converges to a GP, we now investigate
the properties of the kernel \(k_\b\) and its corresponding reproducing
kernel Hilbert space (RKHS) \(\H_\b\), which characterize the properties
of GP samples and the optimal learning rates for function estimation
from such GPs. The formal result is stated in Proposition~\ref{prp-rkhs}
below, but we first introduce the relevant function spaces and explain
their relationships.

When \(p=1\), Proposition~\ref{prp-rkhs} implies that \(k_\b\) is
equivalent to the Laplacian kernel, and the RKHS \(\H_\b\) is
norm-equivalent to the Sobolev space \[
W_1([0, 1]) := \left\{ f \in L^2([0, 1]) : Df\text{ exists and } Df\in L^2([0, 1]) \right\},
\] where \(Df\) is the weak derivative of \(f\) \citep[see,
e.g.,][]{kanagawa2018gaussian}. When \(p>1\), \(k_\b\) is a product
kernel, and Proposition~\ref{prp-rkhs} establishes that the
corresponding RKHS is norm-equivalent to a tensor product of Sobolev
spaces, which can be characterized analogously to \(W_1([0, 1])\) as \[
S_1([0,1]^p) := \bigotimes_{v=1}^p W_1([0,1]) 
= \left\{ f \in L^2([0, 1]^p) : D^{\vb a} f\in L^2([0, 1]^p) \text{ for all } \norm{\vb a}_\infty \le 1 \right\}.
\] where \(D^{\vb a}\) is the weak partial derivative of \(f\) along the
variables indexed by \(\vb a\); for example, if \(p=2\) then
\(D^{(1, 1)}f\) is the (weak) mixed partial derivative of \(f\). This
space is also referred to as a Sobolev space with dominating mixed
derivatives. We can now state the result.

\begin{proposition}[]\protect\hypertarget{prp-rkhs}{}\label{prp-rkhs}

Suppose \(\Pr(D\ge p)>0\) and \(\E[D^p]<\infty\), and \(F_S\) is
Lipschitz continuous. Then there exists constants \(0<c<C<\infty\) such
that \[
cK_\otimes\preceq k_\b\preceq CK_\otimes,
\] where
\(K_\otimes(\bu,\bu'):=\prod_{v=1}^p \exp(-\left|u_v-u_v'\right|)\) is
the \(\ell_1\)-Laplacian kernel on \([0, 1]^p\), and the constants
\(c, C\) depend on the noise variance \(\sigma^2_\mu\), the dimension
\(p\) and the depth distribution \(f_D\). Consequently, the RKHS
\(\H_\b\) corresponding to \(k_\b\) is norm-equivalent to the tensor
product Sobolev space \(S_1([0,1]^p)\).

\end{proposition}

Intuitively, the moment requirement on \(D\) ensures that the trees are
not too deep, which would lead to functions that are too flexible; a
stronger moment requirement is needed for GP convergence. The
requirement that \(\Pr(D\ge p)>0\) ensures that trees exist that can
capture \(p\)-way interactions, which is necessary for the RKHS to
contain all functions in \(S_1([0, 1]^p)\).

We will say that a symmetric BART prior is \emph{regular} if it
satisfies the prior conditions of Proposition~\ref{prp-rkhs}. Usefully,
the Poisson-distributed \(D\) of Corollary~\ref{cor-cov-pois} is regular
for all \(p\), as is the standard BART depth prior introduced by
\citet{chipman2010bart}. Note that regularity is slightly weaker than
the conditions required for Theorem~\ref{thm-gp} above, which rely on
more finite moments of \(D\).

It is also instructive to compare the definition of \(S_1([0, 1]^p)\) to
that of the \(p\)-dimensional first-order Sobolev space \[
W_1([0, 1]^p) := \left\{ f \in L^2([0, 1]^p) : D^{\vb a} f\in L^2([0, 1]^p) \text{ for all } \norm{\vb a}_1 \le 1 \right\}.
\] \(S_1([0, 1]^p)\) is a strictly smaller space than \(W_1([0, 1]^p)\),
since it requires more (weak) mixed partial derivatives to exist.
Consequently, the norm on \(S_1([0, 1]^p)\) imposes a much stronger
constraint than the norm on \(W_1([0, 1]^p)\). However, as
\citet{zhang2023regression} discuss, \(S_1([0, 1]^p)\) is a strictly
larger space than the order-\(p\) Sobolev space \(W_p([0, 1]^p)\), which
requires \(p\) derivatives to exist. Those authors also give the
following equivalent characterization of \(S_1([0, 1]^p)\) as the
closure of the set \[
\left\{ f = \sum_{m=1}^N \prod_{k=1}^p f_{mk} : N<\infty, f_{mk}\in W_1([0, 1]) \right\}
\] with respect to a certain norm. Thus \(\H_\b\) can also be thought of
as the space of functions which are additively decomposable into
\(p\)-way interactions of univariate functions in \(W_1([0, 1])\)---not
dissimilar to symmetric BART functions.\footnote{Of course, the
  component leaf functions in BART do not belong to \(W_1([0, 1])\).}

The more restrictive nature of the space \(S_1([0, 1]^p)\) versus
\(W_1([0, 1]^p)\), and specifically the metric entropy of their unit
balls, translates to improved rates of function estimation, as we
discuss in the next section.

\section{BART as a Random Feature Approximation}\label{sec-rf}

As discussed in the preceding section, the RKHS \(\H_\b\) corresponding
to the BART kernel has attractive properties as a hypothesis class for
function estimation: functions in \(\H_\b\) decompose additively into
\(p\)-way interactions of univariate functions. However, the limiting
Gaussian process requires infinitely many trees. In practice, BART
models use a finite number of trees. This section investigates the
behavior of BART models with a finite number of \emph{random tree
features}, i.e.~models whose trees are drawn i.i.d. from the BART prior
and not treated as parameters, as in standard BART models. We establish
learning rates for these models, which both help explain BART's
theoretical performance and motivate random tree features as a practical
approximation to full BART modeling in their own right. Usefully, our
rates show that under mild regularity conditions, as few as
\(T_n \asymp n^{1/3} \log(n)^{1+2(p-1)/3}\) trees are sufficient to
achieve a minimax-optimal rate for regression in \(S_1([0,1]^p)\). This
rate avoids the curse of dimensionality and is faster than \(n^{-1/4}\)
for any \(p\).

We consider a standard nonparametric regression setup. Assume
\((X_i, Y_i)_{i=1}^n\iid \Pr\); then we can write \[
Y = g_0(\bx) + \epsilon,
\] for the conditional expectation function (CEF)
\(g_0(\bx) := \E[Y\mid X=\bx]\) and a mean-zero error term \(\epsilon\).
Throughout, we impose two regularity conditions on \(\Pr\):

\begin{enumerate}
\def\labelenumi{(\arabic{enumi})}
\tightlist
\item
  \(X\) has density \(\rho\) with respect to Lebesgue measure which is
  bounded above and below:
  \(0 < \rho_- \le \rho(x) \le \rho_+ < \infty\).
\item
  The noise is sub-exponential: there exist \(\sigma^2,B<\infty\) with
  \(\E\left[\left|Y\right|^c \mid X\right] \le \frac{1}{2} c! \sigma^2 B^{c-2}\)
  for all \(c \geq 2\).
\end{enumerate}

The main modeling assumption we make is that \(g_0\in\mathcal{H}_\b\).
More precisely, we assume that the conditional expectation function
admits a version that belongs to \(\mathcal{H}_\b\) with
\(\norm{g_0}_{\mathcal{H}_\b} \leq R\) for some constant \(R < \infty\).
Our results below do not require a priori knowledge of \(R\).

\subsection{Kernel ridge regression}\label{kernel-ridge-regression}

A well-studied estimator for \(g_0\) in this setting is the kernel ridge
regression estimate \[
\hat g_n = \argmin_{g \in \mathcal{H}_\b} \frac{1}{n} \sum_{i=1}^n \left(Y_i - g(X_i) \right)^2 + \lambda \norm{g}_{\mathcal{H}_\b}^2,
\] for regularization parameter \(\lambda > 0\). The statistical
properties of kernel ridge regression and other kernel learning methods
are strongly determined by the eigenvalues of their corresponding
integral operator
\begin{equation}\protect\phantomsection\label{eq-kernel-op}{
\Sigma_k f(\bx) := \int_{[0, 1]^p} k(\bx, \bx') f(\bx') \rho(\bx') d\bx',
}\end{equation} where \(\rho\) is the density of \(\bx\). Intuitively,
the slower the rate of eigenvalue decay, the larger the corresponding
RKHS, and realizations from the GP will be less smooth. Note that while
the RKHS \(\mathcal{H}_\b\) depends on \(k_\b\), the operator
\(\Sigma_\b := \Sigma_{k_\b}\) depends on \emph{both} \(k_\b\) and the
covariate density \(\rho\).

The following result characterizes the eigenvalues of the operator
\(\Sigma_\b\) under a regular BART prior and the assumed regularity
conditions on \(X\). It follows from the sandwiching relationship
between \(k_\b\) and the \(\ell_1\)-Laplacian kernel established in
Proposition~\ref{prp-rkhs}.

\begin{corollary}[]\protect\hypertarget{cor-eigen}{}\label{cor-eigen}

If the covariate density \(\rho\) is bounded above and bounded away from
zero on \([0,1]^p\), then for a regular BART prior, the eigenvalues
\(\{\varsigma_j\}_{j=1}^\infty\) of \(\Sigma_\b\) satisfy \[
\varsigma_j \asymp j^{-2}\log^{2(p-1)}(j).
\]

\end{corollary}

The eigenvalue decay rate established in Corollary~\ref{cor-eigen}
allows \(\hat g_n\) to achieve a minimax-optimal rate for regression in
\(S_1([0, 1]^p)\), up to logarithmic factors. For example, the following
result is immediate from Theorem 4 of \citet{bak2025effect}.

\begin{proposition}[]\protect\hypertarget{prp-krr-rate}{}\label{prp-krr-rate}

Under the stated assumptions on the data, kernel ridge regression using
\(k_\b\) for a regular BART prior achieves \[
\sqrt{\E(\hat g_{n}(X) - g_0(X))^2}
\lesssim n^{-1/3}\log^{p-1} n.
\]

\end{proposition}

In practice, however, working with \(k_\b\) directly requires \(O(n^3)\)
computation, which is impractical for large \(n\). We next show that the
leading rate in Proposition~\ref{prp-krr-rate} can be obtained with a
finite number of trees.

\subsection{Random tree features}\label{random-tree-features}

Recall that \(\Psi_j(\vb x)\) is the vector of indicators for membership
in the leaf nodes of tree \(j\) given input \(\vb x\). Define
\(\phi_T(\vb x) = T^{-1/2} (\Psi_1(\vb x), \dots, \Psi_T(\vb x))\) as
the concatenated vector of all leaf indicators across all trees,
rescaled by \(T^{-1/2}\). We call the random vector \(\phi_T(\vb x)\) a
\emph{random tree feature} representation of the covariates \(\vb x\).

Let \(\vbg\Phi \in \R^{n \times M}\) represent the matrix of random tree
features, where \(M=\sum_{j=1}^T 2^{D_j}\) is the total number of leaf
nodes across all trees. Then consider an approximate estimator
\(\hat g_n\) obtained by performing ridge regression on the random
features, \begin{equation}\protect\phantomsection\label{eq-rtf-ridge}{
\begin{aligned}
\hat g_{T,n}(\vb x) = \phi_T(\vb x)^\top\hat\beta, \quad\text{where}\quad
\hat\beta &= \argmin_{\beta\in\R^M} \frac{1}{n} \sum_{i=1}^n \left(Y_i - \phi_T(X_i)^\top\beta\right)^2 + \frac{\lambda}{\sigma_\mu^2}\norm{\beta}_2^2 \\
&= \left(\vbg\Phi^\top\vbg\Phi + (n\lambda/\sigma_\mu^2) I\right)^{-1}\vbg\Phi^\top\vb y.
\end{aligned}
}\end{equation} This estimator is approximate in two senses: it
approximates the infinite-tree kernel ridge estimator \(\hat g_n\), and
it approximates the full BART model with \(T\) trees, where the tree
structures are learned in addition to the leaf parameters \(\vb*\mu\).

Intuitively, as \(T \to \infty\),
\(\sigma^2_\mu(\vbg\Phi\vbg\Phi^\top)_{ij}\to k_\b(\bx_i, \bx_j)\).
Early work in random features established rates on this convergence
\citep{rahimi2007random}. However, without further assumptions,
convergence of \(\vbg\Phi\vbg\Phi^\top\) entrywise does not imply
convergence of its inverse, nor does it cleanly map onto a convergence
rate for \(\hat g_{T,n}\) itself. Thus studying the convergence of
\(\hat g_{T,n}\) is more complicated than for \(\hat g_n\) or for the
kernel matrix \(\sigma^2_\mu(\vbg\Phi\vbg\Phi^\top)\). The specific
setting here poses unique theoretical challenges as well. While recent
work has established convergence rates for a large class of random
feature estimates \citep[e.g.,][]{rudi2017generalization}, these results
generally require the random features to be univariate and continuous.
These properties are not satisfied by random tree features and therefore
cannot be directly applied.

We extend results from \citet{rudi2017generalization} to explicitly
consider random tree features and allow for jump discontinuities and
random dimension \(2^D\) in the features. We establish the following
minimax-optimal convergence rate for \(\hat g_{T,n}\), provided the
penalty \(\lambda\) and number of trees \(T\) are chosen appropriately.

\begin{theorem}[]\protect\hypertarget{thm-finite-tree-conv}{}\label{thm-finite-tree-conv}

Under the stated assumptions on the data, for a regular BART prior, if
\(\lambda_n \asymp n^{-2/3}\log(n)^{2(p-1)/3}\) and
\(T_n \ge c_T n^{2/3}\log(n)^{1-2(p-1)/3}\) for some sufficiently large
\(c_T\), then there exists a constant \(n_0\) such that for all
\(n>n_0\), \[
\sqrt{\E(\hat g_{T,n}(X) - g_0(X))^2}
\lesssim n^{-1/3} \log(n)^{(p-1)/3}.
\] Moreover, if \(D\le d^+<\infty\), then the same result holds with
\(T_n \ge c_T n^{1/3}\log(n)^{1 + 2(p-1)/3}\).

\end{theorem}

Theorem~\ref{thm-finite-tree-conv} demonstrates two advantages of
regression on random tree features. First, random tree features can
reduce computational complexity compared to both full BART and kernel
ridge regression with \(k_\b\). The latter requires \(O(n^3)\)
computation, while the theorem shows that for bounded tree depths,
\(O(n^{5/3})\) computation is sufficient, up to logarithmic factors. We
note that it is likely possible to relax the boundedness condition on
\(D\) in Theorem~\ref{thm-finite-tree-conv} to a condition on its tails,
and still require only the \(n^{1/3}\) lower bound on \(T_n\), but such
a result would require substantial modifications to the existing proof
strategy. In any event, standard BART priors place vanishingly small
probability on large \(D\) (\(\Pr(D > 8)< 10^{-5}\) under the Poisson
prior) and so the boundedness condition is not restrictive in practice.

Second, random tree feature regression achieves learning rates that
depend only logarithmically on \(p\), thus avoiding the curse of
dimensionality while still learning functions in a rich function space.
Moreover, as the results in Figure~\ref{fig-ablate-intro} suggest, when
\(T\) is large, very little learning of the tree structure in full BART
occurs, so Theorem~\ref{thm-finite-tree-conv} is also helpful as a
simpler-to-analyze model for studying full BART, existing learning rates
for which still suffer from the curse of dimensionality.\footnote{Or
  require strong assumptions about sparsity in the covariates.}

\section{Empirical Studies}\label{sec-sim}

To complement the theoretical findings in the preceding sections, we
conduct three empirical studies of the BART model and the performance of
random tree features. We extend the ablation study of
Section~\ref{sec-ablate} to additional datasets, and consider how tuning
\(\sigma^2_\mu\) affects performance. We then compare the predictive
performance of the proposed random tree features with full BART,
gradient boosted trees, and random forests, and show that random tree
features perform comparably to these alternative methods under default
settings. Finally, we examine uncertainty quantification for Bayesian
linear regression models fit to random tree features, finding that
learning only leaf parameters, as random tree features do, does not
degrade uncertainty quantification compared to the full BART model.

\begin{table}

\begin{minipage}[t]{0.50\linewidth}

{\def\LTcaptype{none} 
\begin{longtable}[]{@{}lrrr@{}}
\toprule\noalign{}
Dataset & \(n\) & \(p\) & BART \(R^2\) \\
\midrule\noalign{}
\endhead
\bottomrule\noalign{}
\endlastfoot
\texttt{diamonds} & \(53\,940\) & \(9\) & \(0.979\) \\
\texttt{abalone} & \(4\,177\) & \(8\) & \(0.548\) \\
\texttt{cane} & \(3\,775\) & \(31\) & \(0.809\) \\
\texttt{amenity} & \(3\,044\) & \(25\) & \(0.718\) \\
\texttt{edu} & \(2\,339\) & \(6\) & \(0.980\) \\
\texttt{budget} & \(1\,729\) & \(10\) & \(0.997\) \\
\texttt{rice} & \(1\,026\) & \(18\) & \(0.982\) \\
\texttt{attend} & \(838\) & \(9\) & \(0.670\) \\
\end{longtable}
}

\end{minipage}%
\begin{minipage}[t]{0.50\linewidth}

{\def\LTcaptype{none} 
\begin{longtable}[]{@{}lrrr@{}}
\toprule\noalign{}
Dataset (\emph{cont.}) & \(n\) & \(p\) & BART \(R^2\) \\
\midrule\noalign{}
\endhead
\bottomrule\noalign{}
\endlastfoot
\texttt{boston} & \(506\) & \(13\) & \(0.893\) \\
\texttt{diabetes} & \(442\) & \(10\) & \(0.514\) \\
\texttt{mpg} & \(392\) & \(7\) & \(0.881\) \\
\texttt{baseball} & \(263\) & \(19\) & \(0.641\) \\
\texttt{cpu} & \(209\) & \(7\) & \(0.881\) \\
\texttt{ais} & \(202\) & \(12\) & \(0.873\) \\
\texttt{servo} & \(167\) & \(4\) & \(0.925\) \\
\texttt{basketball} & \(96\) & \(4\) & \(0.281\) \\
\end{longtable}
}

\end{minipage}%

\caption{\label{tbl-data}Sixteen datasets studied in Sections
\ref{sec-sim-ablate} and \ref{sec-sim-perf}, arranged by sample size
\(n\). The column \(p\) refers to the number of covariates before
one-hot encoding any categorical covariates, and \(R^2\) refers to the
out-of-sample \(R^2\) value of the full BART model with default
parameters, averaged across cross validation splits and replicates.}

\end{table}%

\subsection{Ablation study}\label{sec-sim-ablate}

We repeat the study in Section~\ref{sec-ablate} on fifteen additional
datasets, summarized in Table~\ref{tbl-data}, which are often used in
machine learning examples.\footnote{ For full disclosure, we initially
  ran experiments on seventeen datasets, but dropped two once we
  realized that they had a panel data structure that was not suited to
  BART modeling without additional structure such as fixed effects. For
  example, in one of these, the \texttt{strikes} dataset, a full BART
  fit yielded an average out-of-sample \(R^2\) of \(-0.208\).} The data
vary across sample sizes \(n\), the \(n/p\) ratio, the number of
continuous covariates, and the signal-to-noise ratio, as measured by the
out-of-sample \(R^2\) of the full BART model. As in the
Section~\ref{sec-ablate} study, each comparison is averaged across 20
random train-test splits and between 1 and 5 replicates (depending on
the dataset size), due to the randomness in fitting the models. Each
split holds out 25\% of the data for evaluating \(R^2\).

\begin{figure}

\centering{

\pandocbounded{\includegraphics[keepaspectratio]{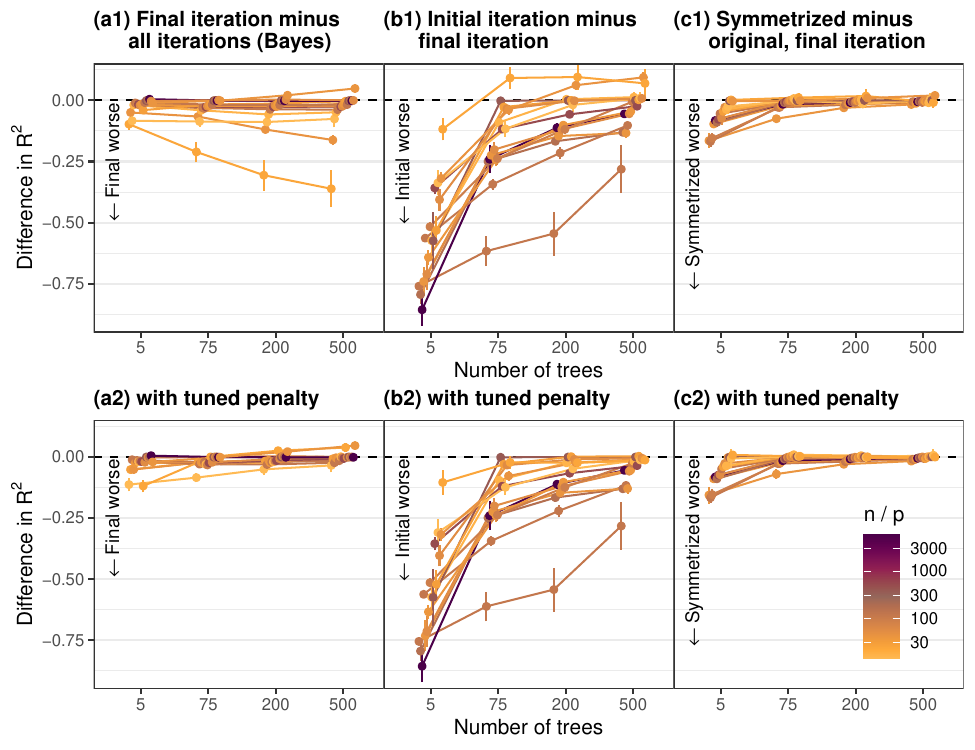}}

}

\caption{\label{fig-ablate-all}\textbf{Full ablation study}. Each panel
compares out-of-sample \(R^2\) values of the posterior mean, for two
different regression fits and different numbers of regression trees,
across 16 different datasets. The average difference in \(R^2\) across
20 cross validation splits and multiple replicates is plotted along with
95\% confidence intervals, with each dataset's line colored according to
the ratio of observations to covariates. The top panel of each pair uses
the original BART prior on the leaf parameters, while the bottom panel
tunes the leaf prior variance using LOOCV. \textbf{(a1) and (a2)
Ablating Bayes:} \(R^2\) for the full posterior mean versus the
posterior mean conditioned on the tree structure at the final MCMC
iteration. \textbf{(b1) and (b2) Ablating tree learning:} \(R^2\)
conditioned on the final-iteration trees versus the initial-iteration
trees. \textbf{(c1) and (c2) Ablating asymmetric trees:} \(R^2\)
conditioned on the final-iteration trees versus a symmetrized version of
those same trees.}

\end{figure}%

The top row of Figure~\ref{fig-ablate-all} shows the results of the
ablation study that correspond directly to
Figure~\ref{fig-ablate-intro}. The overall pattern is extremely similar:
as the number of trees increases, the gain in performance from averaging
across Bayesian uncertainty, learning tree structure, or using
asymmetric trees is minimal.

However, there are several datasets that do not follow this overall
trend, and some datasets that significantly outperform full BART in
panel (b1). We suspect that this is due to two factors: low \(n/p\), and
incorrect tuning of the leaf prior variance \(\sigma^2_\mu\) or the
closely related \(\kappa\) parameter. Indeed, when we re-fit the
comparisons after tuning \(\sigma^2_\mu\) using leave-one-out cross
validation, the results are less variable, and the qualitative patterns
are the same across datasets. Incidentally, the dataset with the highest
gap in panel (b2) is \texttt{cane}, which contains several
high-cardinality categorical covariates; a more careful encoding of
these covariates may improve the performance of full BART and the
ablated models \citep{deshpande2025flexbart}.

\subsection{Predictive performance of random BART
features}\label{sec-sim-perf}

Next, we compare the predictive performance of ridge regression on
random tree features to full BART as well as the two leading tree-based
machine learning methods, gradient boosted trees and random forests.

For each of the 16 datasets in Table~\ref{tbl-data}, we fit three ridge
regression models: one with 75 random tree features, one with 200 such
features, and one where the 200 random tree features are augmented with
the original covariates entering linearly. In each model, we tune the
ridge penalty using leave-one-out cross validation, which can be
computed from the singular value decomposition of the feature matrix, or
an efficient Monte Carlo estimate of generalized cross validation (GCV),
when \(n>1000\). We compare these random tree feature regressions
against BART \citep[\texttt{dbarts},][]{dbarts}, random forests
\citep[\texttt{ranger},][]{ranger}, and gradient boosted trees
\citep[\texttt{xgboost},][]{xgboost}. We use the default settings in
each package, including 75 trees for BART and 500 trees for random
forests. For gradient boosted trees, we use 250 trees (rounds of
boosting), and run the method twice: once with the defaults, and once
5-fold cross-validating key hyperparameters, including the learning rate
and maximum tree depth. As in the ablation study, all results are
averaged over 20 train-test splits and between 1--5 replicates.

\begin{figure}

\centering{

\pandocbounded{\includegraphics[keepaspectratio]{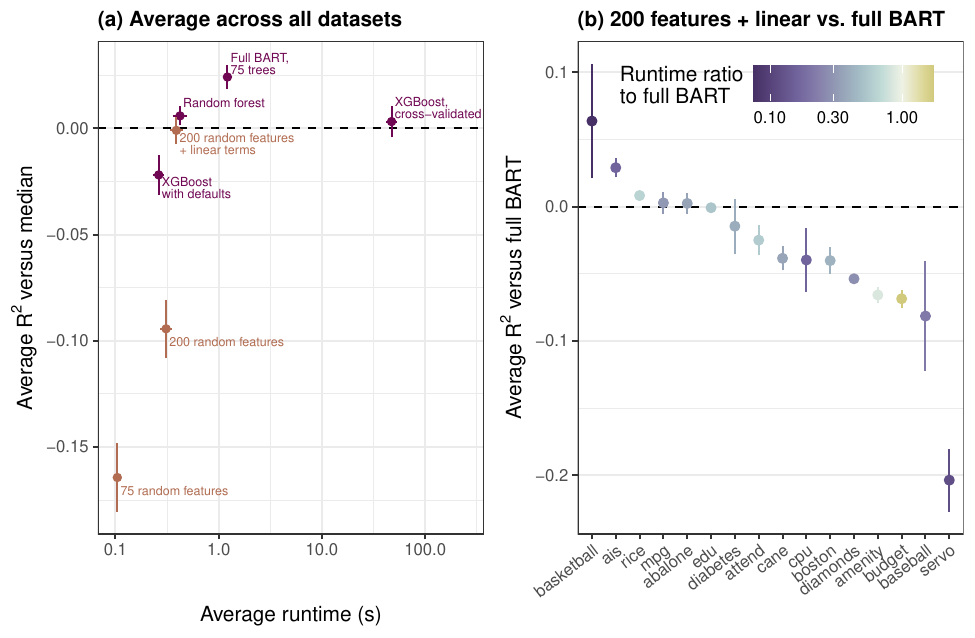}}

}

\caption{\label{fig-perf}\textbf{Predictive performance of random BART
features}. \textbf{(a)} Out-of-sample \(R^2\) values and average runtime
for different methods, averaged across 16 different datasets. \(R^2\)
values are normalized by subtracting the median \(R^2\) for all methods
in a given dataset. \textbf{(b)} Out-of-sample \(R^2\) values for 200
random BART features and additional linear features for continuous
covariates, versus the \(R^2\) value for the full BART fit. Points are
colored by the ratio of the runtime of the random features fit to the
full BART fit.}

\end{figure}%

Figure~\ref{fig-perf}(a) shows the results of this comparison.
Regression on 75 random tree features is by far the fastest method,
though it also suffers a notable loss in performance compared to the
median \(R^2\) across methods, indicated by the horizontal line.
However, increasing the number of trees to 200 improves performance, and
augmenting the trees with a linear covariate specification increases
performance further, to where it exceeds gradient-boosted trees and is
statistically indistinguishable from random forests. The performance of
gradient boosting can be improved through cross validation, but only at
significant computational cost. Full BART is the best-performing method
overall, but it is also 3--4 times slower than the best-performing
random tree feature setup. Overall, random tree features are at or near
the performance-speed frontier, and their place on that frontier can be
adjusted by varying the number of random features.

There is significant variability across datasets in the relative
performance of the methods, as Figure~\ref{fig-perf}(b) shows, comparing
the best-performing random tree feature model to full BART. On some
datasets, like the low-\(R^2\) \texttt{basketball} data, random features
outperform BART at nearly ten times the speed. On others, like the
\texttt{servo} data, random features are moderately faster than full
BART but perform around 0.2 worse in \(R^2\).

Each of the comparison methods is to some extent sensitive to choices of
hyperparameters, and we expect, especially on particular datasets, that
the relative performance of the methods could vary with extensive
hyperparameter tuning. Thus, we see these results as illustrative of
broad patterns in the common case where practitioners use methods
off-the-shelf, and not as a definitive comparison of the methods'
performance when each is tuned optimally.

Given the comparable performance of random features in
Figure~\ref{fig-perf} to leading tree-based methods, which are often the
best-performing machine learning methods for tabular data, it is
reasonable to expect that random tree features could perform well in
models more complicated than regression. For example, random tree
features could be combined with fixed effects, or used to flexibly fit
hazard functions that vary with covariates, or interacted with a main
variable of interest in a varying-coefficient model. It may be more
difficult or impossible for practitioners to adapt the other tree-based
methods to these settings, especially if dedicated software is not
available.

\subsection{Uncertainty
quantification}\label{uncertainty-quantification}

Finally, we investigate how well Bayesian linear regression on random
tree features quantifies uncertainty compared to the full BART model.
While Theorem~\ref{thm-finite-tree-conv} establishes learning rates for
the posterior mean \(\hat g_{T,n}\) when \(T_n\) and the coefficient
prior are chosen appropriately, it does not yield any guarantees on
uncertainty quantification. However, the convergence result in
Theorem~\ref{thm-gp} suggests that the posterior with infinitely many
trees will behave like a Gaussian process. We empirically show in the
following subsection that even with a finite number of trees, the
uncertainty quantification will be similar.

We generate \(n=400\) points from a correlated design bounded to
\([0, 0.75]^2\), and draw a true \(g_0\) from a Gaussian process with
the BART kernel \(k_\b\). Adding noise to these points yields an outcome
\(Y\), which we plot in Figure~\ref{fig-uncert}(a) along with \(g_0\)
itself.

To compare uncertainty quantification, we first fit a Gaussian process
model with the BART kernel \(k_\b\) to the data. Then we fit the full
BART model through \texttt{dbarts}, which automatically rescales the
observed points, and four Bayesian regression models on random tree
features, two models with 75 trees and two with 250 trees. We also vary
whether the tree split points are drawn from \([0,1]\) or just from the
observed range \([0,0.75]\). We expect the former to behave similarly to
the GP model, to which it converges as \(T\to\infty\), while the latter
may look more like the \texttt{dbarts} fit, which only makes splits in
the observed range. After fitting these six models, we measure the
posterior standard deviation of the underlying regression function
\(g_0\) across a grid of points in \([0,1]^2\) and plot the results in
panels (b1)--(b6) of Figure~\ref{fig-uncert}.

\begin{figure}

\centering{

\pandocbounded{\includegraphics[keepaspectratio]{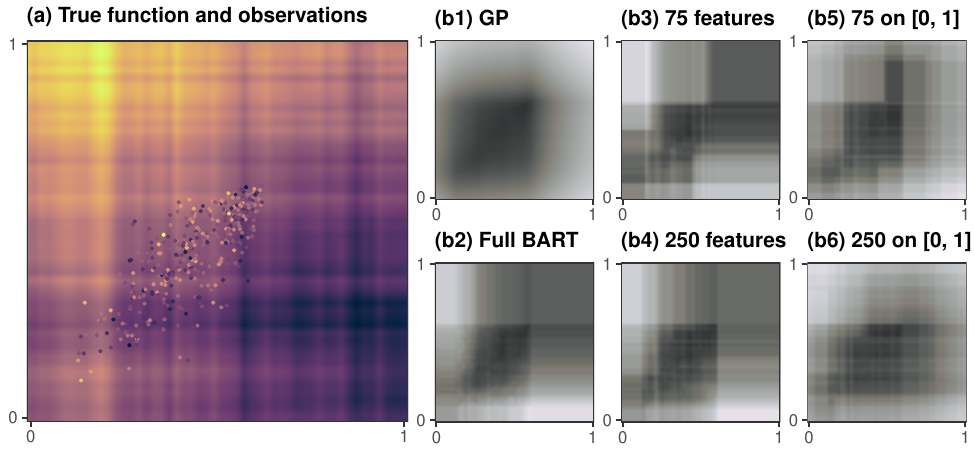}}

}

\caption{\label{fig-uncert}\textbf{Comparison of uncertainty
quantification}. \textbf{(a)} A function on \([0,1]^2\) drawn from
\(\GP(0, k_\b)\), along with 400 points drawn from a design bounded to
\([0, 0.75]^2\). \textbf{(b)} Posterior standard deviation for the
underlying regression function from the full GP model, a BART model with
75 trees and default parameters, and Bayesian linear regressions fit to
random BART features with 75 and 250 trees. Brighter colors indicate
higher posterior variance. The rightmost panels (b5) and (b6) use splits
for each variable drawn uniformly in \([0,1]\), while the middle panels
use splits drawn from the observed data range.}

\end{figure}%

The results are encouraging. Regression on both 75 and 250 random tree
features drawn on \([0,0.75]^2\) yields uncertainty estimates that look
similar to that of the full BART model, with particularly strong
agreement with 250 features. Notably, uncertainty in \(g_0\) does not
increase for any of these models outside the support of the covariates.
In contrast, the GP uncertainty increases with distance from the sampled
points, as does the uncertainty for regression on random tree features
drawn from \([0,1]^2\). The appendix contains additional results
comparing average credible interval length and coverage across a range
of sample sizes and residual variance; the patterns are similar to those
in Figure~\ref{fig-uncert}, though coverage and interval length are
smaller on average for trees using splits outside the observed range.

Thus, despite regression on random tree features not learning the tree
structure as full BART does, it appears that it can still yield
comparable uncertainty quantification, even with a finite number of
trees. Indeed, the similarity to full BART, while exhibiting differences
from the GP model, suggests that some of the advantages of a finite
number of trees \citep{chipman2010bart, jeong2023art} may be retained
when using random tree features.

\section{Discussion}\label{sec-concl}

By considering the \(T\to\infty\) limit of the BART model, we have shown
that BART models converge to a Gaussian process in this limit, and that
growing \(T\) with \(n\) can yield favorable learning rates for
regression, even when no learning of the tree structure occurs. Notably,
these rates are sufficiently fast to use within modern double-debiased
machine learning methods \citep{chernozhukov2018double}. The random tree
features introduced for our study of the learning rates also have
practical utility as a flexible and efficient tool for function
modeling, and appear to capture uncertainty about as well as full BART.

We conclude with several directions for future work, both theoretical
and empirical. On the theoretical side, further generalizing and
simplifying the arguments establishing learning rates for random
features, and extending them to general nuisance function learning with
different loss functions, are interesting avenues of future research. As
discussed in the introduction, the results here draw connections between
random features, past research on learning rates for BART, and
regression in spaces like \(S_1([0, 1]^p)\). We suspect there are
further interesting results to be found in the intersection of these
areas, and beyond the BART model specifically. Within BART, it would be
useful to establish formal results on uncertainty quantification, as
well as a posterior contraction formulation of
Theorem~\ref{thm-finite-tree-conv}.

On the empirical side, we see three interesting directions for future
work. First, many flavors of BART have been proposed, including ``soft''
varieties that use probabilistic splits and can handle smoother
functions. Researchers could empirically study the corresponding GP
kernel for these models by estimating \(\Pr[\bx\adj \bx']\), and could
use random draws from these models' priors as random features. Adapting
the type of random features to the data at hand in this way may yield
improved performance at a smaller number of trees. Second, with a small
number of random tree features, there is significant variability due to
the randomness in the tree structures. It would be of interest to
investigate variance reduction techniques, such as quasi-Monte Carlo
sampling, to reduce this variability. Such techniques have been
fruitfully applied to other types of random features
\citep{liu2021random}. Finally, when not all covariates matter equally,
BART models have been meaningfully improved by adding a hierarchical
variable selection prior on top of the prior on \(V\), the splitting
variable \citep{linero2018var}. It would be of great interest to
understand if it is possible to approximate these priors through a
different penalty on the random tree feature coefficients, and other
ways to perform variable selection while using random features.

\section*{References}\label{references}
\addcontentsline{toc}{section}{References}

\renewcommand{\bibsection}{}
\bibliography{refs.bib}

\newpage{}

\appendix

\renewcommand\thefigure{\thesection.\arabic{figure}}

\setcounter{figure}{0}

\section{Gaussian process proofs}\label{sec-app-proofs-gp}

\subsection{Covariance function
proofs}\label{covariance-function-proofs}

\subsubsection{\texorpdfstring{Proof of
Proposition~\ref{prp-tree-cov}}{Proof of Proposition~}}\label{proof-of-prp-tree-cov}

\begin{proof}
First,
\(\E[\tf(\bx')] = \E[\tf(\bx)] = \E[\E[\tf(\bx) \mid D, \vb V, \vb S]] = 0\).
Thus \[
\begin{aligned}
\Cov[\tf(\bx), \tf(\bx')]
&= \E[\tf(\bx)\tf(\bx')] \\
&= \E\left[\left(\sum_{\vb l\in\L} \mu_{\vb l} \psi_{\vb l}(\bx)\right)\left(\sum_{\vb l'\in\L}\mu_{\vb l'}\psi_{\vb l'}(\bx')\right)\right] \\
&= \E\left[\sum_{\vb l\in\L}\sum_{\vb l'\in\L} \psi_{\vb l}(\bx)\psi_{\vb l'}(\bx')
    \E\left[\mu_{\vb l}\mu_{\vb l'} \middle| \psi_{\vb l}(\bx), \psi_{\vb l'}(\bx'), D\right] \right] \\
&= \E\left[\sum_{\vb l\in\L}\sum_{\vb l'\in\L} \psi_{\vb l}(\bx)\psi_{\vb l'}(\bx')
    \sigma^2_\mu \ind\{\vb l=\vb l'\}\right] \\
&= \sigma^2_\mu \E\left[\sum_{\vb l\in\L}\psi_{\vb l}(\bx)\psi_{\vb l}(\bx')\right] \\
\end{aligned}
\] For any \(\vb l\), \(\psi_{\vb l}(\bx)\psi_{\vb l}(\bx')=1\) if and
only if \(\bx\) and \(\bx'\) are both in leaf node \(\vb l\). Thus
\(\E[\sum_{\vb l\in\L}\psi_{\vb l}(\bx)\psi_{\vb l}(\bx')]=\Pr[\bx\adj\bx']\).
\end{proof}

\subsubsection{\texorpdfstring{Proof of
Proposition~\ref{prp-tree-cov-exact}}{Proof of Proposition~}}\label{proof-of-prp-tree-cov-exact}

\begin{proof}
From Proposition~\ref{prp-tree-cov}, we have \[
\begin{aligned}
\Pr[\bx\adj \bx']
&= \E\left[\sum_{\vb l\in\L}\psi_{\vb l}(\bx)\psi_{\vb l}(\bx')\right] \\
&= \E\left[\sum_{\vb l\in\L}\prod_{k=1}^D
    \ind\{l_kx_{v_k} <_{l_k} l_ks_k\} \ind\{l_kx'_{v_k} <_{l_k} l_ks_k\} \right] \\
&= \E\left[\prod_{k=1}^D \sum_{l\in\{-1, 1\}} \E\left[
    \ind\{lx_{V_k} <_l lS_k\} \ind\{lx'_{V_k} <_l lS_k\}\middle| D\right] \right]. \\
\end{aligned}
\] The final step follows by the distributive property and the fact that
each summand is independent of the others given \(D\), and so the
expectation may be brought inside the product. Let \(U\sim\Unif[0, 1]\),
so \(F_S(S_k)\sim U\). Then we can compute the inner expectation,
conditioning additionally on \(V_k\), as \[
\begin{aligned}
\E&\left[\ind\{lx_{V_k} <_l l S_k\} \ind\{lx'_{V_k} <_l l S_k\} \middle| V_k, D\right] \\
&= \Pr\left[x_{V_k} < S_k, x'_{V_k} < S_k \middle| V_k, D\right]\ind\{l=1\}
+ \Pr\left[x_{V_k} \ge S_k, x'_{V_k} \ge S_k \middle| V_k, D\right]\ind\{l=-1\} \\
&= \Pr[F_S(x_{V_k}) < U, F_S(x'_{V_k}) < U]\ind\{l=1\} \\
&\quad+ \Pr[F_S(x_{V_k}) \ge U, F_S(x'_{V_k}) \ge U]\ind\{l=-1\} \\
&= (1 - F_S(x_{V_k} \vmax x'_{V_k})){\ind\{l=1\}}
    + F_S(x_{V_k} \vmin x'_{V_k}){\ind\{l=-1\}},
\end{aligned}
\] where \(\vmax\) indicates the maximum and \(\vmin\) the minimum. Then
summing over \(l\in\{-1, 1\}\) yields \[
\begin{aligned}
\sum_{l\in\{-1, 1\}} &\E\left[\ind\{lx_{V_k} <_l lS_k\} \ind\{lx'_{V_k} <_l lS_k\}\middle| D, V_k\right] \\
&= (1 - F_S(x_{V_k} \vmax x'_{V_k}))
    + F_S(x_{V_k} \vmin x'_{V_k}) \\
&= 1 - |F_S(x_{V_k}) - F_S(x'_{V_k})|.
\end{aligned}
\] Notice that \[
\begin{aligned}
\E\left[1 - |F_S(x_{V_k}) - F_S(x'_{V_k})| \middle| D\right]
&= 1 - \sum_{v=1}^p |F_S(x_{v}) - F_S(x'_{v})|f_V(v) \\
&= 1 - d_\b(\bx, \bx'),
\end{aligned}
\] so all together we have \[
\sum_{l\in\{-1, 1\}} \E\left[\ind\{lx_{V_k} <_l l S_k\} \ind\{lx'_{V_k} <_l l S_k\} \middle| D\right]
= 1 - d_\b(\bx, \bx').
\] Substituting the above result, which no longer depends on \(D\), we
find \[
\begin{aligned}
\Pr[\bx\adj \bx']
&= \E\left[\prod_{k=1}^D (1 - d_\b(\bx, \bx')) \right] \\
&= \E\left[(1 - d_\b(\bx, \bx'))^D \right] \\
&= \sum_{k=1}^\infty \left(1 - d_\b(\bx, \bx')\right)^k f_D(k).
\end{aligned}
\]
\end{proof}

\subsubsection{\texorpdfstring{Proof of
Corollary~\ref{cor-cov-pois}}{Proof of Corollary~}}\label{proof-of-cor-cov-pois}

\begin{proof}
In this case \[
f_D(k) = \frac{e^{-r}}{1-e^{-r}}\cdot\frac{r^k}{k!},
\] so \[
\begin{aligned}
k_\b(\bx, \bx')
&= \sigma^2_\mu \frac{e^{-r}}{1-e^{-r}} \sum_{k=1}^\infty
    \frac{1}{k!} \left(r - r d_\b(\bx, \bx')\right)^k \\
&= \sigma^2_\mu \frac{e^{-r}}{1-e^{-r}}
     \left(\exp(r - r d_\b(\bx, \bx'))-1\right) \\
&= \frac{\sigma^2_\mu}{1-e^{-r}}
    \left(\exp(-r d_\b(\bx, \bx')) - e^{-r}\right).
\end{aligned}
\]
\end{proof}

\subsubsection{\texorpdfstring{Proof of Proposition~\ref{prp-rkhs} and
Corollary~\ref{cor-eigen}}{Proof of Proposition~ and Corollary~}}\label{proof-of-prp-rkhs-and-cor-eigen}

We begin with a lemma that allows us to work with the simpler version of
\(\Sigma_\b\) in which both \(S\) and \(X\) are uniform.

\begin{lemma}[]\protect\hypertarget{lem-coord-meas}{}\label{lem-coord-meas}

Let \(F(\bx)=(F_S(x_1),\ldots,F_S(x_p))\), and suppose the standing
assumptions on \(f_S\) hold. Define the transformed BART kernel \[
k_u(\bu,\bu')
:=k_\b(F^{-1}(\bu),F^{-1}(\bu')),
\] so that \(k_\b(\bx,\bx')=k_u(F(\bx),F(\bx'))\). Then:

\begin{enumerate}
\def\labelenumi{\arabic{enumi}.}
\tightlist
\item
  The RKHS \(\H_\b\) of \(k_\b\) is the pullback by \(F\) of the RKHS
  \(\H_u\) of \(k_u\): \[
   \H_\b=\{g\circ F:g\in\H_u\},
   \qquad
   \norm{g\circ F}_{\H_\b}=\norm{g}_{\H_u}.
   \] In particular, \(\H_u\) is norm-equivalent to \(S_1([0,1]^p)\) if
  and only if \(\H_\b\) is.
\item
  Suppose \(\rho\) is bounded above and bounded away from zero. Define
  the transformed-kernel operator under the uniform density \[
   (\Sigma_u h)(\bu) :=\int_{[0,1]^p}k_u(\bu,\bu')h(\bu')\,d\bu',
   \] on \(L^2([0,1]^p)\). If \(\varsigma_j(A)\) denotes the ordered
  nonzero eigenvalues of an operator \(A\), then there exist constants
  \(0<m\le M<\infty\) such that, for every \(j\), \[
   m\varsigma_j(\Sigma_u) \le \varsigma_j(\Sigma_\b) \le M\varsigma_j(\Sigma_u).
   \]
\end{enumerate}

\end{lemma}

\begin{proof}
Continuity and strict positivity of \(f_S\) on \([0,1]\) imply that
\(0<m_S\le f_S\le M_S<\infty\) for some constants \(m_S,M_S\). Thus
\(F\) is a coordinatewise bijection with bounded derivative and bounded
inverse derivative. The standard pullback characterization of RKHSs
gives \[
\H_\b=\{g\circ F:g\in\H_u\},
\qquad \norm{g\circ F}_{\H_\b}=\norm{g}_{\H_u}.
\] For \(\vb a\in\{0,1\}^p\), the weak chain rule expresses
\(D^{\vb a}(g\circ F)\) as \((D^{\vb a}g)\circ F\) times a product of
the corresponding \(f_S(x_v)\). The bounds above and a change of
variables therefore show that composition by \(F\) is bounded on
\(S_1([0,1]^p)\). The same argument for \(F^{-1}\) proves bounded
invertibility and hence the first claim.

For the second claim, the pushforward of \(\rho(\bx)d\bx\) under \(F\)
has density \[
q(\bu)=\frac{\rho(F^{-1}(\bu))}{\prod_{v=1}^p f_S(F_S^{-1}(u_v))}
\] with respect to \(d\bu\), and \(0<m\le q\le M<\infty\). Define \[
(\Sigma_{u,q}h)(\bu)
:=\int_{[0,1]^p}k_u(\bu,\bu')h(\bu')q(\bu')\,d\bu',
\] as an operator on \(L^2(q(\bu)d\bu)\). The unitary map
\(h\mapsto h\circ F\) identifies \(\Sigma_\b\) with \(\Sigma_{u,q}\).
Multiplication by \(\sqrt q\) further identifies \(\Sigma_{u,q}\) with
\(M_{\sqrt q}\Sigma_u M_{\sqrt q}\) on \(L^2([0,1]^p)\), where \(M_h\)
denotes multiplication by \(h\). Thus \(\Sigma_\b\), \(\Sigma_{u,q}\),
and \(M_{\sqrt q}\Sigma_u M_{\sqrt q}\) have the same eigenvalues. These
in turn equal those of \(\Sigma_u ^{1/2}M_q\Sigma_u^{1/2}\), since the
latter two operators are \(AA^*\) and \(A^*A\) for
\(A=M_{\sqrt q}\Sigma_u ^{1/2}\). Then, since
\(mI\preceq M_q\preceq MI\) by boundedness of \(q\), \[
m\Sigma_u \preceq \Sigma_u^{1/2}M_q\Sigma_u^{1/2} \preceq M\Sigma_u,
\] and the result follows from the min--max principle for eigenvalues.
\end{proof}

We can now prove Proposition~\ref{prp-rkhs} and
Corollary~\ref{cor-eigen}.

\begin{proof}[Proof of Proposition~\ref{prp-rkhs}]
By Lemma~\ref{lem-coord-meas}, it suffices to work in the transformed
coordinates \(u_v=F_S(x_v)\). We first consider the kernel conditional
on \(D=d\), \[
k_d(\bu,\bu')=\left(1-\sum_{v=1}^p f_V(v)\left|u_v-u_v'\right|\right)^d.
\] We will sandwich this kernel in the Loewner order between two kernels
which are easier to analyze. To compare the fixed-depth kernels, let
\(\ell_m(u,u')=(1-|u-u'|)^m\) for \(m\ge1\), and let \(\ell_0=1\). This
lets us expand \(k_d\) as \[
k_d(\bu,\bu')
=\sum_{\substack{\vb n\in\{0,1,\ldots\}^p\\\sum_v n_v=d}}
\binom{d}{n_1,\ldots,n_p}
\prod_{v=1}^p f_V(v)^{n_v}\ell_{n_v}(u_v,u_v'),
\] where \(\vb n=(n_1,\ldots,n_p)\) are the multinomial counts. Note
that \(\ell_1(u,u')=1-|u-u'|\) is itself a nonnegative definite kernel:
it is the probability that \(u\) and \(u'\) are not separated by a
single uniform split. Since \(\ell_m= (\ell_1)^m\), the Schur product
theorem shows that \(\ell_m\) is also a nonnegative definite kernel for
every integer \(m\ge1\); the same is plainly true of the constant kernel
\(\ell_0\).

We next compare these kernels quantitatively. Extend \(\ell_m\) to the
stationary kernel \(\bar\ell_m(h)=(1-|h|)_+^m\) on \(\R\). The
triangular kernel \(\bar\ell_1\) is nonnegative definite, since it is
the convolution of two indicator functions, and hence so is each power
\(\bar\ell_m\). For \(m\ge2\) and \(\omega\ne0\), the Fourier transform
can be integrated by parts twice: \[
\begin{aligned}
\widehat{\bar\ell}_m(\omega)
&=2\int_0^1(1-t)^m\cos(\omega t)\,dt \\
&=\frac{2m}{\omega}\int_0^1(1-t)^{m-1}\sin(\omega t)\,dt \\
&=\frac{2m}{\omega^2}\left\{1-(m-1)
\int_0^1(1-t)^{m-2}\cos(\omega t)\,dt\right\}.
\end{aligned}
\] For \(m\ge2\), the absolute value of the integral is at most
\(\int_0^1(1-t)^{m-2}dt=(m-1)^{-1}\), so the quantity in braces is at
most two. For \(m=1\), direct integration gives
\(\widehat{\bar\ell}_1(\omega)=2(1-\cos\omega)/\omega^2\le4/\omega^2\).
Thus \(\widehat{\bar\ell}_m(\omega)\le4m/\omega^2\) for every \(m\ge1\)
when \(|\omega|\ge1\). For \(|\omega|\le1\), the Fourier-integral
representation instead gives
\(\widehat{\bar\ell}_m(\omega)\le2\int_0^1(1-t)^m dt=2/(m+1)\).
Combining the two bounds gives
\(\widehat{\bar\ell}_m(\omega)\le Cm(1+\omega^2)^{-1}\) for an absolute
constant \(C\). The Fourier transform of the Laplacian kernel
\(\bar L(h)=\exp(-|h|)\) is \(2(1+\omega^2)^{-1}\). The spectral
characterization of stationary kernels therefore gives \[
\bar\ell_m\preceq Cm\bar L
\] in the Loewner order on \(\R\), and restricting both kernels to
\([0,1]\) preserves this ordering.

For each fixed \(m\ge2\), we can also bound \(\bar\ell_m\) from below.
Indeed, the quantity in braces above is \(1-\E[\cos(\omega T)]\), where
\(T\) has density \((m-1)(1-t)^{m-2}\) on \([0,1]\). It is strictly
positive for \(\omega\ne0\), tends to one as \(|\omega|\to\infty\), and
its second-order expansion at zero is positive. These facts imply that,
for some \(b_m>0\), \(1-\E[\cos(\omega T)]\ge b_m\min\{\omega^2,1\}\)
for all \(\omega\). Substituting this into the preceding display gives
\(\widehat{\bar\ell}_m(\omega)\ge2mb_m\) for \(0<|\omega|\le1\) and
\(\widehat{\bar\ell}_m(\omega)\ge2mb_m\omega^{-2}\) for \(|\omega|>1\);
the first bound also holds at zero by continuity. These two bounds give
\(\widehat{\bar\ell}_m(\omega)\ge a_m(1+\omega^2)^{-1}\) for some
\(a_m>0\). Comparing this spectral density to that of \(\bar L\) gives
\[
a_m\bar L \preceq \bar\ell_m
\] for each \(m\ge2\), and we can likewise restrict this to \([0,1]\).
For \(m=1\), write \[
\ell_1(u,u')=\min\{u,u'\}+\min\{1-u,1-u'\}.
\] The two summands are the covariance kernels of Brownian motion
started at the left and right endpoints, respectively. Their RKHSs are
the subspaces of \(W_1([0,1])\) of functions vanishing at \(0\) and at
\(1\). The RKHS of their sum is the sum of these two spaces, which is
all of \(W_1([0,1])\): for \(g\in W_1([0,1])\), the functions \(ug(u)\)
and \((1-u)g(u)\) belong to the two respective subspaces and sum to
\(g\). This decomposition is bounded in the Sobolev norm, so the
resulting RKHS norm is equivalent to the \(W_1\) norm. Thus we have
extended the preceding lower bound to \(m=1\) as well, and have
established
\begin{equation}\protect\phantomsection\label{eq-ell-loewner}{
a_m L \preceq \ell_m \preceq CmL
}\end{equation} in the Loewner order, where \(L\) is the restriction of
\(\bar L\) to \([0,1]\), where \(a_m>0\) may depend on \(m\), while
\(C\) does not. This ordering equivalently maps to an inclusion ordering
of the kernel's corresponding RKHSs.

Finally, the RKHS of \(\ell_0\) is the one-dimensional space of constant
functions. Its inclusion in the RKHS of \(L\) is bounded, so the
standard RKHS inclusion criterion gives \(\ell_0\preceq C_0L\) for some
\(C_0<\infty\). Increasing \(C\) if necessary, we therefore have
\(\ell_m\preceq C(1+m)L\) for every \(m\ge0\).

Having sandwiched \(\ell_m\), we can sandwich \(k_d\) as well. Since
kernel ordering is preserved under tensor products, for every count
vector \(\vb n\) appearing in the multinomial expansion, \[
\prod_{v=1}^p\ell_{n_v}
\preceq C^p\prod_{v=1}^p(1+n_v)\,K_\otimes,
\qquad\text{where}\quad
K_\otimes:=\prod_{v=1}^p L.
\] The multinomial coefficients and the factors \(\prod_v f_V(v)^{n_v}\)
are nonnegative and sum to one over \(\vb n\). Moreover,
\(\prod_v(1+n_v)\le(1+d)^p\) whenever \(\sum_vn_v=d\). Summing the
preceding comparison over the multinomial expansion therefore yields \[
k_d\preceq C^p(1+d)^pK_\otimes.
\]

If \(d\ge p\), the multinomial sum also contains a term with every
\(n_v\ge1\), since \(f_V\) has support on all of \([p]\). Fix one such
count vector \(\vb n^*\). By the lower bound in
Eq.~\ref{eq-ell-loewner}, \[
\prod_{v=1}^p\ell_{n_v^*}\succeq
\left(\prod_{v=1}^pa_{n_v^*}\right)K_\otimes.
\] All other terms in the expansion are nonnegative definite, so they
can only increase the kernel in the nonnegative-definite ordering. It
follows that \[
k_d\succeq c_dK_\otimes
\] for some \(c_d>0\) whenever \(d\ge p\).

We can now extend the bounds on \(k_d\) to bounds on \(k_\b\) itself.
Representing \(k_\b\) as a mixture over the \(k_d\) as \[
k_\b=\sigma_\mu^2\sum_{d\ge1}f_D(d)k_d,
\] we consider comparing \(k_\b\) to kernels built from portions of the
total sum. Since \(\Pr(D\ge p)>0\), there is some fixed \(d_0\ge p\) for
which \(f_D(d_0)>0\). Keeping only that nonnegative-definite term in the
mixture gives \[
k_\b\succeq \sigma_\mu^2f_D(d_0)k_{d_0}
\succeq \sigma_\mu^2f_D(d_0)c_{d_0}K_\otimes.
\] In the other direction, summing the fixed-depth upper bounds gives \[
k_\b\preceq
\sigma_\mu^2 C^p\E[(1+D)^p]K_\otimes
\preceq C'K_\otimes,
\] where \(C'<\infty\) because \(\E[D^p]<\infty\). Hence
\(cK_\otimes\preceq k_\b\preceq C'K_\otimes\) for some
\(0<c<C'<\infty\). The two-sided kernel ordering shows that the
corresponding RKHSs are norm-equivalent. Since the RKHS of \(L\) is
\(W_1([0,1])\), the RKHS of \(K_\otimes\) is the tensor product
\(\bigotimes_{v=1}^p W_1([0,1])=S_1([0,1]^p)\). This proves
Proposition~\ref{prp-rkhs}.
\end{proof}

\begin{proof}[Proof of Corollary~\ref{cor-eigen}]
By the second part of Lemma~\ref{lem-coord-meas}, the eigenvalues of
\(\Sigma_\b\) are equivalent up to constants to those of the transformed
kernel under the uniform density. By the kernel sandwich in
Proposition~\ref{prp-rkhs} and the min--max characterization of
eigenvalues, these are in turn equivalent up to constants to the
eigenvalues \(\{\rho_j\}\) of the uniform-density operator associated
with \(K_\otimes\).

Finally, the eigenvalues of the univariate operator associated with
\(L\) are of order \(j^{-2}\) \citep{ritter1995}. The eigenvalues of its
\(p\)-fold tensor product \(K_\otimes\) are therefore all products
\(\prod_{v=1}^p\rho^{(1)}_{j_v}\), where \(\rho^{(1)}_j\asymp j^{-2}\).
The number of these products exceeding \(\lambda\) is, up to constants,
the number of integer tuples satisfying
\(\prod_vj_v\lesssim\lambda^{-1/2}\), which is of order
\(\lambda^{-1/2}\log^{p-1}(1/\lambda)\). Inverting this counting
relation shows that the reordered eigenvalues satisfy
\(\rho_j\asymp j^{-2}\log^{2(p-1)}(j)\); see also
\citet{zhang2023regression}, Appendix B.
\end{proof}

\subsection{GP convergence proofs}\label{gp-convergence-proofs}

\subsubsection{\texorpdfstring{Proof of
Lemma~\ref{lem-clt-fdd}}{Proof of Lemma~}}\label{proof-of-lem-clt-fdd}

\begin{proof}
Let \(Z_{j} := (\tf_{j}(\bx_1), \dots, \tf_{j}(\bx_k))^\top\) be the
vector of outputs from a tree \(\tf_j\) drawn independently from the
BART prior, so \[
(\bf_T(\bx_1), \dots, \bf_T(\bx_k)) \sim T^{-1/2} \sum_{j=1}^T Z_j.
\] For every \(j\), we have \(\E[Z_j]=0\) by iterated expectations and
\(\Cov[Z_j] = K\) by Proposition~\ref{prp-tree-cov}. Then the result
follows from the multivariate central limit theorem.
\end{proof}

\subsubsection{\texorpdfstring{Proof of
Theorem~\ref{thm-gp}}{Proof of Theorem~}}\label{proof-of-thm-gp}

The proof of Theorem~\ref{thm-gp} requires several preliminary lemmas
and setup.

We say that a tree function prior has \emph{\(K\)-Lipschitz separation
probability} if the conclusion of Lemma~\ref{lem-leaf-prob} holds, i.e.,
there exists a \(1\le K<\infty\) such that for any
\(\bx,\by\in[0,1]^p\),
\(\Pr[\bx\not\adj\by\mid D] \le K D \norm{\bx-\by}_1\).

\begin{lemma}[]\protect\hypertarget{lem-leaf-prob}{}\label{lem-leaf-prob}

Let \(\bx,\by\in[0,1]^p\) and \(\tf\) be a tree function drawn from the
symmetric BART prior. If \(F_S\) is Lipschitz continuous with constant
\(1\le K_S<\infty\), then \[
\Pr[\bx\not\adj\by\mid D] \le K_S D \norm{\bx-\by}_1,
\]

\end{lemma}

\begin{proof}
By the proof of Proposition~\ref{prp-tree-cov}, \[
\Pr[\bx\not\adj\by] = 1 - \sum_{k=1}^\infty \left(1-d_\b(\bx, \by)\right)^k f_D(k)
= 1 - \E\left[\left(1-d_\b(\bx, \by)\right)^D\right],
\] so \(\Pr[\bx\not\adj\by\mid D]=1 - (1-d_\b(\bx, \by))^D\). Now, \[
\begin{aligned}
\frac{\partial}{\partial\delta}(1-\delta)^d &= -d (1-\delta)^{d-1}
\qand \\
\frac{\partial^2}{\partial\delta^2}(1-\delta)^d &= d(d-1)(1-\delta)^{d-2}.
\end{aligned}
\] Thus \(\Pr[\bx\not\adj\by\mid D]\) is concave in \(d_\b(\bx, \by)\)
everywhere, and we can therefore upper bound it by the tangent line at
the origin, i.e., \[
\Pr[\bx\not\adj\by\mid D] \le \Pr[\bx\not\adj\bx\mid D] + D(1-0)^{D-1} d_\b(\bx, \by) = D d_\b(\bx, \by).
\] Then, \[
\begin{aligned}
d_\b(\bx, \by) &:= \sum_{v=1}^p f_V(v) |F_S(x_v) - F_S(y_v)| \\
&\le \sum_{v=1}^p |F_S(x_v) - F_S(y_v)|  \\
&\le \sum_{v=1}^p K_S|x_v - y_v|
= K_S\norm{\bx - \by}_1.
\end{aligned}
\] Thus \(\Pr[\bx\not\adj\by\mid D] \le K_S D \norm{\bx-\by}_1\), as
claimed.
\end{proof}

Let \(S \subseteq [p]\) represent a set of coordinates. Let \(\bx_S\)
represent a vector, where the coordinates in \(S\) are identical to
\(\bx\), and the rest are zero. Let \(\bx_{-S}\) be defined identically,
with the complement of \(S\). For a point \(\bx\) and an offset
\(\epsilon\in\R^p_{\ge 0}\), define the \emph{closed block}
\([\bx, \bx^{+\epsilon}]\) by \[
[\bx, \bx^{+\epsilon}] = \prod_{q=1}^p [x_q, x_q + \epsilon_q].
\] The behavior of tree functions on these closed blocks is key to
establishing tightness of the BART process through several critical
lemmas.

We call the \(q\) where \(\epsilon_q>0\) the \emph{active dimensions} of
the (closed) block, and we define the \emph{volume} of the block as
\(\nu([\bx, \bx^{+\epsilon}]) := \prod_{q:\epsilon_q>0} |\epsilon_q|\),
and the \emph{increment} of \(\tf\) around the block as \[
\Delta \tf([\bx, \bx^{+\epsilon}]) :=
\sum_{S \subseteq \{q: \epsilon_q>0\}} (-1)^{|S|} \tf(\bx_S + \bx_{-S}^{+\epsilon}).
\] For example, if \(p=3\), \(\bx=0\), and \(\epsilon=(1, 0, 1)\), then
the active dimensions are \(\{1, 3\}\), the volume is 1, and \[
\Delta \tf([\bx, \bx^{+\epsilon}])
= \tf(0, 0, 0) - \tf(1, 0, 0) - \tf(0, 0, 1) + \tf(1, 0, 1).
\]

The points
\(\V([\bx, \bx^{+\epsilon}]) = \{\bx_S + \bx_{-S}^{+\epsilon}: S \subseteq \{q: \epsilon_q>0\} \}\)
represent the active vertices of the closed block
\([\bx, \bx^{+\epsilon}]\). For symmetric trees, we can bound the
probability that all the active vertices are in different leaf nodes by
the volume of the block.

\begin{lemma}[]\protect\hypertarget{lem-vtx-bound}{}\label{lem-vtx-bound}

Let \(\A([\bx, \bx^{+\epsilon}])\) represent the event that all the
active vertices \(\V([\bx, \bx^{+\epsilon}])\) for a closed block
\([\bx, \bx^{+\epsilon}]\) are in different leaf nodes in a symmetric
tree \(\tf\), i.e., \[
\A([\bx, \bx^{+\epsilon}]) :=
\bigcap_{\vb a\neq\vb b \in\V([\bx, \bx^{+\epsilon}])} \vb a \not\adj \vb b.
\] If \(\tf\) has \emph{\(K\)-Lipschitz separation probability}, then \[
\Pr\left[\A([\bx, \bx^{+\epsilon}])\middle| D\right] \le
(K D)^p \ \nu([\bx, \bx^{+\epsilon}]).
\]

\end{lemma}

\begin{proof}
We can rewrite \(\Pr[\A([\bx, \bx^{+\epsilon}])\mid D]\) as \[
\begin{aligned}
\Pr\left[\A([\bx, \bx^{+\epsilon}])\middle| D\right]
&= \Pr \left[ \bigcap_{\vb a\neq \vb b \in \V([\bx, \bx^{+\epsilon}])}
     \vb a \not\adj \vb b \middle| D\right] \\
&= \Pr \left[ \bigcap_{q:\epsilon_q>0}
\bigcap_{S^a=S^b\cup\{q\}}
(\bx_{S^a}+\bx_{-S^a}^{+\epsilon}) \not\adj (\bx_{S^b}+\bx_{-S^b}^{+\epsilon})\middle| D\right],
\end{aligned}
\] where the inner intersection is over all
\(S^a,S^b\subseteq \{q\in [p]:\epsilon_q>0\}\) that differ only in their
inclusion of \(q\).

In other words, for each active dimension \(q\), we examine pairs of
points that differ only in coordinate \(q\). But in fact these events
are all the same due to the symmetry of the tree: if two points that
differ in coordinate \(q\) are in different leaf nodes, then there
exists some decision rule \(k\) with \(V_k=q\) and
\(x_q< S_k\le x_q+\epsilon_q\). This decision rule will also separate
any other pair of points that differ in coordinate \(q\). Thus \[
\Pr \left[ \bigcap_{q:\epsilon_q>0}
\bigcap_{S^a=S^b\cup\{q\}}
(\bx_{S^a}+\bx_{-S^a}^{+\epsilon}) \not\adj (\bx_{S^b}+\bx_{-S^b}^{+\epsilon})\middle| D\right]
= \Pr\left[\bigcap_{q:\epsilon_q>0} \bx\not\adj(\bx_{-\{q\}} + \bx^{+\epsilon}_{\{q\}})\middle| D \right].
\] This final probability is the event that a decision rule separates
\(\bx\) from \(\bx^{+\epsilon}\) in each coordinate. There are \(D\)
decision rules, drawn independently given \(D\). As a result, the set of
events considered on the right-hand side are negatively associated,
conditional on \(D\) \citep{dubhashi1996balls}. This means that \[
\Pr \left[ \bigcap_{q:\epsilon_q>0} \bx \not\adj (\bx_{-\{q\}} + \bx_{\{q\}}^{+\epsilon})\middle| D \right]
\le \prod_{q:\epsilon_q>0}\Pr\left[\bx \not\adj (\bx_{-\{q\}} + \bx_{\{q\}}^{+\epsilon})\middle| D\right].
\] Since
\(\norm{\bx - (\bx_{-\{q\}} + \bx_{\{q\}}^{+\epsilon})}_1=\epsilon_q\),
by assumption we conclude \[
\Pr\left[\A([\bx, \bx^{+\epsilon}])\middle| D\right]  \le
 \prod_{q:\epsilon_q>0}  KD|\epsilon_q|
\le (KD)^p \ \nu([\bx, \bx^{+\epsilon}]). \qedhere
\]
\end{proof}

Consider now a general tree \(\tf\) and block \(B\). The following
result on the moments of \(\tf(B)\) will be useful.

\begin{lemma}[]\protect\hypertarget{lem-inc-moments}{}\label{lem-inc-moments}

Let \(\tf\) be a tree drawn from the symmetric BART prior, and let
\(B=[\bx, \bx^{+\epsilon}]\) be a closed block. Then
\(\E[\Delta\tf(B)]=0\) and if \(\tf\) has \emph{\(K\)-Lipschitz
separation probability} and \(\E[D^p]<\infty\), then \[
\E\left[\Delta_\tf(B)^2\right]\le \sigma^2_\mu (2K)^p\,\E[D^p]\,\nu(B).
\]

\end{lemma}

\begin{proof}
Since \(\tf\) is mean zero pointwise, each term in \(\Delta_\tf(B)\) is
mean zero, and the first result is immediate.

For the second moment bound, notice that unless there is a decision rule
separating each pair of corner points along each coordinate dimension,
the increment \(\Delta_\tf(B)\) will be zero. To see this, let \(q\) be
the dimension without a separating rule. Then for any
\(S^a,S^b\subseteq [p]\) differing only in the inclusion of \(q\), the
points \(\bx_{S^a}+\bx_{-S^a}^{+\epsilon}\) and
\(\bx_{S^b}+\bx_{-S^b}^{+\epsilon}\) are in the same leaf node, and so
\(\tf(\bx_{S^a}+\bx_{-S^a}^{+\epsilon})=\tf(\bx_{S^b}+\bx_{-S^b}^{+\epsilon})\).
Thus \[
\begin{aligned}
\Delta_\tf(B) &= \sum_{S\subseteq \{q:\epsilon_q > 0\}}
    (-1)^{|S|} \tf(\bx_S + \bx_{-S}^{+\epsilon}) \\
&= \sum_{S^a=S^b\cup\{q\}} (-1)^{|S^a|}(
    \tf(\bx_{S^a} + \bx_{-S^a}^{+\epsilon}) - \tf(\bx_{S^b} + \bx_{-S^b}^{+\epsilon}))
= 0.
\end{aligned}
\] Thus by the law of total expectation, \[
\E\left[\Delta_\tf(B)^2\right]
= \E\left[\Delta_\tf(B)^2 \middle| \A\{\V(B)\}\right]
    \Pr[\A\{\V(B)\}].
\]

Since the leaf parameters are drawn independently, conditional on
\(\A\{\V(B)\}\) (all corner points in different leaf nodes), the value
of \(\tf\) at each corner point is independent, and so
\begin{equation}\protect\phantomsection\label{eq-tree-h-b}{
\begin{aligned}
\E\left[\Delta_\tf(B)^2\right]
&= \E\left[\left(\sum_{S\subseteq \{q:\epsilon_q > 0\}} (-1)^{|s|} \tf(\bx_S + \bx_{-S}^{+\epsilon})\right)^2
    \middle| \A\{\V(B)\}] \Pr[\A\{\V(B)\}\right] \\
&= \sum_{S\subseteq \{q:\epsilon_q > 0\}} \E\left[(-1)^{2|s|} \tf(\bx_S + \bx_{-S}^{+\epsilon})^2
    \middle| \A\{\V(B)\}] \Pr[\A\{\V(B)\}\right] \\
&= 2^{|\{q:\epsilon_q > 0\}|} \sigma^2_\mu \times \E[\Pr[\A\{\V(B)\}\mid D]] \\
&\le \sigma^2_\mu (2K)^p\,\E[D^p]\,\nu(B). 
\end{aligned}
}\end{equation}

\qedhere

\end{proof}

We now turn our attention to \[
\Delta\bf_T([\bx, \bx^{+\epsilon}])
= \frac{1}{T^{1/2}} \sum_{j=1}^T \Delta_j([\bx, \bx^{+\epsilon}]),
\] where we write \(\Delta_j=\Delta_{\tf_j}\) for simplicity, and in
particular the product of the increments around two blocks \(B\) and
\(C\). We say that two blocks \(B=[\bx,\bx^{+\epsilon^B}]\) and
\(C=[\by,\by^{+\epsilon^C}]\) are \emph{compatible} if for each
\(q\in[p]\), either \([x_q, x_q+\epsilon^B_q]=[y_q, y_q+\epsilon^C_q]\)
or the intervals intersect at a point or not at all. For compatible
blocks, we can calculate a volume \(\nu(B\cap C)\); by our definition of
volume above, this is the volume over the dimensions where the blocks
are identical, i.e., \[
\nu(B\cap C) := \prod_{q: [x_q, x_q+\epsilon^B_q]=[y_q, y_q+\epsilon^C_q]} |\epsilon^B_q|.
\] When there are no such dimensions, we define \(\nu(B\cap C)=1\). For
example, when \(p=3\) with two cubes \(B\) and \(C\) of the same size,
all of the following arrangements are compatible: \(B\) and \(C\)
identical, \(B\) and \(C\) sharing a face, \(B\) and \(C\) aligned on an
axis but not touching, and \(B\) and \(C\) diagonally separated.

\begin{lemma}[]\protect\hypertarget{lem-inc-bc}{}\label{lem-inc-bc}

If \(\tf\) has \emph{\(K\)-Lipschitz separation probability} and
\(\E[D^{2p}]<\infty\), then there exists a constant \(K_p<\infty\) such
that for any BART function \(\bf_T\) with \(T\) trees, and any two
compatible blocks \(B\) and \(C\), \[
\E\left[|\Delta \bf_T(B)|^2 |\Delta \bf_T(C)|^2\right]
\le K_p \nu(B)\nu(C) \left(1 + \frac{1}{T\,\nu(B\cap C)}\right).
\]

\end{lemma}

\begin{proof}
Substituting, \[
\begin{aligned}
\E \left[|\Delta \bf_T(B)|^2 |\Delta \bf_T(C)|^2 \right]
=& \E \left[ \left( \frac{1}{T^{1/2}} \sum_{j=1}^T \Delta_j(B)\right)^2
\left(\frac{1}{T^{1/2}} \sum_{k=1}^T \Delta_k(C) \right)^2 \right] \\
=& \frac{1}{T^{2}} \sum_{i,j,k,l=1}^T \E\left[
    \Delta_i(B) \times \Delta_j(B) \times \Delta_k(C) \times \Delta_l(C) \right]
\end{aligned}
\]

There are four cases for the indices \(i,j,k,l\).

\emph{Case 1: \(i = j \neq k = l\) (\(T^2 - T\) total terms).} In this
setting, the expectation can be written as: \[
\begin{aligned}
\E\left[\Delta_i(B)^2 \times \Delta_k(C)^2 \right] &= \E\left[\Delta_i(B)^2 \right] \times \E\left[\Delta_k(C)^2 \right],
\end{aligned}
 \] since trees \(i\) and \(k\) are independent. Substituting in the
result of Eq.~\ref{eq-tree-h-b}, \[
\E\left[\Delta_i(B)^2 \times \Delta_k(C)^2\right]
\le \sigma^4_\mu (2K)^{2p} \E[D^p]^2 \nu(B)\nu(C).
\]

\emph{Case 2: \(i=k\ne j=l\) or \(i=l\ne j=k\) (\(2T^2-2T\) terms).} We
can write these terms as \[
\begin{aligned}
\E\left[\Delta_i(B) \times \Delta_i(C) \times \Delta_j(B) \times \Delta_j(C)\right]
= \E\left[\Delta_i(B) \times \Delta_i(C)\right]^2
\end{aligned}
\] with the equality following because trees \(i\) and \(j\) are i.i.d.
Then applying Cauchy-Schwarz and substituting in the result of
Eq.~\ref{eq-tree-h-b}, \[
\begin{aligned}
\E\left[\Delta_i(B) \times \Delta_i(C) \times \Delta_j(B) \times \Delta_j(C)\right]
&\le \E\left[\Delta_i(B)^2 \right] \E\left[\Delta_i(C)^2\right] \\
&\le \sigma^4_\mu (2K)^{2p} \E[D^p]^2 \nu(B)\nu(C).
\end{aligned}
\]

\emph{Case 3: \(i=j=k=l\) (\(T\) terms).} By the same argument as above,
increment \(\Delta_i(B)\) is zero unless all corner points in \(\V(B)\)
are in different leaf nodes, and similarly for \(\Delta_i(C)\). By the
argument in Lemma~\ref{lem-vtx-bound}, this event is the equivalent to
the event that a decision rule separates each pair of interval endpoints
for the intervals that define \(B\) and \(C\). Because \(B\) and \(C\)
are compatible, there is one interval per dimension for both \(B\) and
\(C\), but there is double-counting of the intervals on which \(B\) and
\(C\) agree. Thus overall, by a similar argument as above, \[
\Pr[\Delta_i(B)\Delta_i(C)\neq 0]
\le K^{2p}\E[D^{2p}] \frac{\nu(B)\nu(C)}{\nu(B\cap C)},
\] and so \[
\begin{aligned}
\E\left[\Delta_i(B)^2\,\Delta_i(C)^2\right]
&= \E\left[\Delta_i(B)^2\,\Delta_i(C)^2 \middle|
    \Delta_i(B)\Delta_i(C)\neq 0\right]
    \Pr[\Delta_i(B)\Delta_i(C)\neq 0] \\
&\le 3 \sigma^4_\mu (2K)^{2p}\E[D^{2p}]\frac{\nu(B)\nu(C)}{\nu(B\cap C)},
\end{aligned}
\] where the bound on the expectation follows from the same argument as
in Eq.~\ref{eq-tree-h-b}; the factor of 3 is due to
\(\E[X^2Y^2]\le3\E[X^2]\E[Y^2]\) for zero-mean Gaussians \(X\) and \(Y\)
without assumptions on their dependence.

\emph{Case 4: All other cases.} In this case, there is at least one
index that is distinct from the others, and so one of the differences is
independent of the others, since it involves a different tree, and the
trees are independent. Since each tree increment has mean zero, each of
these terms is zero as well.

Combining the four cases, we have shown \[
\begin{aligned}
\E &\left[|\Delta \bf_T(B)|^2 |\Delta \bf_T(C)|^2 \right] \\
&\le \frac{1}{T^{2}}\left(
    3(T^2 - T)\cdot \sigma^4_\mu (2K)^{2p} \E[D^p]^2 \nu(B)\nu(C)  +
    T\cdot 3\sigma^4_\mu (2K)^{2p} \E[D^{2p}] \frac{\nu(B)\nu(C)}{\nu(B\cap C)}
\right) \\
&\le \sigma^4_\mu (2K)^{2p} \E[D^{2p}] \nu(B)\nu(C)
    \left(3\frac{T^2 - T}{T^2}  + \frac{3}{T\,\nu(B\cap C)}\right) \\
&\le 3\sigma^4_\mu (2K)^{2p} \E[D^{2p}]  \nu(B)\nu(C)
    \left(1 + \frac{1}{T\,\nu(B\cap C)}\right).
\end{aligned}
\] Letting \(K_p := 3\sigma^4_\mu (2K)^{2p} \E[D^{2p}]\) completes the
proof.
\end{proof}

Let \(\X\) be a rectangular subset of \([0,1]^p\) and let \(E\) be a
Banach space with norm \(\norm{\cdot}_E\). When \(E=\R\), this norm is
the usual absolute value and we drop the subscript. For a function
\(f:\X\to E\), we say that \(f\in \mathcal{C}_T(\X, E)\) if there exists
a constant \(K_p<\infty\) such that for all \(\lambda > 0\) and every
pair of compatible blocks \(B\) and \(C\), \[
\Pr[\min \{\norm{\Delta_f(B)}_E, \norm{\Delta_f(C)}_E\} \ge \lambda]
\le K_p\lambda^{-4} \nu(B)\nu(C) \left(1 + \frac{1}{T\,\nu(B\cap C)}\right).
\] Applying Chebyshev's inequality and the fact that
\(\min \{\norm{\Delta_f(B)}_E, \norm{\Delta_f(C)}_E\}^4\le \norm{\Delta_f(B)}_E^2\norm{\Delta_f(C)}_E^2\),
to show that \(f\in \mathcal{C}_T(\X, E)\) it suffices to show \[
\E\left[\norm{\Delta_f(B)}_E^2 \norm{\Delta_f(C)}_E^2\right]
\le K_p\nu(B)\nu(C) \left(1 + \frac{1}{T\,\nu(B\cap C)}\right).
\] \(\cC_T(\X, E)\) is effectively a modified condition \(\cC(2, 4)\)
from \citet{bickel1971convergence}, one which allows for a \(1/T\) term
but which is more restrictive in holding for any compatible \(B\) and
\(C\), not just those that are adjacent.

To state the critical result, for \(1\le q\le p\), let \[
    f^{(q)}_x (x_1,\dots, x_{q-1},x_{q+1},\dots,x_p) := f(x_1,\dots,x_{q-1},x,x_{q+1},\dots,x_p).
\] Equivalently, we could write
\(f^{(q)}_x(\by) := f(\by_{-\{q\}} + e_q x)\), where \(e_q\) is the
\(q\)th standard basis vector in \(\R^p\). We can view \(f^{(q)}_x\) as
a function on \([0,1]\) mapping to the space of functions \(\D_{p-1}\)
with the supremum norm \(\norm{\cdot}_\infty\). Then define \[
\begin{aligned}
m_q(x, y, z)(f) &:= \min\{\norm{f_y^{(q)} - f_x^{(q)}}_\infty, \norm{f_z^{(q)} - f_y^{(q)}}_\infty\} \\
&= \min\{\Delta_{f^{(q)}}([x, y]), \Delta_{f^{(q)}}([y, z])\};
\end{aligned}
\] the second equality represents \(m_q\) as a comparison of increments
of the one-dimensional function \(f^{(q)}\). Finally, define \[
M''_q(f) := \sup_{0\le x\le y\le z\le 1} m_q(x, y, z)(f) \qand
M''(f) := \max_{1\le q\le p} M''_q(f).
\]

Then, extending Theorem 1 from \citet{bickel1971convergence}, we have
the following lemma.

\begin{lemma}[]\protect\hypertarget{lem-condition-bound}{}\label{lem-condition-bound}

Let \(q\in [p]\) be a dimension, \(0\le\chi<\xi\le 1\), and define
\(\X:=\{\bx\in[0, 1]^p : x_q\in[\chi, \xi]\}\). Then there exists a
constant \(L_p<\infty\) depending only on \(p\) and \(K_p\) such that
for all \(\lambda>0\), and for any \(f\in \cC_T(\X,E)\), \[
\Pr[M''_q(f) \ge\lambda] \le \lambda^{-4} L_p \nu(\X)^2 = \lambda^{-4} L_p (\xi-\chi)^2.
\]

\end{lemma}

\begin{proof}
The proof closely follows the structure of the proof of Theorem 1 in
\citet{bickel1971convergence}, specializing certain aspects (e.g.,
Lebesgue measure, \(\beta=2\)) while generalizing to cover the
additional term scaled by \(1/T\) and functions which may not vanish on
their lower boundary. We proceed by induction on \(p\).

When \(p=1\), and the blocks \(B\) and \(C\) are adjacent the condition
\(f\in\cC_T(\X, E)\) means in fact that \[
\Pr[\min \{\norm{\Delta_f(B)}_E, \norm{\Delta_f(C)}_E\} \ge \lambda]
\le \lambda^{-4} K'_1 \nu(B)\nu(C),
\] since in one dimension the term \(\nu(B\cap C)=1\) and the \(1/T\)
term can be absorbed into the constant. Since for adjacent \(B\) and
\(C\), \(\nu(B)\nu(C)\le\frac{1}{4}(\nu(B)+\nu(C))^2\), this condition
is stronger than \(\cC(2,4)\) in Bickel \& Wichura's notation, and the
base case then follows from the base case (i) proven in Theorem 1 of
their work. That case does not depend on the function vanishing at its
lower boundary, which is a condition of their Theorem 1.

Before turning to the inductive step, we first claim that for any \(p\)
and any \(f:\X\to E\), that \[
\norm{f}_\infty \le p M''(f) + \max_{\vbg\delta\in\V(\X)} \norm{f(\vbg\delta)}_E,
\] i.e., any function is bounded by \(p\) times its maximum increment
\(M''\) plus the maximum of its values at the vertices of the block
\(\X\). To establish this claim, without any loss of generality we can
take \(\X=[0,1]^p\), which will simplify the notation. First notice that
for any \(a,b,c\in E\), \[
\norm{a}_E \le \min\{\norm{a-b}_E, \norm{c-a}_E\} + \max\{\norm{b}_E, \norm{c}_E\}.
\] Then, for any \(\vbg\delta\in\{0,1\}^k\) for \(1\le k\le p\), let
\(\vb x_{\vbg\delta\bullet}\) denote the point in \([0,1]^p\) with
\(x_i = \delta_i\) for \(1\le i\le k\) and with the remaining
coordinates identical to those of \(\bx\). Thus, for example,
\(\bx_{11\bullet}\) and \(\bx_{10\bullet}\) differ only in coordinate 2.
Then for any \(\bx\in[0,1]^p\), \[
\begin{aligned}
\norm{f(\bx)}_E &\le \min\{\norm{f(\bx) - f(\bx_{0\bullet})}_E, \norm{f(\bx_{1\bullet}) - f(\bx)}_E\}
    +\max\{\norm{f(\bx_{0\bullet})}_E, \norm{f(\bx_{1\bullet})}_E\} \\
&\le M''_1(f) +\max\{\norm{f(\bx_{0\bullet})}_E, \norm{f(\bx_{1\bullet})}_E\} \\
\end{aligned}
\] By the same argument, \[
\begin{aligned}
\norm{f(\bx_{0\bullet})}_E &\le M''_2(f) + \max\{\norm{f(\bx_{00\bullet})}_E, \norm{f(\bx_{01\bullet})}_E\} \qand \\
\norm{f(\bx_{1\bullet})}_E &\le M''_2(f) + \max\{\norm{f(\bx_{10\bullet})}_E, \norm{f(\bx_{11\bullet})}_E\}.
\end{aligned}
\] Substituting these two inequalities into the previous one, we can
pull the \(M''_2(f)\) terms out and collapse the maximum, yielding \[
\norm{f(\bx)}_E \le M''_1(f) + M''_2(f) + \max\{\norm{f(\bx_{00\bullet})}_E, \norm{f(\bx_{01\bullet})}_E,
    \norm{f(\bx_{10\bullet})}_E, \norm{f(\bx_{11\bullet})}_E\}.
\] Continuing in this way for all \(p\) dimensions, we obtain \[
\norm{f(\bx)}_E \le \sum_{q=1}^p M''_q(f) + \max_{\vbg\delta\in\{0,1\}^p} \norm{f(\vbg\delta)}_E
\le \sum_{q=1}^p M''(f) + \max_{\vbg\delta\in\{0,1\}^p} \norm{f(\vbg\delta)}_E,
\] which proves the claimed inequality.

For the inductive step, we can without loss of generality take \(q=1\).
Define functions \(g := f_y^{(1)}-f_x^{(1)}\) and
\(h := f_z^{(1)}-f_y^{(1)}\) (no relation to BART or tree functions). We
have by definition of \(m_1\) and application of the above inequality
that for any \((x,y,z)\) with \(\chi\le x\le y\le z\le\xi\), \[
\begin{aligned}
m_1(x, y, z)(f) &= \min\{\norm{g}_\infty, \norm{h}_\infty\} \\
&\le \min\{ (p-1)M''(g) + \max_{\vbg\delta\in\V(\X_{-1})} \norm{g(\vbg\delta)}_E,\
          (p-1)M''(h) + \max_{\vbg\delta\in\V(\X_{-1})} \norm{h(\vbg\delta)}_E \} \\
&\le \min\{ (p-1)\max\{M''(g),M''(h)\} + \max_{\vbg\delta\in\V(\X_{-1})} \norm{g(\vbg\delta)}_E, \\
    &\qquad\qquad (p-1)\max\{M''(g),M''(h)\} + \max_{\vbg\delta\in\V(\X_{-1})} \norm{h(\vbg\delta)}_E \} \\
&= (p-1)\max\{M''(g),M''(h)\} + \min\{\max_{\vbg\delta\in\V(\X_{-1})} \norm{g(\vbg\delta)}_E,
          \max_{\vbg\delta\in\V(\X_{-1})} \norm{h(\vbg\delta)}_E \} \\
&\le (p-1)\max\{M''(g),M''(h)\} + \max_{\vbg\delta,\vbg\delta'\in\V(\X_{-1})}
    \min\{\norm{g(\vbg\delta)}_E, \norm{h(\vbg\delta')}_E \},
\end{aligned}
\] where \(\X_{-1}\) is \(\X\) with the \(q=1\) coordinate removed. We
will bound each term on the right-hand side separately.

To bound \(M''(g)\), notice that for any block \(B'\) in \(\X_{-1}\),
\(\Delta_g(B') = \Delta_f(B)\), where \(B=[x, y]\times B'\). Thus
\(g\in\cC_T\) in \(p-1\) dimensions, with \(\nu\) scaled by \(y-x\) in
\(\X_{-1}\). By the inductive hypothesis
\citep[see][]{bickel1971convergence}, we therefore have for each
remaining coordinate \(r\) that \[
\Pr[M''_r(g) \ge \lambda] \le L_{p-1}\lambda^{-4}(y-x)^2.
\] A union bound over coordinates then yields \[
\Pr[M''(g) \ge \lambda] \le (p-1)L_{p-1}\lambda^{-4}(y-x)^2.
\] The same inequality follows for \(M''(h)\), with \(z-y\) in place of
\(y-x\). Thus by a union bound, \[
\begin{aligned}
\Pr&[(p-1)\max\{M''(g),M''(h)\} \ge \lambda] \\
&\le \Pr[(p-1)M''(g) \ge \lambda] + \Pr[(p-1)M''(h) \ge \lambda] \\
&\le L_{p-1} (p-1)^5 \lambda^{-4} (z-x)^2 .
\end{aligned}
\] By a similar argument, we can bound
\(\min\{\norm{g(\vbg\delta)}_E, \norm{h(\vbg\delta')}_E\}\) for any
\(\vbg\delta,\vbg\delta'\in\V(\X_{-1})\), and then multiply this bound
by \(2^{2(p-1)}\) to obtain a bound on
\(\max_{\vbg\delta,\vbg\delta'\in\V(\X_{-1})} \min\{\norm{g(\vbg\delta)}_E, \norm{h(\vbg\delta')}_E \}\).
But by the original hypothesis that \(f\in\cC_T(\X,E)\), for any
\(\vbg\delta,\vbg\delta'\), \[
\begin{aligned}
\Pr[\min\{\norm{g(\vbg\delta)}_E, \norm{h(\vbg\delta')}_E\} \ge \lambda]
&\le K_p \lambda^{-4} (z-y)(y-x)\left(1+\frac{1}{T}\right) \le 2K_p \lambda^{-4} (z-x)^2,
\end{aligned}
\] since the intervals \([x, y]\) and \([y, z]\) do not overlap.

Since for any random variables \(U\) and \(V\),
\(\Pr[U+V\ge \lambda] \le \Pr[U\ge \lambda/2] + \Pr[V\ge \lambda/2]\),
we have established \[
\Pr[m_1(x, y, z)(f) \ge \lambda]
\le 2^4\cdot(L_{p-1} (p-1)^5 + 2^{2p}K_p)\lambda^{-4} (z-x)^2 .
\] But this in fact means that \(f^{(1)}\) meets the hypotheses of the
base case \(p=1\) as a univariate \((E')\)-valued process, as discussed
above. Thus the conclusion for the inductive step follows by letting
\(L_p := K_1(2,4) 2^4\cdot(L_{p-1} (p-1)^5 + 2^{2p}K_p)\), where
\(K_1(2,4)\) is the universal constant from the univariate case in
\citet{bickel1971convergence}.
\end{proof}

We can now prove the main convergence result, Theorem~\ref{thm-gp}.

\begin{proof}[Proof of Theorem~\ref{thm-gp}]
We apply the corollary to Theorem 2 of \citet{bickel1971convergence},
which requires convergence of the finite-dimensional distributions,
existence of a continuous version of the limiting process, and a
tightness condition, specifically the convergence of a certain modulus
to zero. Lemma~\ref{lem-clt-fdd} establishes f.d.d. convergence.
Existence of a sample-continuous version of the limiting GP is given by
the Kolmogorov continuity theorem and Lemma~\ref{lem-leaf-prob}. The
argument of Theorem 3 of \citet{bickel1971convergence} then establishes
the tightness condition, with our Lemma~\ref{lem-condition-bound}
replacing their Theorem 1. The only additional condition for Theorem 3
to apply is that \(\bf_T\in\cC_T([0,1]^p, \R)\) with a constant uniform
in \(T\). This is established by Lemma~\ref{lem-inc-bc}, whose Lipschitz
hypothesis is satisfied by Lemma~\ref{lem-leaf-prob} under the
assumptions of this theorem.
\end{proof}

\newpage{}

\section{Random tree features proofs}\label{sec-app-proofs-rf}

To prove Theorem~\ref{thm-finite-tree-conv}, we will need to prove a
slight generalization of Theorem 6 of \citet{rudi2017generalization}
(herein RR17), which is sufficient along with the argument below in
Appendix~\ref{sec-prf-finite-conv} to establish
Theorem~\ref{thm-finite-tree-conv}. In particular, RR17 considers the
setting in which each random feature maps to a single function; this
means that a feature's contribution to the kernel operator is rank-1. In
contrast, for BART, each random tree produces \(2^D\) total leaf
indicators, and so each feature's contribution to the kernel operator is
rank-\(2^D\), where \(D\) is random. In the first subsection, we carry
out this generalization, culminating in Theorem~\ref{thm-rr17-gen}. The
second subsection then specializes the result to our case and proves
Theorem~\ref{thm-finite-tree-conv}.

\subsection{Generalizing RR17}\label{generalizing-rr17}

First, we need to set up some additional notation and establish some
basic properties. Then we will generalize certain intermediate results
from RR17, which will be used in the proof of the main result.

Throughout we use \(\E_X\) and \(\norm{\cdot}_{X}\) to denote the
expectation and \(L^2\) norm with respect to the distribution of the
covariates \(X\). Expectations without a subscript average over all
relevant random quantities. We use \(\norm{\cdot}_{op}\) to denote the
operator norm for linear operators on \(L^2(X)\), where the input and
output norms are \(\norm{\cdot}_X\).

To start, defined a rescaled \(\Sigma\) by
\(\Sigma_\b = \sigma_\mu^2 \Sigma\). Then let
\(\Sigma_\omega: L^2(X) \to L^2(X)\) represent the one-tree
approximation of \(\Sigma\), defined as \[
\Sigma_{\omega} := \sum_{\vb l\in\L} \psi_{\vb l} \otimes \psi_{\vb l},
\] where
\((\Sigma_w g)(x) = \sum_{\vb l\in\L} \psi_{\vb l}(x) \E_X[\psi_{\vb l}(X) g(X)]\),
and \(\psi_{\vb l}\) is the indicator function for leaf \(\vb l\) of the
tree defined by \(\omega\). Then we have
\(\Sigma=\E_\omega[\Sigma_\omega]\), i.e., averaging over the
distribution of the random features.

Because we have assumed \(g_0 \in \mathcal{H}_\b\), the source condition
is satisfied, where for \(\tilde g \in L^2(X)\),
\[g_0(x) = (\Sigma^{1/2} \tilde g)(x) \text{ a.s.}\] We will define
\(R\) as the \(L^2(X)\) norm of the function \(\tilde g\) (i.e.,
\(R := \norm{\tilde g}_X\)).

We define additional operators analogously to RR17. As in RR17, these
operators are implicitly conditional on a particular draw of the random
features \(\omega\). As in the main text, let
\(M = \sum_{j=1}^T 2^{D_j}\) be the total number of leaf nodes across
the \(T\) trees, which depends on \(\omega\) as well.

\begin{itemize}
\tightlist
\item
  \(S_T: R^M \to L^2(X), \quad (S_T\beta)(\vb x) = \phi_T(\vb x)^\top \beta\)
\item
  \(S^*_T: L^2(X) \to R^M, \quad (S^*_T f)_j = \E_X[f(X)\phi_{T,j}(X)]\)
\item
  \(\Sigma_T: L^2(X) \to L^2(X), \quad (\Sigma_T f)(\vb z) = \E_X[\phi_T(X)^\top \phi_T(\vb z) f(X)]\),
  i.e., \(\Sigma_T = T^{-1} \sum_{j=1}^T \Sigma_{\omega_j}\)
\item
  \(C_T: R^M \to R^M, \quad C_T(\vb z) = \E_X[\phi_T(X) \phi_T(X)^\top] \vb z\)
\end{itemize}

We define \(\hat S_T\), \(\hat S^*_T\), and \(\hat C_T\) as the
empirical versions of the above operators, i.e., with the expectation
taken with respect to the empirical distribution of the covariates
\(X_1,\dots,X_n\).

Remark 8 of RR17 carries over, stating that

\begin{enumerate}
\def\labelenumi{(\arabic{enumi})}
\tightlist
\item
  \(\Sigma\) and \(\Sigma_T\) are trace class
\item
  \(\Sigma_T=S_T S^*_T\), \(C_T = S^*_T S_T\), and
  \(\hat C_T = \hat S^*_T \hat S_T\).
\item
  \(\Sigma\), \(\Sigma_T\), \(C_T\), and \(\hat C_T\) are self-adjoint
  and positive operators, with spectrum in \([0, 1]\).
\end{enumerate}

We can also establish some basic properties of the kernel operators.
Intuitively, these are sufficient for the random tree features to behave
like the continuous, univariate features studied in RR17.

\begin{lemma}[]\protect\hypertarget{lem-sigma-prop}{}\label{lem-sigma-prop}

\(\Sigma\), \(\Sigma_T\), and \(\Sigma_\omega\) have the following
properties:

\begin{enumerate}
\def\labelenumi{\arabic{enumi}.}
\tightlist
\item
  \(\Sigma - \Sigma_T\) is an average of \(T\) i.i.d., zero-mean,
  self-adjoint random operators on \(L^2(X)\).
\item
  Uniform operator bound: \(\norm{\Sigma_\omega}_{op} \leq 1\) for all
  \(\omega\)
\item
  Uniform trace bound: \(\text{Tr}(\Sigma_\omega) = \leq 1\)
\item
  Sub-idempotency: \(\Sigma^2_\omega \leq \Sigma_\omega\)
\end{enumerate}

\end{lemma}

\begin{proof}
We will prove each property in turn.

\emph{Property 1.} We can re-write \(\Sigma - \Sigma_T\) as \[
\Sigma - \Sigma_T = \frac{1}{T} \sum_{j=1}^T (\Sigma - \Sigma_{\omega_j}).
\] Since \(\omega_1,\dots,\omega_T\) are i.i.d.,
\(\Sigma - \Sigma_{\omega_j}\) are also i.i.d. By definition,
\(\Sigma = \E_\omega[\Sigma_\omega]\), and
\(\E[\Sigma - \Sigma_T] = \Sigma - \E_\omega[\Sigma_\omega] = 0\).
Finally, we can show the self-adjointness of \(\Sigma_\omega\):
\begin{align*}
  \langle \Sigma_\omega f, g \rangle 
  & := \E_X[(\Sigma_w f)(X)g(X)] \\
  &= \E_X\left[ \left(\sum_{\vb l} \psi_{\vb l}(X)
      \E_{Z\sim X}[\psi_{\vb l}(Z) f(Z)] \right) g(X) \right] \\
  &= \sum_{\vb l}
      \E_X[\psi_{\vb l}(X) f(X)] \E_X[\psi_{\vb l}(X) g(X)] \\
  &= \E_X\left[ \left(\sum_{\vb l} \psi_{\vb l}(X)
      \E_{Z\sim X}[\psi_{\vb l}(Z) g(Z)] \right) f(X) \right] \\
  &= \langle f, \Sigma_\omega g \rangle.
\end{align*}
Since \(\Sigma\) and \(\Sigma_T\) are linear combinations of
self-adjoint operators, \(\Sigma - \Sigma_T\) is self-adjoint as a
result.

\emph{Property 2.} We can bound the operator norm of \(\Sigma_\omega\)
by applying Cauchy-Schwarz. More concretely, for any \(g \in L^2(X)\),
\begin{align*}
\norm{\Sigma_\omega g}_X^2
&= \norm{\sum_{\vb l} \psi_{\vb l}(X) \E_{Z\sim X}[\psi_{\vb l}(Z) g(Z)]}_X^2 \\
&= \sum_{\vb l} \E_X[\psi_{\vb l}(X)] (\E_{Z\sim X}[\psi_{\vb l}(Z) g(Z)])^2 \\
&\le \sum_{\vb l} \E_X[\psi_{\vb l}(X)] \E_{Z\sim X}[\psi_{\vb l}(Z)^2]\E_{Z\sim X}[g(Z)^2] \\
&= \norm{g}_X^2 \sum_{\vb l} \E_X[\psi_{\vb l}(X)]^2  \\
&\le \norm{g}_X^2 \sum_{\vb l} \E_X[\psi_{\vb l}(X)] \\
&= \norm{g}_X^2
\end{align*} 
where the second line follows from the fact that the \(\psi_{\vb l}\)
are orthogonal, the third line is Cauchy-Schwarz, and the penultimate
lin follows from \(\E_X[\psi_{\vb l}(X)] \leq 1\) for all \({\vb l}\).
Taking the square root of both sides yields
\(\norm{\Sigma_\omega}_{op} \leq 1\).

\emph{Property 3.} We compute the trace of \(\Sigma_\omega\) using the
linearity of the trace:
\begin{align*} 
\text{Tr}(\Sigma_\omega) &= \sum_{\vb l } \text{Tr}(\psi_{\vb l}\otimes \psi_{\vb l}) \\
&= \sum_{\vb l} \norm{\psi_{\vb l}}^2_X \\
&= 1, 
\end{align*}
which follows directly from the fact that for \(\vb l \in \mathcal{L}\),
\(\psi_{\vb l}\) can only take on a value of 1 for one \(\vb l\).

\emph{Property 4.} Sub-idempotency follows from applying the uniform
operator bound (i.e., property 2): \[
\Sigma^2_\omega \leq \norm{\Sigma_\omega}_{op} \Sigma_\omega \leq \Sigma_\omega. \qedhere
\]
\end{proof}

Finally, we define a generalized notion of the maximum random features
dimension as \[
\cF_{\infty}(\lambda) := \sup_\omega \Tr(\Sigma_\lambda^{-1} \Sigma_\omega), \text{ where } \Sigma_\lambda = \Sigma + \lambda I.
\] \(\cF_\infty(\lambda)\) is bounded by \(1/\lambda\), by property 3 in
Lemma~\ref{lem-sigma-prop}. Furthermore, we define \[
\cN(\lambda) := \Tr(\Sigma_\lambda^{-1} \Sigma) = \E_\omega[\Tr(\Sigma_\lambda^{-1} \Sigma_\omega)].
\]

\subsubsection{Generalized result and proof
approach}\label{generalized-result-and-proof-approach}

With these preliminaries, out of the way, now state a generalization of
Theorem 6 in RR17 for our setting. We maintain the use of \(\kappa\),
the bound on the random features, which here equals 1, to aid in
comparison with RR17 as well as future generalizations. So, e.g., we
have \(\cF_\infty(\lambda) \leq \kappa^2/\lambda\), since \(\kappa:=1\)
here, but the same general bound is used by RR17, and would hold here
for an operator \(\Sigma\) bounded by \(\kappa^2\) rather than 1.

\begin{theorem}[Alternative RR17 Theorem
6]\protect\hypertarget{thm-rr17-gen}{}\label{thm-rr17-gen}

For ridge penalty \(0 < \lambda \leq \frac{3}{4} \norm{\Sigma}_{op}\),
and for all \(\delta \in (0,1)\), if
\begin{equation}\protect\phantomsection\label{eq-M-bound-rudi}{
T \geq 18(q_0 + \cF_\infty(\lambda)) \log\left(\frac{84\kappa^2}{\lambda \delta}\right) 
}\end{equation} and
\begin{equation}\protect\phantomsection\label{eq-n-bound-rudi}{
n \geq 18(2 + \kappa^2/\lambda) \log\left( \frac{28\kappa^2}{\lambda \delta} \right),
}\end{equation} where
\(q_0 := 2(2+\kappa^2/\norm{\Sigma}_{op} + \kappa^2)\), then with
probability at least \(1 - \delta\),
\begin{equation}\protect\phantomsection\label{eq-rudi-bound}{
\norm{\hat g_{\lambda,T,n} - g_0}_X \le 
8 \left( \frac{(B+ 2R\kappa) \kappa}{\sqrt{\lambda}n } 
+ \sqrt{\frac{(\sigma + R\kappa)^2 \cN(\lambda)}{n}} \right)
\log \left(\frac{14}{\delta} \right) + 3R\sqrt{\lambda},
}\end{equation} where \(B\) and \(\sigma\) are determined by the
sub-exponential condition on \(Y\), i.e., for
\(\E\left[\left|Y\right|^c \mid X\right] \le \frac{1}{2} c! \sigma^2 B^{c-2}\),
for all \(c \geq 2\).

\end{theorem}

The proof of Theorem 6 of RR17 relies on a decomposition of the excess
risk, which we restate here using the notation and operators defined
above. The three terms below correspond to sample error, computational
error, and approximation error, respectively. Note that
\(\cS(\lambda,T,n)\), \(\cC(\lambda,T)\), and \(\beta\) are random
quantities that depend on the drawn set of random features.

\begin{theorem}[]\protect\hypertarget{thm-rr17-decomp}{}\label{thm-rr17-decomp}

For any \(\lambda > 0\) and \(T\in \N\), \[
\norm{\hat g_{\lambda,T,n} - g_0}_X \le \beta(\cS(\lambda,T,n) + \cC(\lambda,T) + R\sqrt{\lambda}),
\] where \[
\cS(\lambda,T,n) := \norm{C^{-1/2}_{T,\lambda} (\hat S^*_T \hat y - S^*_T g_0)}_X 
    + R\norm{C^{-1/2}_{T,\lambda} (C_T - \hat C_T)}_{op},
\] with \(\hat y = n^{-1/2}(Y_1,\dots,Y_n)\), \[
\cC(\lambda,T) := R\sqrt{\lambda} \norm{
    \Sigma_\lambda^{-1/2}(\Sigma-\Sigma_T)\Sigma_\lambda^{-1/2}
    }_{op}, \qand
\] \[
\beta := \max\{1, (1-\beta_1)^{-1}\}\max\{1, (1-\beta_2)^{-1/2}\}, 
\] with
\(\beta_1 := \varsigma_{\max} (C^{-1/2}_{T,\lambda} (C_T - \hat C_T))\)
and
\(\beta_2 := \varsigma_{\max} (\Sigma_\lambda^{-1/2}(\Sigma-\Sigma_T)\Sigma_\lambda^{-1/2})\),
where \(\varsigma_{\max}\) denotes the maximum eigenvalue.

\end{theorem}

\begin{proof}
Theorem~\ref{thm-rr17-decomp} follows from an identical argument to the
proof of Theorem 4 in RR17, which proceeds by adding and subtracting
terms to \(\hat g_{\lambda,T,n} - g_0\), grouping them, and applying the
triangle inequality. Each new term is then bounded by Lemmas 2--5 in
RR17. These lemmas go through identically, generally because each
involves purely algebraic manipulations, and because of the properties
shown above in Lemma~\ref{lem-sigma-prop}. In RR17, Lemmas 1 and 2
(which relies on Lemma 1) involve reexpressing an integral, but this
reexpression is equally valid when the features are multivariate. The
fact that the dimension of the features \(\Psi(\vb x)\) is also random
does not affect the argument, which in RR17 is implicitly made pointwise
within a single \(\omega\).
\end{proof}

From here, we will adapt the arguments in RR17 to bound each term in
Theorem~\ref{thm-rr17-decomp}, before combining them to prove
Theorem~\ref{thm-rr17-gen}. While most of the key results from RR17 can
be generalized to our operator setting, several pieces require
adaptation. The first is the adaptive dimension concentration result
(Proposition 10), which is used to bound the sampling error \(\cS\). The
second is re-proving the bound on \(\cC\).

\subsubsection{Bernstein bound lemma}\label{bernstein-bound-lemma}

The following lemma will be repeatedly used in bounding each term in
Theorem~\ref{thm-rr17-decomp}.

\begin{lemma}[]\protect\hypertarget{lem-c_lambda}{}\label{lem-c_lambda}

Suppose \(0<\lambda\le \norm{\Sigma}_{op}\). Then for
\(\delta \in (0,1]\), with probability at least \(1-\delta\),
\begin{equation}\protect\phantomsection\label{eq-clambda-bound}{
\norm{\Sigma_\lambda^{-1/2} (\Sigma - \Sigma_T) \Sigma_\lambda^{-1/2}}_{op} \leq \frac{2 (\cF_{\infty}(\lambda) + 1)\log\left( \frac{8\kappa^2}{\lambda \delta} \right)}{3T} + \sqrt{\frac{2\cF_\infty(\lambda) \log\left( \frac{8\kappa^2}{\lambda \delta} \right)}{T}}.
}\end{equation}

\end{lemma}

\begin{proof}
Define \(A := \Sigma_\lambda^{-1/2} \Sigma \Sigma_\lambda^{-1/2}\) and
\(B_j := \Sigma_\lambda^{-1/2} \Sigma_{\omega_j} \Sigma_\lambda^{-1/2}\),
so that \(\E[B_j] = A\), and let \(X_j := A - B_j\). Expanding
\(\Sigma_T\) in terms of \(\Sigma_\omega\), \[
\norm{\Sigma_\lambda^{-1/2} (\Sigma - \Sigma_T) \Sigma_\lambda^{-1/2}}_{op}
= \norm{\frac{1}{T} \sum_{j=1}^T X_j}_{op}.
\] Note that \(X_j\) is an i.i.d., self-adjoint, zero-mean random
operator. Then, for all \(j\), \[
\norm{X_j}_{op} \leq \norm{A}_{op} + \norm{B_j}_{op} \leq \cF_\infty(\lambda) + 1,
\] where the inequality follows from the fact that
\(\norm{B_j}_{op}=\norm{\Sigma_\lambda^{-1/2} \Sigma_{\omega_j} \Sigma_\lambda^{-1/2}}_{op}\leq  \Tr(\Sigma_\lambda^{-1} \Sigma_{\omega_j}) \leq \cF_\infty(\lambda)\)
by the cyclic property of the trace and the definition of
\(\cF_\infty(\lambda)\), and
\(\norm{A}_{op}=\norm{\Sigma_\lambda^{-1/2} \Sigma \Sigma_\lambda^{-1/2}}_{op} \le 1\)
since \(\Sigma_\lambda=\Sigma + \lambda I\).

The second moment of \(X_j\) is also bounded by \[
\begin{aligned} 
\E_\omega[X_j^2] &= A^2 - A\E_\omega[B_j] - \E_\omega[B_j]A + \E_\omega[B_j^2] \\
&= \E_\omega[B_j^2] - A^2 \\
&\le \E_\omega[B_j^2] \\
&\leq \E_\omega[\cF_\infty (\lambda) B_j] = \cF_\infty(\lambda) A,
\end{aligned} 
\] where the second to last line follows because \(A\) is positive
semidefinite, and the last line follows from spectral calculus: if
\(\norm{A}_{op}\leq c\), then \(A^2 \leq c A\) in the Loewner order,
which is preserved by expectations.

As such, we have shown that the second moment of \(X_j\) is bounded, and
the operator norm is also bounded. We can then apply a two-sided
Bernstein bound for operators, which is RR17, Proposition 3 using
\(\delta_0=\delta/2\). Substituting into that bound yields \[
\norm{\frac{1}{T}\sum_{j=1}^T X_j}_{op}
\leq \frac{2(\cF_\infty(\lambda)+1)\beta}{3T}
+ \sqrt{\frac{2\cF_\infty(\lambda)\beta}{T}},
\] with probability at least \(1-\delta\), where
\(\beta := \log\left(\frac{2\Tr(A)}{\norm{A}_{op}\delta_0}\right)\).
Since \(\lambda\leq\norm{\Sigma}_{op}\), \[
\norm{A}_{op} = \frac{\norm{\Sigma}_{op}}{\norm{\Sigma}_{op}+\lambda}
\geq \frac{1}{2}
\qand
\Tr(A) \le \frac{\Tr(\Sigma)}{\lambda}
\le \frac{\kappa^2}{\lambda}.
\] Therefore, \[
\beta
= \log\left(\frac{2\Tr(A)}{\norm{A}_{op}\delta_0}\right)
\leq \log\left(\frac{4\kappa^2}{\lambda\delta_0}\right)
= \log\left(\frac{8\kappa^2}{\lambda\delta}\right),
\] and substituting this upper bound for \(\beta\) gives the result.
\end{proof}

\subsubsection{Concentration of effective
dimension}\label{concentration-of-effective-dimension}

We will next adapt the effective dimension concentration result from
RR17 (Proposition 10). To start, define
\(\cN_T(\lambda) := \Tr(\Sigma_T \Sigma_{T,\lambda}^{-1})\) as the
empirical effective dimension, computed using only the finite subsample
of \(T\).

\begin{lemma}[]\protect\hypertarget{lem-eff-dim-bnd}{}\label{lem-eff-dim-bnd}

For \(\delta\in(0,1]\), \(\lambda\le\norm{\Sigma}_{op}\), and
\(T\ge \left(4+18\cF_\infty(\lambda)\right)\log\left(\frac{12\kappa^2}{\lambda\delta}\right)\),
with probability at least \(1-\delta\), \[
|\cN_T(\lambda) - \cN(\lambda)| \le 1.55 \cN(\lambda).
\]

\end{lemma}

\begin{proof}
The proof structure follows almost identically to RR17's original
derivation, with the key extension being that we must show that the
operator analogue to the original objects are also well-behaved (i.e.,
finite mean/variance and boundedness), allowing us to apply Bernstein's
inequality. To start, following RR17 Proposition 7, which is a purely
operator-theoretic result, we write the difference between
\(\cN_T(\lambda)\) and \(\cN(\lambda)\) as \[
|\cN_T(\lambda)  - \cN(\lambda)| \leq \frac{d(\lambda)^2}{1-c(\lambda)} + \lambda e(\lambda),
\] where we define
\(\tilde B := \Sigma_{\lambda}^{-1/2} (\Sigma - \Sigma_{T}) \Sigma_{\lambda}^{-1/2}\),
and \(c(\lambda) = \varsigma_{\max}(\tilde B)\),
\(d(\lambda) = \norm{\tilde B}_{HS}\), and
\(e(\lambda) = \Tr(\Sigma_\lambda^{-1/2} \tilde B \Sigma_{\lambda}^{-1/2})\).
We will now bound each piece: \(c(\lambda)\), \(\lambda e(\lambda)\),
and \(d(\lambda)\).

\emph{Bounding \(c(\lambda)\).} Although Lemma~\ref{lem-c_lambda} gives
a two-sided norm bound, here we need only to bound the largest
eigenvalue. Applying the one-sided Bernstein argument from its proof, we
find with probability at least \(1 - \delta/3\) that \[
c(\lambda) \leq \frac{2\log\left( \frac{12\kappa^2}{\lambda \delta} \right)}{3T} + \sqrt{\frac{2\cF_\infty(\lambda) \log\left( \frac{12\kappa^2}{\lambda \delta} \right)}{T}}.
\]

\emph{Bounding \(\lambda e(\lambda)\).} To bound \(\lambda e(\lambda)\),
we first re-write \(e(\lambda)\) as \[
\begin{aligned} 
e(\lambda) &= \Tr(\Sigma_\lambda^{-1/2} \tilde B \Sigma_\lambda^{-1/2})\\
&= \Tr(\Sigma_\lambda^{-1/2} \Sigma_\lambda^{-1/2} (\Sigma - \Sigma_T) \Sigma_\lambda^{-1/2} \Sigma_\lambda^{-1/2})\\
&= \Tr(\Sigma_\lambda^{-1} (\Sigma - \Sigma_T) \Sigma_\lambda^{-1})\\
&= \Tr(\Sigma_\lambda^{-2} \Sigma) - \Tr(\Sigma_\lambda^{-2} \Sigma_T)\\
&= \Tr(\Sigma_\lambda^{-2} \Sigma) -\frac{1}{T}\sum_{j=1}^T \Tr(\Sigma_\lambda^{-2} \Sigma_{\omega_j})
\end{aligned} 
\] Define
\(\eta_j = \Tr(\lambda \Sigma_\lambda^{-2} \Sigma_{\omega_j})\), which
is i.i.d., as the random trees are i.i.d. The random variable \(\eta_j\)
is bounded, with finite mean and variance. In particular, boundedness
follows from \[
\begin{aligned} 
\lvert \eta_j \rvert &= \Tr(\lambda \Sigma_\lambda^{-2} \Sigma_{\omega_j}) \\
&=^{(a)} \lambda  \Tr(\Sigma_\lambda^{-1} \Sigma_{\lambda}^{-1} \Sigma_{\omega_j})\\
&\leq^{(b)} \lambda \norm{\Sigma_\lambda^{-1}}_{op} \cdot \Tr(\Sigma_\lambda^{-1} \Sigma_{\omega_j}) \\
&\leq^{(c)} \cF_\infty(\lambda).
\end{aligned}  
\] Step (a) follows from \(\Sigma_\lambda^{-1}\) being self-adjoint, (b)
follows from positivity of \(\Sigma_\lambda^{-1}\) and the operator
inequality in the Loewner order
\(\Sigma_\lambda^{-1} \preceq \norm{\Sigma_\lambda^{-1}}_{op}I\), and
(c) follows from \(\lambda \norm{\Sigma_\lambda^{-1}}_{op} \leq 1\).

Furthermore, by linearity of the trace operator, \[
\begin{aligned} 
\E_\omega[\eta_j] &= \E_\omega[\Tr(\lambda \Sigma_\lambda^{-2} \Sigma_{\omega_j}) ] \\
&= \lambda \Tr(\Sigma_\lambda^{-2}\E_\omega[\Sigma_{\omega_j}]) \\
&= \lambda \Tr(\Sigma_\lambda^{-2}\Sigma) \\
&\leq \lambda \norm{\Sigma_\lambda^{-1}}_{op} \Tr(\Sigma_\lambda^{-1} \Sigma)\\
&= \lambda \norm{\Sigma_\lambda^{-1}}_{op} \cN(\lambda) \\
&\leq \cN(\lambda).
\end{aligned} 
\] This implies \[
|\eta_j - \E[\eta_j]| \leq \cF_\infty(\lambda) + \cN(\lambda) \leq 2\cF_\infty(\lambda).
\] The uncentered second moment (and thus the centered moment too) can
then be directly bounded as \[
\E[(\eta_j - \E[\eta_j])^2] \le \E_\omega[\eta_j^2] 
\leq \left( \sup |\eta_j| \right) \E_\omega(\eta_j) \leq \cF_\infty(\lambda) \cN(\lambda).
\]

Then, \(\lambda e(\lambda)\) can be written as deviations of the sample
mean of random variable \(\eta_j\), \[
\begin{aligned} 
\lambda e(\lambda) &= \lambda \Tr(\Sigma_\lambda^{-2} \Sigma) - \lambda \cdot \frac{1}{T} \sum_{j=1}^T \Tr(\Sigma_\lambda^{-2} \Sigma_{\omega_j})\\
&= \E_\omega[\eta_j] - \frac{1}{T} \sum_{j=1}^T \eta_j,
\end{aligned} 
\] and we can apply Bernstein's inequality for sum of random variables
(RR17, Proposition 1) to bound \(\lambda e(\lambda)\). More
specifically, with probability at least \(1-\delta/3\),
\[\lambda |e(\lambda)| \leq \frac{4 \cF_\infty(\lambda) \log(6/\delta)}{3T} + \sqrt{\frac{2 \cF_\infty(\lambda) \cN(\lambda) \log(6/\delta)}{T}}.\]

\emph{Bounding \(d(\lambda)\).} Finally, we bound \(d(\lambda)\). Let
\(A_j := \Sigma_\lambda^{-1/2} \Sigma_{\omega_j} \Sigma_\lambda^{-1/2}\).
Since \(A_j\) is positive semidefinite and
\(\norm{A_j}_{op}\le \Tr(A_j)\le F_\infty(\lambda)\), we have
\(A_j^2\le \norm{A_j}_{op} A_j \leq \cF_\infty(\lambda) A_j\) in the
Loewner order. This means \[
\norm{A_j}_{HS}^2 \le \cF_\infty(\lambda)\Tr(A_j) \le \cF_\infty(\lambda)^2.
\]

Furthermore, this quantity has finite variance: \[
\begin{aligned}
\E_\omega\left[  \norm{A_j}_{HS}^2 \right]&\leq \cF_\infty(\lambda) \E_\omega \left[  \Tr(\Sigma_\lambda^{-1/2} \Sigma_{\omega_j} \Sigma_\lambda^{-1/2}) \right] \\
&= \cF_\infty(\lambda) \cN(\lambda).
\end{aligned} 
\] As above, this implies bounds on the centered terms: \[
\norm{A_j-\E[A_j]}_{HS} \leq 2\cF_\infty(\lambda),
\qand
\E_\omega\left[\norm{A_j-\E[A_j]}_{HS}^2\right] \leq \cF_\infty(\lambda) \cN(\lambda).
\] For \(p\geq2\), the preceding bounds imply \[
\E\left[\norm{A_j-\E[A_j]}_{HS}^p\right]
\leq \left(2\cF_\infty(\lambda)\right)^{p-2}
\cF_\infty(\lambda)\cN(\lambda)
\leq \frac{1}{2}p!\left(2\cF_\infty(\lambda)\cN(\lambda)\right)
\left(2\cF_\infty(\lambda)\right)^{p-2}.
\] Thus the moment condition of RR17 Proposition 2 holds with
\(\sigma^2=2\cF_\infty(\lambda)\cN(\lambda)\) and
\(M=2\cF_\infty(\lambda)\).

Then, applying Bernstein's inequality for sum of random vectors (i.e.,
RR17's Proposition 2), with probability at least \(1-\delta/3\), \[
d(\lambda) \leq \frac{4 \cF_\infty(\lambda) \log(6/\delta)}{T} 
+ \sqrt{\frac{4 \cF_\infty(\lambda) \cN(\lambda) \log(6/\delta)}{T}}.
\]

Now let \(b := \log\left(\frac{12\kappa^2}{\lambda\delta}\right)\) and
\(q := \frac{4\cF_\infty(\lambda)b}{3T}\). Since
\(\lambda\leq\norm{\Sigma}_{op}\leq\kappa^2\), we have
\(\log(6/\delta)\leq b\). Moreover, the assumed lower bound on \(T\)
gives \(q \leq \frac{2}{27}\) and \(c(\lambda) \leq \frac{1}{3}\).
Finally, \[
\cN(\lambda) \geq \norm{\Sigma_\lambda^{-1}\Sigma}_{op}
= \frac{\norm{\Sigma}_{op}}{\norm{\Sigma}_{op}+\lambda}
\geq \frac{1}{2}.
\] On the intersection of the three concentration events, which has
probability at least \(1-\delta\) by a union bound, the preceding bound
on \(\cN(\lambda)\) implies \[
\lambda|e(\lambda)| \leq q+\sqrt{\frac{3q\cN(\lambda)}{2}},
\qquad
d(\lambda) \leq 3q+\sqrt{3q\cN(\lambda)}.
\] Since \(e(\lambda)\leq|e(\lambda)|\), substituting these inequalities
yields \[
\frac{|\cN_T(\lambda)-\cN(\lambda)|}{\cN(\lambda)}
\leq
\frac{q}{\cN(\lambda)}
+ \sqrt{\frac{3q}{2\cN(\lambda)}}
+ \frac{\left(3q/\sqrt{\cN(\lambda)}+\sqrt{3q}\right)^2}{1-c(\lambda)}
< 1.55,
\] as claimed.
\end{proof}

\subsubsection{\texorpdfstring{Bounding
\(\cS(\lambda, T, n)\)}{Bounding \textbackslash cS(\textbackslash lambda, T, n)}}\label{bounding-cslambda-t-n}

Lemma 7 in RR17, which relies on their Lemma 6, establishes the
following result once recast in our notation, and once an algebra error
in the definition of their \(\bar\sigma\) is corrected.

\begin{lemma}[Corrected RR17 Lemma
7]\protect\hypertarget{lem-s-bound}{}\label{lem-s-bound}

Let \(\delta\in(0,1/3]\) and \(n\in\N\); then if
\(T\ge (4 + 18\cF_\infty(\lambda))\log(\frac{12\kappa^2}{\lambda\delta})\),
with (\(\omega\)-) probability at least \(1-3\delta\), \[
\cS(\lambda,T,n) \le 4\left(
    \frac{(B+2R\kappa)\kappa}{n\sqrt{\lambda}} 
    + \sqrt{\frac{(\sigma + R\kappa)^2\cN(\lambda)}{n}}
\right)\log\frac{2}{\delta}.
\]

\end{lemma}

\begin{proof}
The proof proceeds identically to the proof of Lemma 7, with the
operators implicitly defined conditional on \(\omega\), and with the
updated Lemma~\ref{lem-eff-dim-bnd}. A few small errors in the original
proof (incorrectly substituting \(\sigma+2\sqrt{R}\kappa\) and applying
the \(\cN_T(\lambda)\) bound incorrectly) do not affect the logical flow
of the argument.
\end{proof}

\subsubsection{\texorpdfstring{Bounding
\(\cC(\lambda, T)\)}{Bounding \textbackslash cC(\textbackslash lambda, T)}}\label{bounding-cclambda-t}

We will now bound \(\cC(\lambda, T)\). Usefully,
\(\cC(\lambda, T)=R\sqrt{\lambda}\norm{\Sigma_\lambda^{-1/2}(\Sigma - \Sigma_T) \Sigma_\lambda^{-1/2}}_{op}\),
and we can directly apply Lemma~\ref{lem-c_lambda} to recover an
identical probability bound as in RR17, Lemma 8 for the special case of
\(r=1/2\), as follows.

\begin{lemma}[]\protect\hypertarget{lem-c-bound}{}\label{lem-c-bound}

Under the assumptions of Lemma~\ref{lem-c_lambda}, with probability at
least \(1-\delta\), \[
\cC(\lambda, T) \leq R\sqrt{\lambda}\left(
\frac{2 (\cF_{\infty}(\lambda) + 1)\log\left( \frac{8\kappa^2}{\lambda \delta} \right)}{3T} 
+ \sqrt{\frac{2\cF_\infty(\lambda) \log\left( \frac{8\kappa^2}{\lambda \delta} \right)}{T}}\right).
\]

\end{lemma}

When
\(T \geq (4+ 18 \cF_\infty(\lambda))\log\left(\frac{8\kappa^2}{\lambda\delta}\right)\),
we can simplify the upper bound to \[
\cC(\lambda, T) \leq R\sqrt{\lambda}\sqrt{\frac{4\cF_\infty(\lambda) \log\left( \frac{8\kappa^2}{\lambda \delta} \right)}{T}} = 2R \sqrt{\frac{\lambda \cF_\infty(\lambda) \log\left( \frac{8\kappa^2}{\lambda \delta} \right)}{T}}.
\]

\subsubsection{\texorpdfstring{Stability
\(\beta\)}{Stability \textbackslash beta}}\label{stability-beta}

\begin{lemma}[]\protect\hypertarget{lem-beta}{}\label{lem-beta}

Let \(\delta \in (0,1/3]\). Then, with probability at least
\(1-3\delta\), \(\beta < 2\) when
\(0 < \lambda \leq \frac{3}{4} \norm{\Sigma}_{op}\), and \[
n \geq 18 \left( 2 + \frac{\kappa^2}{\lambda} \right)\log \left( \frac{4 \kappa^2}{\lambda \delta}\right),
\]

\[
T \geq 18(2 + \cF_\infty(\lambda)) \log \left( \frac{4 \kappa^2}{\lambda \delta}\right) \vee 32 \left(\frac{\kappa^2}{\norm{\Sigma}_{op}} + \kappa^2 \right) \log (2/\delta)
\]

\end{lemma}

\begin{proof}
The proof proceeds identically to the proof of Lemma 10 in RR17,
bounding \(\beta\) in probability by first bounding \(\beta_1\) and
\(\beta_2\). In bounding \(\beta_1\), the operators are implicitly
defined conditional on \(\omega\), and the bound goes through. For
\(\beta_2\), write the deviation as an average of \(X_j=A-B_j\), as in
the proof of Lemma~\ref{lem-c_lambda}. The one-sided argument used to
bound \(c(\lambda)\) in Lemma~\ref{lem-eff-dim-bnd} gives
\(\varsigma_{\max}(X_j)\leq1\) and
\(\E[X_j^2]\preceq\cF_\infty(\lambda)A\). Applying the one-sided
operator Bernstein inequality therefore gives the required bound on
\(\beta_2\). The two bounds are combined identically, yielding the
stated result.
\end{proof}

\subsubsection{\texorpdfstring{Proof of
Theorem~\ref{thm-rr17-gen}}{Proof of Theorem~}}\label{proof-of-thm-rr17-gen}

We can now derive a generalized generalization bound.

\begin{proof}
From Theorem~\ref{thm-rr17-decomp}, we have the following decomposition:
\[
\norm{\hat g_{\lambda, T, n} - g_0}_X \leq \beta \left(\cS(\lambda, T, n) + \cC(\lambda, T) + R\sqrt{\lambda} \right).
\] We can now apply the lemmas derived to bound each piece in
probability.

\emph{Bounding \(\beta\).} Let \(\tau := \delta/7\). When the following
conditions hold,

\begin{itemize}
\tightlist
\item
  (c1) \(0 < \lambda \leq \frac{3}{4} \norm{\Sigma}_{op}\)
\item
  (c2)
  \(n \geq 18(2 + \kappa^2/\lambda) \log \left(\frac{4 \kappa^2}{\lambda \tau} \right)\)
\item
  (c3)
  \(T \geq 18(2 + \cF_\infty(\lambda)) \log \left( \frac{4 \kappa^2}{\lambda \tau}\right) \vee 32 \left(\frac{\kappa^2}{\norm{\Sigma}_{op}} + \kappa^2 \right) \log (2/\tau)\)
\end{itemize}

by Lemma~\ref{lem-beta}, \(\Pr(\beta < 2) \geq 1-3\tau\).

\emph{Bounding \(\cS(\lambda, T, n)\).} Under the conditions:

\begin{itemize}
\tightlist
\item
  (c4) \(0 < \lambda < \norm{\Sigma}_{op}\)
\item
  (c5)
  \(T \geq (4 + 18 \cF_\infty(\lambda)) \log \left( \frac{12\kappa^2}{\lambda \tau}\right)\),
\end{itemize}

we can directly apply Lemma~\ref{lem-s-bound}, yielding \[
\Pr\left(\cS(\lambda, T, n) \leq  4\left(
    \frac{(B+2R\kappa)\kappa}{n\sqrt{\lambda}} 
    + \sqrt{\frac{(\sigma + R\kappa)^2\cN(\lambda)}{n}}
\right)\log\frac{2}{\tau} \right) \geq 1 - 3\tau .
\]

\emph{Bounding \(\cC(\lambda, T)\).} Finally, under condition (c1) and
the following

\begin{itemize}
\tightlist
\item
  (c6)
  \(T \geq (4 + 18 \cF_\infty(\lambda))\log \left(\frac{8\kappa^2}{\lambda \tau} \right)\),
\end{itemize}

we can apply Lemma~\ref{lem-c-bound}, which gives \[
\Pr\left( \cC(\lambda, T) \leq 2R \sqrt{\frac{\lambda \cF_\infty(\lambda) \log \left(\frac{8\kappa^2}{\lambda \tau} \right)}{T}}\right) \geq 1-\tau.
\] Notice that the stated assumption on \(T\) implies (c2), (c3), (c5),
and (c6), by the definition of \(q_0\) (and that \(q_0\ge 2\)).

Combining these with a union bound, with probability at least
\(1-\delta\), \[
\begin{aligned} 
&\norm{\hat g_{\lambda, T, n} - g_0}_X  \\
\leq&8 \left(
    \frac{(B+2R\kappa)\kappa}{n\sqrt{\lambda}} 
    + \sqrt{\frac{(\sigma + R\kappa)^2\cN(\lambda)}{n}}
\right)\log\frac{14}{\delta} + 2 R \sqrt{\lambda} + 
4R \sqrt{\frac{\lambda \cF_\infty(\lambda) \log \left(\frac{56\kappa^2}{\lambda \delta} \right)}{T}}
\end{aligned} 
\] The assumed bound on \(T\) implies that the final term is at most
\(R\sqrt{\lambda}\). Indeed, if \[
T \ge 18(q_0+ \cF_\infty(\lambda)) \log \left( \frac{84 \kappa^2}{\lambda \delta}\right)
> 16 \cF_\infty(\lambda) \log \left( \frac{56\kappa^2}{\lambda \delta}\right),
\] then \[
4R \sqrt{\frac{\lambda \cF_\infty(\lambda) \log \left( \frac{56\kappa^2}{\lambda \delta}\right)}{T}} \leq R \sqrt{\lambda}.
\] Substituting this into the preceding display gives the stated result.
\end{proof}

\subsection{\texorpdfstring{Proof of
Theorem~\ref{thm-finite-tree-conv}}{Proof of Theorem~}}\label{sec-prf-finite-conv}

\subsubsection{Preliminary results}\label{preliminary-results}

We begin with two key BART-specific lemmas: one that establishes rates
on \(\cN(\lambda)\), and one that shows
\(\cF_\infty(\lambda)\asymp \cN(\lambda)\) when tree depths are bounded
above.

\begin{lemma}[]\protect\hypertarget{lem-eff-dim-bart}{}\label{lem-eff-dim-bart}

If the eigenvalues of an operator \(\Sigma\) satisfy
\(\varsigma_j \asymp j^{-2} \log(j)^{2(p-1)}\), then as \(\lambda\to 0\)
its effective dimension satisfies \[
\cN(\lambda) \asymp \lambda^{-1/2} \log(1/\lambda)^{p-1}.
\]

\end{lemma}

\begin{proof}
Let \(j_\lambda := \max\{j: \varsigma_j \geq \lambda\}\) represent the
index of the last eigenvalue that is larger than or equal to
\(\lambda\); since the eigenvalues are sorted in descending order, as
\(\lambda\) decreases, \(j_\lambda\) increases. We will first establish
that \(\lambda \asymp j_\lambda^{-2} \log(j_\lambda)^{2(p-1)}\). We will
also show that \(\cN(\lambda) \asymp j_\lambda\). Solving the first
result for \(j_\lambda\) as a function of of \(\lambda\) will then
establish the claimed relationship.

First, we substitute \(j_\lambda\) into the assumed eigenvalue decay
rate. This yields \[
c_1 j_\lambda^{-2} \log(j_\lambda)^{2(p-1)} \leq \varsigma_{j_\lambda} \leq C_1 j_\lambda^{-2} \log(j_\lambda)^{2(p-1)}
\] for some constants \(c_1, C_1 > 0\), by the definition of asymptotic
equivalence. Similarly, for \(\varsigma_{j_\lambda + 1}\), there exists
\(c_2, C_2 > 0\) such that \[
c_2 (j_\lambda+1)^{-2} \log(j_\lambda+1)^{2(p-1)} \leq \varsigma_{j_\lambda+1} \leq C_2 (j_\lambda+1)^{-2} \log(j_\lambda+1)^{2(p-1)}.
\] Because
\(\varsigma_{j_\lambda+1} < \lambda \leq \varsigma_{j_\lambda}\), we
have \[
c_2 (j_{\lambda}+1)^{-2} \log(j_\lambda + 1)^{2(p-1)} < \lambda 
\le C_1 j_{\lambda}^{-2} \log(j_\lambda)^{2(p-1)}.
\] As such, \(\lambda \asymp j_\lambda^{-2} \log(j_\lambda)^{2(p-1)}\).

Second, we will show \(\cN(\lambda) \asymp j_\lambda\). To start,
decompse \(\cN(\lambda)\) into two sums: \[
\cN(\lambda) := \Tr((\Sigma + \lambda I)^{-1} \Sigma) 
= \sum_{j=1}^\infty\frac{\varsigma_j}{\varsigma_j + \lambda} 
= \sum_{j \leq j_\lambda} \frac{\varsigma_j}{\varsigma_j + \lambda} 
  + \sum_{j > j_\lambda}  \frac{\varsigma_j}{\varsigma_j + \lambda}.
\] We consider the two sums separately.

\emph{Sum 1: \(j \leq j_\lambda\).} For each \(j\), we must have by
definition that \(\varsigma_j \geq \lambda\). As such,
\(\frac{\varsigma_j}{\varsigma_j + \lambda} \geq 1/2\), and so
\(\sum_{j \leq j_\lambda} \frac{\varsigma_j}{\varsigma_j + \lambda} \geq \frac{1}{2} j_\lambda\).
In the other direction, since
\(\frac{\varsigma_j}{\varsigma_j + \lambda} \leq 1\), we must have
\(\sum_{j \leq j_\lambda} \frac{\varsigma_j}{\varsigma_j + \lambda} \leq j_\lambda\).
We have therefore established that
\begin{equation}\protect\phantomsection\label{eq-sum1}{
\frac{1}{2} j_\lambda \leq \sum_{j \leq j_\lambda} \frac{\varsigma_j}{\varsigma_j + \lambda} \leq j_\lambda \implies \sum_{j \leq j_\lambda} \frac{\varsigma_j}{\varsigma_j + \lambda} \asymp j_\lambda.
}\end{equation}

\emph{Sum 2: \(j > j_\lambda\).} Since
\(\frac{\varsigma_j}{\varsigma_j + \lambda} \leq \frac{\varsigma_j}{\lambda}\),
we can simplify this term to
\begin{equation}\protect\phantomsection\label{eq-sum2-simp}{
\sum_{j > j_\lambda} \frac{\varsigma_j}{\varsigma_j + \lambda} \leq \frac{1}{\lambda} \sum_{j > j_\lambda} \varsigma_j
\lesssim \frac{1}{\lambda} \sum_{j > j_\lambda}  j^{-2} \log(j)^{2(p-1)},
}\end{equation} where the asymptotic inequality follows from the assumed
eigenvalue decay. Because \(j^{-2} \log(j)^{2(p-1)}\) is positive and
monotonically decreasing for sufficiently large \(j\), we can apply the
integral test to the sum and conclude that \[
\begin{aligned}
\sum_{j > j_\lambda} j^{-2} \log(j)^{2(p-1)} 
&\le j_\lambda^{-2} \log(j_\lambda)^{2(p-1)} + \int_{j_\lambda}^\infty  j^{-2} \log(j)^{2(p-1)}\,dj \\
&= j_\lambda^{-2} \log(j_\lambda)^{2(p-1)} + \int_{\log(j_\lambda)}^\infty u^{2(p-1)} e^{-u}\,du \\
&\asymp j_\lambda^{-2} \log(j_\lambda)^{2(p-1)} + \log(j_\lambda)^{2(p-1)} \exp(-\log(j_\lambda)), \\
&= j_\lambda^{-2} \log(j_\lambda)^{2(p-1)} + j_\lambda^{-1} \log(j_\lambda)^{2(p-1)} \\
&\asymp j_\lambda^{-1} \log(j_\lambda)^{2(p-1)},
\end{aligned}
\] where the second line follows from the substitution \(u = \log(j)\)
and the third line arises from noting that the integral is an incomplete
upper Gamma function, which has the stated asymptotic behavior.

Substituting this result into Eq.~\ref{eq-sum2-simp}, and since we have
established \(\lambda \asymp j_\lambda^{-2} \log(j_\lambda)^{2(p-1)}\),
\begin{equation}\protect\phantomsection\label{eq-sum2}{
\sum_{j > j_\lambda} \frac{\varsigma_j}{\varsigma_j + \lambda} 
\leq \frac{1}{\lambda} \sum_{j > j_\lambda} \varsigma_j 
\lesssim \frac{j_\lambda^{-1} \log(j_\lambda)^{2(p-1)}}{j_{\lambda}^{-2} \log(j_\lambda)^{2(p-1)}} 
= j_\lambda.
}\end{equation}

Combining both Eq.~\ref{eq-sum1} and Eq.~\ref{eq-sum2}, we have shown
that \(\cN(\lambda) \asymp j_\lambda\).

Finally, we will solve for \(j_\lambda\): \[
\begin{aligned}
\lambda \asymp j_\lambda^{-2} \log(j_\lambda)^{2(p-1)} 
&\iff j_\lambda \asymp \lambda^{-1/2}\log(j_\lambda)^{p-1} \\
&\iff j_\lambda \asymp \lambda^{-1/2} \log(1/\lambda)^{p-1} \\
&\implies \cN(\lambda) \asymp \lambda^{-1/2} \log(1/\lambda)^{p-1},
\end{aligned}
\] where the second line follows from the fact that
\(\log (j_\lambda) \sim \frac{1}{2}\log(1/\lambda)\), which is itself
established by a comparison of leading terms in the first expression;
indeed, taking logarithms, \[
\log(1/\lambda) = 2\log(j_\lambda) - 2(p-1)\log\log(j_\lambda) + O(1),
\] and \(\log\log(j_\lambda)=o(\log(j_\lambda))\).
\end{proof}

Lemma~\ref{lem-eff-dim-bart} is theoretically significant because it
shows that the effective dimension of \(\Sigma_\b\) depends on the
covariate dimension only logarithmically. Can this same rate be extended
to \(\cF_\infty(\lambda)\)? By definition
\(\cN(\lambda)\le \cF_\infty(\lambda)\), but the reverse inequality
requires stricter conditions on the BART prior.

\begin{lemma}[]\protect\hypertarget{lem-f-inf-bnd}{}\label{lem-f-inf-bnd}

For a regular BART prior with bound depth \(D\leq d^+<\infty\), \[
\cF_\infty(\lambda)\lesssim \cN(\lambda),
\] with the constant uniform in \(\lambda\).

\end{lemma}

\begin{proof}
By Lemma~\ref{lem-coord-meas}, it suffices to work with uniform split
and covariate distributions on \([0,1]^p\). We can also absorb the fixed
rescaling \(\Sigma_\b=\sigma_\mu^2\Sigma\) and write \(\Sigma\) for the
population operator on \(L^2([0,1]^p)\).

Let \(L(s,t)=\exp(-\left|s-t\right|)\) be the Laplacian kernel and
\(K_\otimes(\bu,\bu')=\prod_{v=1}^pL(u_v,u_v')\), and let \(\Sigma_L\)
and \(\Sigma_L^{\otimes p}\) denote the corresponding integral
operators. Regularity gives \(\Pr(D\geq p)>0\), so the bounded-depth
condition necessarily has \(d^+\geq p\). The lower kernel bound in
Proposition~\ref{prp-rkhs} therefore implies
\(\Sigma\succeq c\Sigma_L^{\otimes p}\) for some \(c>0\). We next record
a property of \(\Sigma_L\). If \((\varsigma_j,e_j)\) is one of its
normalized eigenpairs (that is, \(\Sigma_Le_j=\varsigma_je_j\) and
\(\norm{e_j}_{L^2([0,1])}=1\)), then
\begin{equation}\protect\phantomsection\label{eq-interval-eigen-bnd}{
\left|\int_I e_j(s)\,ds\right|^2\lesssim\varsigma_j
}\end{equation} for every interval \(I\subseteq[0,1]\) and every \(j\),
with a common implicit constant. Indeed, the eigenpairs of the
exponential covariance kernel on a symmetric interval \([-b,b]\) are
known explicitly \citep[see][Example 4.1]{xiu2010numerical}: \[
\varsigma_j=\frac{2}{1+\omega_j^2},
\qquad
e_j(s)=c_j\cos\{\omega_j(s-1/2)\}
\quad\text{or}\quad
e_j(s)=c_j\sin\{\omega_j(s-1/2)\},
\] where the explicit normalizing constants \(c_j\) are bounded above.
Moreover, \(\Tr(\Sigma_L)=\int_0^1L(s,s)\,ds=1\), so
\(\varsigma_j\leq1\) and the eigenvalue formula gives \(\omega_j\geq1\).
Consequently, integration over any interval \(I\subseteq[0,1]\) gives \[
\left|\int_Ie_j(s)\,ds\right|
\leq\frac{2\sup_j|c_j|}{\omega_j};
\] squaring establishes \[
\left|\int_Ie_j(s)\,ds\right|^2
\lesssim \frac{1}{\omega_j^2}\leq\frac{2}{1+\omega_j^2}=\varsigma_j,
\] as claimed.

Now, the eigenfunctions of \(\Sigma_L^{\otimes p}\) are the tensor
products \(e_{\vb j}(\bu)=\prod_{v=1}^p e_{j_v}(u_v)\), with eigenvalues
\(\varsigma_{\vb j}=\prod_{v=1}^p\varsigma_{j_v}\). Thus
Eq.~\ref{eq-interval-eigen-bnd} implies, for every rectangular block
\(R=\prod_{v=1}^p I_v\), with \(\psi_R:=\ind\{\bu\in R\}\), \[
\left|\langle \psi_R,e_{\vb j}\rangle\right|^2
=\prod_{v=1}^p\left|\int_{I_v}e_{j_v}(s)\,ds\right|^2
\lesssim \varsigma_{\vb j}.
\] Consequently, \[
\begin{aligned}
\langle \psi_R,(c\Sigma_L^{\otimes p}+\lambda I)^{-1}\psi_R\rangle
&\lesssim \sum_{\vb j}\frac{\varsigma_{\vb j}}{c\varsigma_{\vb j}+\lambda}\\
&=\frac{1}{c}\Tr\{(c\Sigma_L^{\otimes p}+\lambda I)^{-1}c\Sigma_L^{\otimes p}\}.
\end{aligned}
\]

Since \(\Sigma\succeq c\Sigma_L^{\otimes p}\), the preceding display
gives
\begin{equation}\protect\phantomsection\label{eq-rectangle-leverage}{
\begin{aligned}
\langle\psi_R,(\Sigma+\lambda I)^{-1}\psi_R\rangle
&\le \langle\psi_R,(c\Sigma_L^{\otimes p}+\lambda I)^{-1}\psi_R\rangle\\
&\lesssim \Tr\{(c\Sigma_L^{\otimes p}+\lambda I)^{-1}c\Sigma_L^{\otimes p}\} \\
&\lesssim \Tr\{(\Sigma+\lambda I)^{-1}\Sigma\}=\cN(\lambda),
\end{aligned}
}\end{equation} where the last line follows since the map
\(t\mapsto t/(t+\lambda)\) preserves the Loewner order.

Finally, for a tree given by a draw \(\omega\), every leaf indicator is
the indicator of a rectangular block, and so applying the definition of
\(\Sigma_\omega\), \[
\Tr(\Sigma_\lambda^{-1}\Sigma_\omega)
=\sum_{\vb l\in\L(\omega)}
\langle\psi_{\vb l},\Sigma_\lambda^{-1}\psi_{\vb l}\rangle.
\] Because \(D\leq d^+\), the tree has at most \(2^{d^+}\) leaves.
Applying Eq.~\ref{eq-rectangle-leverage} to each leaf and taking the
supremum over \(\omega\) proves \[
\cF_\infty(\lambda)\lesssim 2^{d^+}\cN(\lambda)\lesssim\cN(\lambda). \qedhere
\]
\end{proof}

We can now combine the derived effective dimension in
Lemma~\ref{lem-eff-dim-bart} with the generalized
Theorem~\ref{thm-rr17-gen} above to specialize the latter to our case.

\begin{lemma}[]\protect\hypertarget{lem-prob-bound}{}\label{lem-prob-bound}

Under the stated assumptions on the data, suppose the BART prior is such
that the effective dimension of its corresponding kernel operator
satisfies \(\cN(\lambda)\asymp \lambda^{-1/2}\log(1/\lambda)^{p-1}\) as
\(\lambda\to 0\). Suppose further that
\(\lambda_n\asymp n^{-2/3}\log(n)^{2(p-1)/3}\). There exist constants
\(c_0,c_1,n_0>0\), not depending on \(n\) or \(\delta\), such that, for
any \(\delta\in(0,1/2]\), if
\begin{equation}\protect\phantomsection\label{eq-n-bound}{
n\geq n_0\vee
c_0\log\left(\frac{1}{\delta}\right)^3
}\end{equation} and
\begin{equation}\protect\phantomsection\label{eq-M-bound}{
T_n\geq
c_1\cF_\infty(\lambda_n)\log\left(\frac{n}{\delta}\right),
}\end{equation} then, with probability at least \(1-\delta\), \[
\norm{\hat g_{T,n}-g_0}_X
\lesssim
n^{-1/3}\log(n)^{(p-1)/3}
\log\left(\frac{1}{\delta}\right).
\]

\end{lemma}

\begin{proof}
We will start by showing that the assumed bound on \(T_n\) implies
Eq.~\ref{eq-M-bound-rudi} in Theorem~\ref{thm-rr17-gen}, and then show
that the assumed bound on \(n\) implies Eq.~\ref{eq-n-bound-rudi}. Those
equations, plus the assumption on \(\lambda_n\), are sufficient to apply
Theorem~\ref{thm-rr17-gen} above. By Lemma~\ref{lem-coord-meas} above,
the fixed factor \(\sigma_\mu^2\) in \(\Sigma_\b\) may be absorbed into
\(c_0\), \(c_1\), and \(n_0\), and we therefore proceed without
considering any constant scaling of \(\Sigma\) or \(\lambda_n\).

We first verify the tree-count condition. Since \[
\cF_\infty(\lambda_n) \geq\cN(\lambda_n) \longrightarrow\infty
\] as \(n\to\infty\), where the convergence follows from
\(\lambda_n\to0\) and the assumed effective-dimension rate, the fixed
term \(q_0\) may be absorbed into \(\cF_\infty(\lambda_n)\). Then
substituting the assumed order of \(\lambda_n\) into
Eq.~\ref{eq-M-bound-rudi}, \[
\begin{aligned}
18(q_0+\cF_\infty(\lambda_n)) \log\left(\frac{84}{\lambda_n\delta}\right)
&\lesssim \cF_\infty(\lambda_n) \left(\log(84)
    - \log(n^{-2/3}\log(n)^{2(p-1)/3}) - \log(\delta) \right)\\
&= \cF_\infty(\lambda_n) \left(\frac{2}{3}\log(n) - \frac{2(p-1)}{3}\log\log(n)
    + \log(1/\delta) + \log(84) \right)\\
&\lesssim \cF_\infty(\lambda_n) (\log(n)+\log(1/\delta)) \\
&= \cF_\infty(\lambda_n) \log\left(\frac{n}{\delta}\right).
\end{aligned}
\] The right-hand side is bounded by Eq.~\ref{eq-M-bound} by the choice
of \(c_1\). Thus Eq.~\ref{eq-M-bound} implies Eq.~\ref{eq-M-bound-rudi}
in Theorem~\ref{thm-rr17-gen}.

Next consider the sample-size condition. Substituting \(\lambda_n\) into
the right-hand side of Eq.~\ref{eq-n-bound-rudi}, simplifying as above,
and expanding its logarithm gives \[
\begin{aligned}
18(2+\lambda_n^{-1}) \log\left(\frac{28}{\lambda_n\delta}\right)
&\lesssim n^{2/3}\log(n)^{-2(p-1)/3}
    \log\left(\frac{28 n^{2/3}}{\log(n)^{2(p-1)/3}\delta}\right)\\
&= n^{2/3}\log(n)^{-2(p-1)/3} \left(\log(28)+\frac23\log(n)
    -\frac{2(p-1)}{3}\log\log(n) +\log(1/\delta) \right) \\
&\lesssim n^{2/3}\log(n)^{-2(p-1)/3} \left(\log(n)+\log(1/\delta)\right).
\end{aligned}
\] Let \(C\) denote the fixed constant implicit in the last inequality.
Choose \(n_0\) large enough that, for every \(n\geq n_0>1\),
\(0<\lambda_n\leq\frac34\norm{\Sigma}_{op}\),
\(\cF_\infty(\lambda_n)\geq1\), and \[
C\log(n)\leq \frac12n^{1/3}\log(n)^{2(p-1)/3}.
\] Such a choice is possible because \(\lambda_n\to0\) and the ratio of
the right-hand side to \(\log(n)\) diverges. Choose \(c_0\geq(2C)^3\).
Now Eq.~\ref{eq-n-bound} gives
\(C\log(1/\delta)\leq\frac12n^{1/3}\leq \frac12n^{1/3}\log(n)^{2(p-1)/3}\).
It follows that \[
\begin{aligned}
18(2+\lambda_n^{-1}) \log\left(\frac{28}{\lambda_n\delta}\right)
&\leq C n^{2/3}\log(n)^{-2(p-1)/3} \left(\log(n)+\log(1/\delta)\right)\\
&\leq n^{2/3}\log(n)^{-2(p-1)/3} \left(n^{1/3}\log(n)^{2(p-1)/3}\right)\\
&= n.
\end{aligned}
\] Thus Eq.~\ref{eq-n-bound} implies Eq.~\ref{eq-n-bound-rudi} in
Theorem~\ref{thm-rr17-gen}. It remains to calculate the bound from
Theorem~\ref{thm-rr17-gen}. The assumed effective-dimension rate and the
choice of penalty give \[
\begin{aligned}
\cN(\lambda_n)
&\asymp \lambda_n^{-1/2}\log(1/\lambda_n)^{p-1}\\
&\asymp n^{1/3}\log(n)^{-(p-1)/3}
    \left( \frac23\log(n) -\frac{2(p-1)}{3}\log\log(n) +O(1) \right)^{p-1}\\
&\asymp n^{1/3}\log(n)^{-(p-1)/3}\log(n)^{p-1}\\
&= n^{1/3}\log(n)^{2(p-1)/3}.
\end{aligned}
\] Here the \(O(1)\) term accounts for the fixed constants in the
assumed order of \(\lambda_n\).

Finally, substituting \(\lambda_n\) and the preceding order of
\(\cN(\lambda_n)\) into Eq.~\ref{eq-rudi-bound}, and using \(\kappa=1\),
gives \[
\begin{aligned}
\norm{\hat g_{T,n}-g_0}_X
&\leq 8\left( \frac{B+2R}{n\sqrt{\lambda_n}}
    +(\sigma+R)\sqrt{\frac{\cN(\lambda_n)}{n}} \right)
    \log\left(\frac{14}{\delta}\right) +3R\sqrt{\lambda_n}\\
&\lesssim \left( \frac{1}{n\left(n^{-2/3}\log(n)^{2(p-1)/3}\right)^{1/2}}
    + \sqrt{\frac{n^{1/3}\log(n)^{2(p-1)/3}}{n}} \right)
    \log\left(\frac{14}{\delta}\right)\\
  &\qquad + \left(n^{-2/3}\log(n)^{2(p-1)/3} \right)^{1/2}\\
&\lesssim \left( n^{-2/3}\log(n)^{-(p-1)/3}
    +n^{-1/3}\log(n)^{(p-1)/3} \right)
    \log\left(\frac{14}{\delta}\right)\\
  &\qquad + n^{-1/3}\log(n)^{(p-1)/3}\\
&\lesssim n^{-1/3}\log(n)^{(p-1)/3} \log\left(\frac{1}{\delta}\right),
\end{aligned}
\] as claimed, where the last line uses \(\delta\leq1/2\).
\end{proof}

\subsubsection{Main result}\label{main-result}

Finally, we can integrate over \(\delta\) in Lemma~\ref{lem-prob-bound}
to prove Theorem~\ref{thm-finite-tree-conv}.

\begin{proof}[Proof of Theorem~\ref{thm-finite-tree-conv}]
Let \(a_n:=n^{-2/3}\log(n)^{2(p-1)/3}\) and
\(Z_n:=\norm{\hat g_{T,n}-g_0}_X^2\), so \(\lambda_n\asymp a_n\), and we
aim to prove \(\E[Z_n]\lesssim a_n\).

We first identify a range of confidence levels on which
Lemma~\ref{lem-prob-bound} applies uniformly. Fix \(K>8/3\) and let
\(\delta_*:=n^{-K}\). For every \(\delta\in[\delta_*,1/2]\), \[
\log\left(\frac{n}{\delta}\right)
\leq\log(n^{K+1})
=(K+1)\log(n).
\] Consequently, the tree-count condition in Lemma~\ref{lem-prob-bound}
holds throughout this range whenever
\begin{equation}\protect\phantomsection\label{eq-M-bound-expectation}{
T_n\geq c_1(K+1)\cF_\infty(\lambda_n)\log(n).
}\end{equation} This is the only condition on \(T_n\) used in the
remainder of the proof. The general bound
\(\cF_\infty(\lambda_n)\leq\lambda_n^{-1}\) gives \[
\begin{aligned}
\cF_\infty(\lambda_n)\log(n)
&\leq\lambda_n^{-1}\log(n)\\
&\lesssim
n^{2/3}\log(n)^{-2(p-1)/3}\log(n)\\
&=
n^{2/3}\log(n)^{1-2(p-1)/3}.
\end{aligned}
\] Therefore the assumed lower bound on \(T_n\), with \(c_T\)
sufficiently large, implies Eq.~\ref{eq-M-bound-expectation} in
Lemma~\ref{lem-prob-bound}. If instead \(D\leq d^+<\infty\),
Lemma~\ref{lem-f-inf-bnd} and Lemma~\ref{lem-eff-dim-bart} give \[
\begin{aligned}
\cF_\infty(\lambda_n)\log(n)
&\lesssim \cN(\lambda_n)\log(n)\\
&\asymp \lambda_n^{-1/2}\log(1/\lambda_n)^{p-1}\log(n)\\
&= (n^{-2/3}\log(n)^{2(p-1)/3})^{-1/2}\left(
    \frac{2}{3}\log(n) - \frac{2(p-1)}{3}\log\log(n) + O(1)\right)^{p-1} \log(n) \\
&\asymp n^{1/3}\log(n)^{-(p-1)/3}\log(n)^{p-1}\log(n)\\
&=n^{1/3}\log(n)^{1+2(p-1)/3},
\end{aligned}
\] so the lower bound on \(T_n\) in the second statement likewise
implies Eq.~\ref{eq-M-bound-expectation}.

The sample-size condition in Lemma~\ref{lem-prob-bound} also holds
uniformly over \(\delta\in[\delta_*,1/2]\) for all sufficiently large
\(n\). Indeed, \[
c_0\log(1/\delta)^3 \le c_0K^3\log(n)^3 =o(n).
\] Hence there is a fixed \(n_1\) such that \(n\geq c_0K^3\log(n)^3\)
for every \(n\geq n_1\). We define \(n_0\) for
Theorem~\ref{thm-finite-tree-conv} to be \(\max\{n_1, n_0'\}\), where
\(n_0'\) is the fixed \(n_0\) in Lemma~\ref{lem-prob-bound}.

By Lemma~\ref{lem-prob-bound}, there is a fixed \(C_2>0\) such that \[
\Pr\left(Z_n>C_2a_n\log(1/\delta)^2 \right) \leq\delta
\] for every \(\delta\in[n^{-K},1/2]\). Now change variables from
\(\delta\) to \(t:=C_2a_n\log(1/\delta)^2\), which has inverse \[
\delta(t) := \exp\left(-\sqrt{\frac{t}{C_2a_n}}\right).
\] Under this change of variables, the endpoints \(\delta=1/2\) and
\(\delta=n^{-K}\) become, respectively, \[
t_0:=C_2a_n\log(2)^2 \qand t_*:=C_2a_nK^2\log(n)^2.
\] Thus, for every \(t\in[t_0,t_*]\), the probability bound becomes
\begin{equation}\protect\phantomsection\label{eq-Z-tail-t}{
\Pr(Z_n>t) \le \exp\left(-\sqrt{\frac{t}{C_2a_n}}\right).
}\end{equation}

We can now integrate the tail probability, first up to \(t_*\) using
only Lemma~\ref{lem-prob-bound}, and then past \(t_*\) additionaly
relying on a separate moment bound based on sub-exponentiality. First,
\[
\begin{aligned}
\E[\min\{Z_n, t_*\}] &=\int_0^{t_*}\Pr(Z_n>t)\,dt \\
&\le t_0 + \int_{t_0}^{t_*} \exp\left(-\sqrt{\frac{t}{C_2a_n}}\right)\,dt \\
&\le t_0 + \int_0^\infty \exp\left(-\sqrt{\frac{t}{C_2a_n}}\right)\,dt.
\end{aligned}
\] In the final integral, make the substitution \(s=\sqrt{t/(C_2a_n)}\),
so that \(t=C_2a_ns^2\) and \(dt=2C_2a_ns\,ds\). Then \[
\begin{aligned}
\int_0^\infty \exp\left(-\sqrt{\frac{t}{C_2a_n}}\right)\,dt
&= 2C_2a_n\int_0^\infty se^{-s}\,ds \\
&= 2C_2a_n.
\end{aligned}
\] Since \(t_0\asymp a_n\), we have
\begin{equation}\protect\phantomsection\label{eq-Z-truncated}{
\E[\min\{Z_n, t_*\}]\lesssim a_n.
}\end{equation}

It remains to control the tail probability above \(t_*\). At the
minimizing \(\hat\beta\), the ridge objective clearly upper bounds the
penalty term alone. At the value \(\beta=0\), the ridge objective
simplifies to \(\E_n[Y^2]\), where \(\E_n\) is the empirical mean, and
must be larger than the objective at \(\hat\beta\). Thus \[
\frac{\lambda_n}{\sigma_\mu^2}\norm{\hat\beta}_2^2 \leq \E_n[Y_i^2].
\] By Cauchy--Schwarz, \[
\begin{aligned}
\norm{\hat g_{T,n}}_X^2
&=\E_X\left[ (\phi_T(X)^\top\hat\beta)^2 \right]\\
&\leq \E_X\left[\norm{\phi_T}_2^2\right]\norm{\hat\beta}_2^2\\
&\leq \frac{\sigma_\mu^2}{\lambda_n} \E_n[Y_i^2],
\end{aligned}
\] where in the last line we have also used that \(\norm{\phi_T}_2=1\),
since for every \(\vb x\), exactly one leaf indicator in each tree is
active. Therefore \[
Z_n \leq 2\norm{\hat g_{T,n}}_X^2+2\norm{g_0}_X^2
\leq \frac{2\sigma_\mu^2}{\lambda_n} \E_n[Y_i^2] +2\norm{g_0}_X^2.
\] Because the regression noise is sub-exponential, \(\E[Y_i^2]\) and
\(\E[Y_i^4]\) are uniformly bounded, and so \(\E[\E_n[Y_i^2]^2]\) is
also uniformly bounded. Squaring the bound on \(Z_n\) and using
\(\lambda_n\to0\) then gives
\begin{equation}\protect\phantomsection\label{eq-Z-moment}{
\E[Z_n^2]\lesssim\lambda_n^{-2}.
}\end{equation}

Now, at the upper endpoint, Eq.~\ref{eq-Z-tail-t} gives \[
\begin{aligned}
\Pr(Z_n>t_*)
&\leq \exp\left(-\sqrt{\frac{C_2a_nK^2\log(n)^2}{C_2a_n}}\right) \\
&= \exp(-K\log(n)) \\
&= n^{-K}.
\end{aligned}
\] By Cauchy--Schwarz and Eq.~\ref{eq-Z-moment}, \[
\begin{aligned}
\E[\max\{Z_n-t_*, 0\}]
&\le \E[Z_n\ind\{Z_n>t_*\}]\\
&\le \E[Z_n^2]^{1/2}\Pr(Z_n>t_*)^{1/2}\\
&\lesssim \lambda_n^{-1}n^{-K/2}.
\end{aligned}
\] Comparing this rate to \(a_n\), we find \[
\begin{aligned}
\frac{\lambda_n^{-1}n^{-K/2}}{a_n}
&\asymp \frac{n^{2/3}\log(n)^{-2(p-1)/3}n^{-K/2}}{
    n^{-2/3}\log(n)^{2(p-1)/3}}\\
&= n^{4/3-K/2}\log(n)^{-4(p-1)/3}\\
&\to 0,
\end{aligned}
\] because \(K>8/3\). Thus \(\E[\max\{Z_n-t_*, 0\}]\lesssim a_n\).
Combining this with Eq.~\ref{eq-Z-truncated} yields \[
\E[Z_n]
=
\E[\min\{Z_n, t_*\}]+\E[(Z_n-t_*)_+]
\lesssim a_n,
\] as claimed.
\end{proof}

\newpage{}

\section{Additional simulation
results}\label{additional-simulation-results}

\subsection{Uncertainty
quantification}\label{uncertainty-quantification-1}

Figure~\ref{fig-uncert-app} plots results from the same experiment as in
Figure~\ref{fig-uncert}, but varying the sample size and residual
standard deviation. Both the coverage and width of credible intervals
are plotted for the regression function, which as in the main text is
drawn from a \(\GP(0, k_\b)\) prior. While proper coverage requires
better tuning of the BART hyperparameters, the coverage and width of the
intervals under the random BART features models closely track those of
the full BART model.

\begin{figure}

\centering{

\pandocbounded{\includegraphics[keepaspectratio]{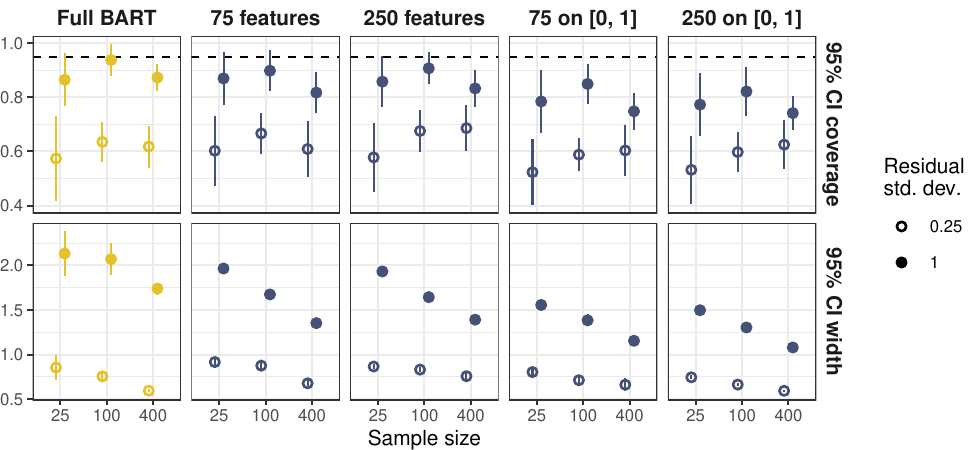}}

}

\caption{\label{fig-uncert-app}\textbf{Coverage and interval width.} For
full BART and the same four variations of random BART features as in the
main text, we plot the coverage and width of 95\% credible intervals of
the regression function, which is drawn from a \(\GP(0, k_\b)\) prior in
reality. Results are plotted for three sample sizes and two different
residual standard deviations.}

\end{figure}%

\end{document}